\documentclass[a4paper]{article}
\usepackage{hyperref}      % hyperlinks
\usepackage{url}           % simple URL typesetting

\usepackage{amsthm,amssymb,amsmath,amsfonts,amscd,mathrsfs,bm}       
\usepackage{xcolor}	
\usepackage{graphicx} %Loading the package
\graphicspath{{Figures/}} %Setting the graphicspath
\usepackage{lineno}
  
\newtheorem{theorem}{Theorem}
\newtheorem{propo}[theorem]{Proposition}

\newtheorem{lem}[theorem]{Lemma}
\theoremstyle{definition}
\newtheorem{Def}{Definition}
\newtheorem{conj}{Conjecture}
\newtheorem{example}{Example}
\newtheorem{rem}{Remark}
\newtheorem{hyp}{Hypothesis}

\date{}

\begin{document}
%\hlinenumbers

\title{When can a population spreading across sink habitats persist ?\thanks{Funded by the grants 200020 196999 and 200020\_219913 from the Swiss National Foundation.}}

\author{{Michel Benaim
\thanks{\href{mailto:michel.benaim@unine.ch}{michel.benaim@unine.ch}. Institut de Math\'ematiques, Universit\'e de Neuch\^atel, Switzerland.}}
\and{Claude Lobry
\thanks{\href{mailto:claude\_lobry@orange.fr}{claude\_lobry@orange.fr}. C.R.H.I, Universit\'e Nice Sophia Antipolis, France.}}
\and{Tewfik Sari
\thanks{Corresponding author. Email: \href{mailto:tewfik.sari@inrae.fr}{tewfik.sari@inrae.fr}. ITAP, Univ Montpellier, INRAE, Institut Agro, Montpellier, France.}}
\and{Edouard Strickler
\thanks{\href{mailto:edouard.strickler@univ-lorraine.fr}{edouard.strickler@univ-lorraine.fr}. Universit\'e de Lorraine, CNRS, Inria, IECL, Nancy, France.}}}
\date{\today}
\maketitle

\begin{abstract}
We consider populations with time-varying growth rates living in sinks. Each population, when isolated, would become extinct. Dispersal-induced growth (DIG) occurs when the populations are able to persist and grow exponentially when dispersal among the populations is present. We provide a mathematical analysis of this surprising phenomenon, in the context of a deterministic model with periodic variation of growth rates and non-symmetric migration which are assumed to be piecewise continuous. We also consider a stochastic model with random variation of growth rates and migration.  This work extends existing results of the literature on the DIG effects obtained for periodic continuous growth rates and time independent symmetric migration. \\
{\bf Keywords}. {Dispersal-induced growth. 
Periodic linear cooperative systems. Principal Lyapunov exponent. 
Averaging. Singular perturbations. 
Perron root. Metzler matrices. Sinks. Stochastic environment. Markov Feller process.}
\end{abstract}

\tableofcontents
\section{Introduction}

Many plant and animal populations live in separate patches which have different environmental conditions and  that are connected by dispersal. The study of the interaction between organism dispersal and environmental heterogeneity, both spatial and temporal, to determine population growth is a central theme of ecological theory \cite{Baguette,Hanski}. 

A patch is called a {\em source} when, in the absence of dispersion, the environmental conditions lead to the persistence of the population, and a {\em sink } when, on the contrary, they lead to the extinction of the population. A basic insight of source-sink theory is that populations in sinks may be sustained, as a result of immigration from source patches \cite{Pulliam}.  A more surprising phenomenon is that {\em Populations can persist in an environment consisting of sink habitats only}  as announced in the title of \cite{Jansen}. In fact, this somewhat paradoxical effect of dispersal has already been pointed by Holt \cite{Holt} on particular systems, later called by him {\em inflation} \cite{HOLT02,Holt-et-al}, 
and has been addressed  in various papers considering discrete time models (\cite{Holt, Holt-et-al, Jansen} ),
 continuous time models 
(\cite{HOLT02,KLA08})
 and stochastic models 
 (\cite{ROY05, SCH10,EVA13}) on two patches (see \cite{benaim} and the references therein for more information). A recent paper (\cite{HOLTPNAS20}) shows the importance of this phenomenon in   the epidemiological context.
 
  A common feature to all these paradoxical effects is that  the local growth rates are, on each patch, periodic of period $T$ (which mimics seasonality or any other periodic environment), of the form $r(t/T)$, where $r(\cdot)$ is 1-periodic, and patches are connected by migration.

Motivated by these pioneering works  on the paradoxical effect of inflation a thorough mathematical study has been undertaken to clarify the origin and characteristics of this phenomenon in the case of continuous time systems  by Katriel  \cite{Katriel} and Benaim et al. \cite{benaim}. The present paper continues these investigations.

The Katriel's paper \cite{Katriel} considers the case of $n\geq 2$ patches,  {constant and symmetric}  migration rates $m\ell_{ij}$, and such that the  matrix $L =(\ell_{ij})$ describing the connections between the patches is irreducible, which means that every  patch is connected to any other one through a sequence of patches.
In this case it is true that, due to the mixing effet of migration, the asymptotic growth rate is the same on every patch. 
Since it depends on the strength $m$ of the migration   and on the period $T$, this common value is denoted by $\Lambda(m,T)$.

Katriel introduces an index $\chi$ which is the mean on the period of  $\max_i r_i(t)$  where $r_i(t)$  is the  growth  rate on patch $i$
 and he shows  that, if and only if $\chi >0$,  there is a critical function $ T_c(m) $  which separates the  space of parameters $T$ and $m$
in two regions such that when  $T > T_c(m) $ the system is increasing, otherwise decreasing. Moreover \cite{Katriel} gives some information on the shape of $ T_c(m) $ and the asymptotic  of $\Lambda(m,T)$ for  $T  \to 0$ (fast regime) and $T \to \infty$ (slow regime).
In particular it turns out that, provided that $T$ is not too small, $\Lambda(m,T)$ is positive when $m$ takes intermediate values, not too small, not too large.

Our paper \cite{benaim} is less general than Katriels's in the sense that we consider the case of only
two patches connected by migration but more general in the sense that  we assume that the $r_i(t)$'s
are piecewise continuous or random.
We show that the conclusion of \cite{Katriel} remains true in the deterministic case and can be  generalized to the stochastic framework. 
In the present paper we address the case of 
a deterministic or stochastic model with $n$ patches ($n\geq 2$) such that:
\begin{itemize}
\item The matrix $L(t)  = (\ell_{ij}(t)) $ describing the connection between patches is still irreducible but is no longer constant nor symmetric, only $T$-periodic.
\item The functions $r_i(t)$ and $\ell_{ij}(t)$ are piecewise continuous or stochastic.
\end{itemize}

This framework makes it possible to generalize  the models of  \cite{Katriel} and \cite{benaim} to models with more realistic assumptions, especially regarding the symmetry of migration rates which is rather unlikely in a real ecological system. We prove that all the results obtained in \cite{Katriel,benaim} remain true in this more general setting except for the existence of the critical function $ T_c(m) $ which is specific to symmetric and constant migration rates.

The style of our mathematical treatment is the one of dynamical systems theory and aside the Perron-Frobenius theorem we make use of two classical tools in this area:  the method of averaging and Tikhonov's theorem on singular perturbations (which is explained in the appendix)

The paper is organized as follows. In Section \ref{PE} we consider a periodic environment where the growth rates and the migration matrix are periodic functions. The formulation of our main results are given in Section \ref{Results} (some proofs are postponed to Section \ref{IP}).  
In Section \ref{SE} we consider a stochastic environment where the growth rates and the migration matrix depend on a Markov process.
In Section \ref{Numerical}, by means of numerous numerical illustrations of the cases with 2 and 3 patches (where {symbolic computation} 
software makes it possible to exploit explicit formulas) we tried to {visually illustrate} 
our principal results. 
In Section \ref{Discussion} we discuss the results in more detail and propose some questions for further research.
In Appendix \ref{PFT} we provide some consequences of the Perron-Frobenius theorem which are needed thorough the paper.
In Appendix \ref{CLDE} we provide some results which are needed in Section \ref{IP}. In Appendix \ref{SecTikhonov} we provide a statement of the theorem of Tikhonov on singular perturbations which is used to prove the asymptotic behavior of the growth rate when the period is large ($T\to \infty$, see Section \ref{HFL}) or the migration rate is large ($m\to \infty$, see Section \ref{FM}).
\color{black}

\section{Periodic environment}\label{PE}
%%%%%%%%%%%%%%%%%%%%%%%%
\subsection{The model}
%%%%%%%%%%%%%%%%%%%%%%%
Katriel \cite{Katriel} considered the model of 
populations of sizes $x_i(t)$ ($1\leq i\leq n$), inhabiting $n$ patches, and subject to time-periodic local growth rates $r_i(t/T)$ ($1\leq i\leq n$), where it is assumed that $r_i(\tau)$ are $1$-periodic continuous
functions, so that $r_i(t/T)$ are periodic with 
period $T>0$. The dispersal among the patches $i$ and $j$ ($i\neq j$) is at rate $m\ell_{ij}$ where the parameter $m\geq 0$ measures the strength, and 
the numbers 
\begin{equation}\label{sym}
\ell_{ij}=\ell_{ji},\qquad (i\neq j)
\end{equation} encode the topology of the dispersal network and the relative rates of dispersal among different patches: 
If $\ell_{ij} = 0$, there is no migration between the patches $i$ and $j$, if $\ell_{ij} > 0$, there is a migration. 
We then have the differential equations
\begin{equation}\label{eqKatriel}
\frac{dx_i}{dt}=r_i(t/T)x_i+m\sum_{j\neq i}\ell_{ij}\left(x_j-x_i\right),
\quad 1\leq i\leq n. 
\end{equation}  
Katriel \cite{Katriel} proved that in the irreducible case , any solution of \eqref{eqKatriel} with $x_i(0)>0$ for $1\leq i\leq n$ satisfies $x_i(t)>0$ for all $t>0$ so that 
{we can define the {local growth rates} $\Lambda[x_i]=\lim_{t\to\infty}\frac{1}{t}\ln(x_i(t))$, provided this limit exists. It is shown in \cite{Katriel} that, when $m>0$, the numbers $\Lambda[x_i]$, $1\leq i\leq n$ are equal, and moreover 
they do not depend on the initial condition. Their common value is called the growth rate of the system (\ref{eqKatriel}) and denoted $\Lambda(m,T)$. In fact $\Lambda(m, T)$
is the maximal Lyapunov exponent of the system. This remark pertains to
the use of the term `Lyapunov exponent' throughout the paper.} The main results in \cite{Katriel} are on the asymptotic properties of $\Lambda(m,T)$ when $T\to 0$ and $T\to\infty$. 
{An important result, which {plays} a significant role in the proofs of the main results of \cite{Katriel} is that for all $m>0$, the function $T\mapsto\Lambda(m,T)$ is increasing. Actually, it is strictly increasing except in the case that all $r_i(\tau)$ are equal, where it is a constant function, see \cite[Lemma 2]{Katriel}. This result follows from general results of Liu et al. \cite{liu}.}

%%%%%%%%%%%%%%%%%%%%%%
{We now extend the model \eqref{eqKatriel} to
the case of asymmetric and time-dependent migration.}
We denote by $\ell_{ij}(\tau)\geq 0$ the migration term, from patch $j$ to patch $i$. At time $\tau$, there is a migration from patch $j$ to patch $i$ if and only if $\ell_{ij}(\tau) > 0$. The differential equations \eqref{eqKatriel} become
\begin{equation}\label{eq1}
	\frac{dx_i}{dt}=r_i(t/T)x_i+m\sum_{j\neq i}\left(\ell_{ij}(t/T)x_j-\ell_{ji}(t/T)x_i\right),\quad 1\leq i\leq n. 
\end{equation}
We make the following assumption 

\begin{hyp} \label{H1}
The functions
$\tau\mapsto r_i(\tau)$ and  
$\tau\mapsto \ell_{ij}(\tau)$
are piece-wise continuous 1-periodic functions, with a finite set of discontinuity points on each period. Moreover, they have left and right limits at the discontinuity points.
\end{hyp}

{Hypothesis \ref{H1} implies
that} the solutions of \eqref{eq1} are continuous and piecewise $\mathcal{C}^1$ functions satisfying  \eqref{eq1} except at the points of discontinuity of the functions $r_i$ and $\ell_{ij}$.
The matrix
$L(\tau)$ whose off diagonal elements are $\ell_{ij}(\tau)$, $i\neq j$, and diagonal elements $\ell_{ii}(\tau)$ are given by
\begin{equation}\label{Lii}
\ell_{ii}(\tau)=-\sum_{j\neq i} \ell_{ji}(\tau),\quad 1\leq i\leq n,
\end{equation}
is called the migration or dispersal matrix. 
Using the matrix $L(\tau)$,   \eqref{eq1} can be written as
\begin{equation}\label{eq3}
\frac{dx}{dt}=A(t/T){x},\qquad A(\tau)=R(\tau)+mL(\tau)
\end{equation}
where  ${x}=\left(
	x_1,\cdots,
	x_n\right)^\top$ and 
$R(\tau)={\rm diag}\left(
	r_1(\tau),\cdots,r_n(\tau)
\right).$
In addition to the assumptions that $L(\tau)$  has non-negative off diagonal elements ($\ell_{ij}(\tau)\geq 0$ for $i\neq j$), we also make the following assumption 
\begin{hyp}\label{H2}
For all $\tau$, $L(\tau)$ is {\it{irreducible}}.
\end{hyp}
{
This assumption, means that at each time, every patch is reachable from every other patch,
either directly or by a path through other patches. When the migration is time-independent, the irreducibility hypothesis is very usual in ecological models with mobility of the populations, see \cite{arino,Katriel}. However, when migration is time-dependent, some results remain valid without the assumption of the irreducibility of the matrix migration $L(\tau)$ for all $\tau$, see Section \ref{INNsec}.}

\subsection{The growth rate}

We use the following notations: for $x\in\mathbb{R}^n$,
$x\geq 0$ means that for all $i$, $x_i\geq 0$, $x> 0$ means that $x\geq 0$ and $x\neq 0$, and $x\gg 0$ means that for all $i$, $x_i> 0$. 

If $m>0$ we have the property that if at $t=0$ the population is present in at least one patch, then it will be present in all patches for all $t>0$.  
{Indeed, 
since $\ell_{ij}(\tau)\geq 0$ for $i\neq j$ and $L(\tau)$ is irreducible, for all $m>0$ and $t\geq 0$, the matrix $A(t/T)$ in \eqref{eq3} is an irreducible cooperative matrix. Hence, using a classical  result on irreducible cooperative linear systems (see \cite[Theorem 1.1] {hirsch} or \cite[Lemma]{Slom}), $x(0)>0$ implies $x(t)\gg 0$ for all $t>0$.}

The previous result needs that the migration matrix $L(\tau)$ is irreducible and that dispersal is present. Indeed, in the absence of dispersal ($m=0$) the population in each patch would evolve independently, and the differential equations are solved to yield
\begin{equation}\label{diagcase}
x_i(t)=x_i(0)e^{\int_0^t r_i(s/T)ds},\qquad 1\leq i\leq n.\end{equation}
Given a function $u:[0,\infty)\rightarrow (0,\infty)$ we will
denote its Lyapunov exponent by
\begin{equation*}
\Lambda[u]=\lim_{t\rightarrow \infty}\frac{1}{t}\ln(u(t)),
\end{equation*}
provided this limit exists.
Note that $\Lambda[u]>0$ 
corresponds to exponential growth, while $\Lambda[u]<0$ 
corresponds to exponential decay - leading to extinction. 
Therefore, as shown by \eqref{diagcase}, we have
\begin{equation}
\label{eq7}
\mbox{If }m=0\mbox{ then }\Lambda[x_i]=\lim_{t\rightarrow\infty} \frac{1}{t}\int_0^t r_i(s/T)ds=\overline{r}_i,
\end{equation}
where
\begin{equation}
\label{barri}
\overline{r}_i=\int_0^{1}r_i(\tau)d\tau,\qquad 1\leq i \leq n
\end{equation}
are the {\it{local}} average growth rates in each of the patches.  

If $m>0$,  for any solution of (\ref{eq3}) with 
$x(0)>0$, we have $x_i(t)>0$ for all $t>0$, so that we can define the Lyapunov exponents $\Lambda[x_i]$, for $1\leq i\leq n$. 
However, the formula \eqref{eq7} giving $\Lambda[x_i]$ is no longer true.
The study of $\Lambda[x_i]$ when the patches are coupled by dispersion ($m > 0$) is more difficult than in the uncoupled case, because equations \eqref{eq3} cannot be solved as in the case where $m=0$.

Since the system \eqref{eq3} is a periodic system, its study reduces to the study of its monodromy matrix $\Phi(T)$, where $\Phi(t)$ is the {fundamental} matrix solution, i.e. the solution of the matrix-valued differential equation
\begin{equation}\label{FMA}
\frac{dX}{dt}=A(t/T)X,
\end{equation}
with initial condition $X(0)=Id$, the identity matrix.
Since the matrix $A(\tau)$ has off diagonal non-negative entries (such a matrix is usually called \emph{Metzler} or \emph{cooperative}), 
{and is irreducible for all $\tau$, the monodromy matrix $\Phi(T)$ of \eqref{eq3} has positive entries, see Lemma \ref{lemme1} in Appendix \ref{CLDE}. Therefore, by the Perron theorem, it has a dominant eigenvalue (an eigenvalue of maximal modulus, called the \emph{Perron root}), which is positive, see Theorem \ref{Ptheorem} in Appendix \ref{PFT}.} 
We denote it by  
$\mu(m,T)$, to emphasize its dependence on $m$ and 
$T$. We have the following result.

\begin{propo}\label{Prop3}
Assume that Hypotheses \ref{H1} and \ref{H2} are satisfied.
Suppose $m>0$ and $T>0$. Let $\mu(m,T)$ be the Perron root of the monodromy matrix $\Phi(T)$ of \eqref{eq3}. If $x(t)$ is a solution of \eqref{eq3} such that $x(0)>0$, then for all $i$ 
 \begin{equation}\label{Lambda}
\Lambda[x_i]=\Lambda(m,T):=\frac{1}{T}\ln \left( \mu(m,T)\right). 
\end{equation}
The function $\Lambda$ is analytic in $m$ and $T$.
\end{propo}
\begin{proof}
{
The proof is given in Section \ref{IP1}.}
\end{proof}

{
\begin{rem}
This result is known in the case of continuous $R$ and constant migration matrix $L$, see \cite[Eq. (20)]{Katriel}. In \cite{Katriel} the growth rate is seen as a principal eigenvalue (the one with largest real part) of the
periodic problem $d\phi/dt=A(t/T)\phi-\lambda\phi$, $\phi(t+T)=\phi(t)$, associated to \eqref{FMA}, see \cite[Section 3.1]{Katriel}. For a general cooperative and irreducible matrix $A(\tau)$, i.e. $L$ is not assumed to be constant or symmetric, it is known that  $\frac{1}{T}\ln ( \mu(T))$, can be seen as the principal eigenvalue of the corresponding 
periodic problem, see \cite[Lemma 3.1]{liu}. 
\end{rem}
}

The fundamental fact is that, when dispersal is present, the Lyapunov exponents $\Lambda[x_i]$ of all components
$x_i(t)$ are equal, and moreover 
they do not depend on the initial condition. Following \cite{Katriel} we adopt the following definition.

\begin{Def}
The growth rate of the system \eqref{eq3} is the common value 
$\Lambda(m,T)$ given by \eqref{Lambda} 
of the Lyapunov exponents $\Lambda[x_i]$ of all components of any solution $x(t)$ such that $x(0)>0$. 
\end{Def}

The main problem is to study the dependence of $\Lambda(m,T)$ in  $m$ and $T$.  
In contrast to autonomous systems, studying the Perron root of the monodromy matrix of periodic systems analytically is 
challenging, and rarely possible. Thus, the formula \eqref{Lambda} is of little practical interest. However, much can be said on the asymptotics of $\Lambda(m,T)$,  for large and small $m$ or $T$, as shown in Section \ref{asymptotics}.

For piecewise constant local growths and migration rates, it is possible to compute the monodromy matrix {in closed form}, and select its largest eigenvalue in modulus, and use the formula \eqref{Lambda} to plot the graph of the function $(m,T)\mapsto\Lambda(m,T)$. For  details and complements, see \cite{benaim} and Section \ref{Numerical}.

\subsection{The DIG threshold}

Following \cite{Katriel} we adopt the following definition.

\begin{Def}\label{DIG} We say that {\it{dispersal-induced growth}} (DIG) occurs 
if all patches are sinks ($\overline{r}_i<0$ for $1\leq i\leq n$), but  $\Lambda(m,T)>0$ for some values of $m$ and $T$.
\end{Def}
This means that each of the populations would become extinct if isolated, but dispersal, at an appropriate rate, induces
exponential growth in all patches. 
The following number was defined by Katriel \cite{Katriel} and plays an important role 
\begin{equation}
\label{chi}
\chi:=\int_0^1\max_{1\leq i\leq n}r_i(\tau)d\tau.
\end{equation}
We have the following result

\begin{theorem} \label{upperborneLambda} 
For all $m>0$ and $T>0$ we have
$\Lambda(m,T) \leq \chi.$ Therefore if $\chi\leq 0$ then  $\Lambda(m,T)\leq 0$ for all $m>0$ and $T>0$, so that DIG does not occur.
\end{theorem} 

\begin{proof}
We define $r_{max}(\tau)=\max_{1\leq i\leq n}r_i(\tau)$. From \eqref{eq1} we have
$$\frac{dx_i}{dt}\leq r_{max}(t/T)x_i+m\sum_{j\neq i}\left(\ell_{ij}(t/T)x_j-\ell_{ji}(t/T)x_i\right),\qquad 1\leq i\leq n.$$
Adding these equations and setting $\rho=\sum_{i=1}^nx_i$ we have
$$\frac{d\rho}{dt}\leq r_{max}(t/T)\rho(t),$$
which implies
$$\begin{array}{lcl}
\Lambda(m,T)&=&\Lambda[x_i]=\lim_{t\to\infty}\frac{1}{t}\ln(x_i(t))
\leq \limsup_{t\to\infty}\frac{1}{t}\ln(\rho(t))\\
&\leq& \limsup_{t\to\infty}\frac{1}{t}\int_0^t
r_{max}(s/T)ds=\int_0^1r_{max}(\tau)d\tau=\chi.
\end{array}
$$
This proves the theorem.
\end{proof}

\begin{rem}\label{remDIG}
Consider an idealized habitat whose growth rate at any time, is that of the habitat with maximal growth at this time. Hence $\chi$, defined by \eqref{chi}, is the average growth rate in this idealized habitat. 
If the population does not grow exponentially in this idealized habitat (i.e. if $\chi\leq 0$), then from Theorem \ref{upperborneLambda} we deduce that DIG does not occur. 
\end{rem}

One of our main results is that the condition $\chi>0$ which is necessary for DIG to occur is also sufficient, i.e. as soon as $\chi>0$ then there are values of $m$ and $T$ for which $\Lambda(m,T)>0$. For this reason we call $\chi$ the \emph{DIG threshold}. To prove this result we will study the asymptotic behavior of $\Lambda(m,T)$ when $m$ and $T$ are infinitely small or infinitely large.

\section{Results}\label{Results}
%%%%%%%%%%%%%%

\subsection{Definitions and notations} 
For the statement of results, it is necessary to recall some classical results. 
If a matrix $A$ is  \emph{Metzler} and irreducible, 
from the Perron-Frobenius theorem, we know that its spectral abscissa, i.e. the maximum of the real parts of its eigenvalues, is an eigenvalue of $A$, usually called its \emph{dominant eigenvalue}, or the \emph{Perron-Frobenius root}  and denoted 
$\lambda_{max}(A)$, 
see Theorem \ref{PFtheorem} in Appendix \ref{PFT}.
If $A$ is symmetric, then $\lambda_{max}(A)$ is simply the maximal eigenvalue of $A$.

We also need the following result, which is well known in the literature, see for example 
\cite[Lemma 1]{cosner},
\cite[Lemma 1]{arino}, \cite[Lemma 4.1]{elbetch2021} or \cite[Lemma 3.1]{elbetch2022}. 
%%%%%%%%%%%%%%%
\begin{lem}\label{lm41}
If a matrix $L$ is Metzler irreducible and its columns sum to $0$, then, $0$ is a simple eigenvalue of $L$ and all non-zero eigenvalues of $L$ have negative real part. Moreover, 
the null space of the matrix $L$ is generated by a positive vector.
If the matrix $L$ is symmetric, then this vector is $\delta=(1,...,1)^\top$.
\end{lem}

This result follows from Theorem \ref{PFtheorem} in Appendix \ref{PFT} and the fact that the spectral abscissa of $L$ is $\lambda_{max}(L)=0$.

\begin{rem}\label{delta}
A positive vector  $\delta=(\delta_1,\ldots, \delta_n)^\top$ which generates the null space of the matrix $L$ is given explicitly by
$\delta_i=(-1)^{n-1}L_{ii}^*$, 
 where  
$L_{ii}^*$ is the co-factor of the $i$-th diagonal entry $l_{ii}$ of  $L$, see \cite[Lemma 2.1]{guo} or \cite[Lemma 3.1]{gao}. 
\end{rem}

The following notations are used:
\begin{itemize}
\item If $u(\tau)$ is any 1-periodic object (number, vector, matrix...), we denote by $\overline{u}=\int_0^1u(\tau)d\tau$ its average on one period.  Therefore the number  \eqref{chi} of Katriel  is denoted
$\chi:=\displaystyle\overline{\max_{1\leq i\leq n}r_i}$.
\item For all $\tau\in(0,1)$, $p(\tau)\gg 0$ is the unique positive eigenvector of $L(\tau)$, such that 
$L(\tau)p(\tau)=0$, $\sum_{i=1}^np_i(\tau)=1$ (exists according to Lemma \ref{lm41}, since $L(\tau)$ is Metzler irreducible and its columns sum to 0). 
{Note that $p(\tau)$ is piecewise continuous, so that the
integrals involving this function defined later are well-defined.}
If the matrix $L(\tau)$ is symmetric, then $p_i(\tau)=1/n$ for all $i$. 
\item Similarly,
$q\gg 0$ is the unique positive eigenvector of $\overline{L}$ such that 
$\overline{L}q=0$ and
$\sum_{i=1}^nq_i=1$ (exists according to Lemma \ref{lm41}, since $\overline{L}$ is Metzler irreducible and its columns sum to 0). It should be noticed that, in general, we do not have $q=\overline{p}$, where $p(\tau)$ is the positive eigenvector of $L(\tau)$.
\item For all $\tau\in[0,1]$, $\lambda_{max}(R(\tau)+mL(\tau))$ is the dominant eigenvalue of the matrix $R(\tau)+mL(\tau)$ (exists, since $R(\tau)+mL(\tau)$ is Metzler irreducible). 
\item Similarly, the dominant eigenvalue $\lambda_{max}\left(\overline{R+mL}\right)$ is also well defined. 

\end{itemize}

\subsection{Asymptotics of $\Lambda(m,T)$ for large or small $m$ and $T$}\label{asymptotics}

We have the following result on the limits of  $\Lambda(m,T)$ as $T\to 0$, $T\to\infty$, $m\to 0$ or $m\to\infty$.

\begin{theorem} \label{thm1} 
{Assume that Hypotheses \ref{H1} and \ref{H2} are satisfied}. The growth rate $\Lambda(m,T)$ of \eqref{eq3} satisfies the following properties

\begin{enumerate}
\item 
{\bf{(Fast regime)}} For all $m>0$ we have 
\begin{equation}\label{T=0}
{
\Lambda(m,0):=\lim_{T\to 0}\Lambda(m,T)}
=\lambda_{max}\left(\overline{R+mL}\right).
\end{equation}

\item 
{\bf{(Slow regime)}} For all $m>0$ we have  
\begin{equation}
\label{T=infini}
{
\Lambda(m,\infty):=\lim_{T\to\infty}\Lambda(m,T)}
=\overline{\lambda_{max}({R}+mL)}.
\end{equation}

\item 
{\bf{(Slow migration)}} For all $T>0$ we have
\begin{equation}
\label{m=0}
{
\Lambda(0,T):=\lim_{m\to0}\Lambda(m,T)}=\max_{1\leq i\leq n}\overline{r}_i.
\end{equation}

\item 
{\bf{(Fast migration)}} For all $T>0$ we have
\begin{equation}
\label{m=infini}
{
\Lambda(\infty,T):=\lim_{m\to\infty}\Lambda(m,T)}=\sum_{i=1}^n\overline{p_ir_i}.
\end{equation}
\end{enumerate}
\end{theorem} 

\begin{proof}
{
The formula \eqref{T=0} is proved in Section \ref{HFL}, \eqref{T=infini} is proved in Section \ref{LFL}, \eqref{m=0} is proved in Section \ref{SM} and \eqref{m=infini} is proved in Section \ref{FM}.}
\end{proof}

 We have the following results on the limit functions $\Lambda(m,0)$ and $\Lambda(m,\infty)$.

\begin{propo}\label{PropDL}
{Assume that Hypotheses \ref{H1} and \ref{H2} are satisfied}.
The functions $\Lambda(m,0)$ and $\Lambda(m,\infty)$ defined by \eqref{T=0} and \eqref{T=infini} respectively
satisfy the following properties.
%%%%%%%%
\begin{align}
\label{T=0m}
{
\Lambda(0,0):=\lim_{m\to0}\Lambda(m,0)}=\max_{1\leq i\leq n}\overline{r}_i,
&\quad
{
\Lambda(\infty,0):=\lim_{m\to\infty}\Lambda(m,0)}=\sum_{i=1}^n
q_i\overline{r}_i.\\
\label{T=infini,m=0}
{
\Lambda(0,\infty):=\lim_{m\to 0}\Lambda(m,\infty)}=\chi,
&\quad
{
\Lambda(\infty,\infty):=\lim
_{m\to \infty}\Lambda(m,\infty)}
=\sum_{i=1}^n\overline{p_i{r}_i}.
\end{align}	
%%%%%%%%%
Moreover, we have
\begin{equation}\label{Lambda(m,0)convex}
\frac{d}{dm}\Lambda(m,0)\leq 0,\qquad \frac{d^2}{dm^2}\Lambda(m,0)\geq 0,
\end{equation}
and equalities hold {if and only if}  $\overline{r}_i=\overline{r}$, for all $i$, {in which case}  $\Lambda(m,\infty)=\overline{r}$ for all $m>0$ and we have
\begin{equation}\label{Lambda(m,infini)convex}
\frac{d}{dm}\Lambda(m,\infty)\leq 0,\qquad \frac{d^2}{dm^2}\Lambda(m,\infty)\geq 0,
\end{equation}
and equalities hold {if and only if ${r}_i(\tau)={r}(\tau)$, for all $i$, in which case}
$\Lambda(m,T)=\overline{r}$ for all $m>0$ and $T>0$. 

In the constant migration case, for all $m>0$ and $T>0$, we have 
\begin{equation}\label{Lambda(m,0)<Lambda(m,T)}
\Lambda(m,0)\leq \Lambda(m,T),
\end{equation}
and hence $\Lambda(m,0)\leq \Lambda(m,\infty)$, for all $m>0$.
\end{propo}

\begin{proof}
The proof is given in Section \ref{ProofPropDL}
\end{proof}

{
 Inequality \eqref{Lambda(m,0)<Lambda(m,T)} is similar to the main result of \cite{Huston}. In the PDE context much has been done on the eigenvalue problem for operators of the form $ v(x,t) \to Hv(x,t)+ h(x,t)v(x,t) $ where $H$ is an unbounded operator on the space $B$  of periodic  functions and $h(x,t)$ is periodic with respect to $t$. In this context $H$ corresponds to our migration matrix $mL$ and the function $h(x,t)$ to our diagonal matrix $R(t)$. Thus it is likely that some of our results, are already present in the PDE literature. Nevertheless our regularity  assumption that $R(t)$ is only piecewise continuous is not usual and one must be very cautious.}

{The formulas \eqref{T=0m} and \eqref{T=infini,m=0} give} the limits of $\Lambda(m,0)$ and $\Lambda(m,\infty)$ as $m\to 0$ or $m\to\infty$.
The formulas \eqref{Lambda(m,0)convex} and \eqref{Lambda(m,infini)convex}
assert that the functions $m\to\Lambda(m,0)$ and $m\to\Lambda(m,\infty)$ are convex decreasing, in contrast with the functions $m\to\Lambda(m,T)$, for $T>0$, which are not always decreasing, see the figures in Section \ref{Numerical}. The last formula  \eqref{Lambda(m,0)<Lambda(m,T)} asserts that when the migration matrix $L$ is time independent, then for any $m>0$, $\Lambda(m,0)$ is a lower bound of $\Lambda(m,T)$. This property is not true in the case where the migration matrix is time dependent, see the supplementary material Figures S2(c), S4(a) and S5(c).

\begin{figure}[ht]
\begin{center}
\includegraphics[width=10cm,
viewport=160 490 450 690]{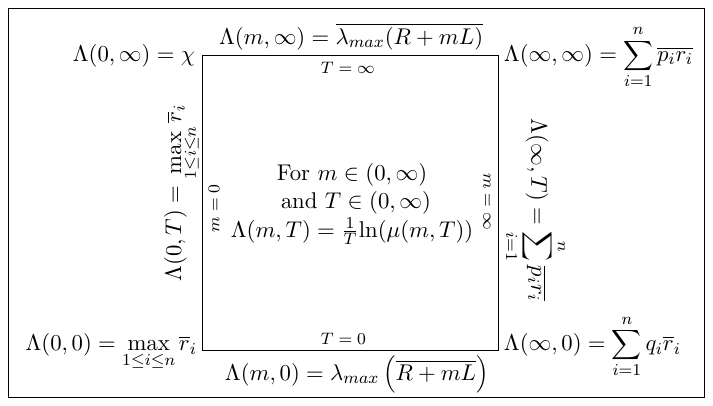}\caption{The definition of $\Lambda(m,T)$ and its limits when $T$ tends to 0 or $\infty$ and/or $m$ tends to 0 or $\infty$.
\label{fig1}}
\end{center}
\end{figure}

The results of Theorems \ref{thm1} and Proposition \ref{PropDL}
are summarized in Figure \ref{fig1}.
Note that 
$\textstyle
\Lambda(0,0)=\Lambda(0,T)$ and 
$\textstyle
\Lambda(\infty,\infty)=
\Lambda(\infty,T)$, for all $T>0$. 
However, in general $\chi\neq\max_{1\leq i\leq n}\overline{r}_i$, so that
$\textstyle
\Lambda(0,\infty)\neq
\lim_{T\to \infty}\Lambda(0,T)$. 
{In addition}, in general 
$\sum_{i=1}^nq_i\overline{r}_i\neq
\sum_{i=1}^n\overline{p_ir_i}$, 
so that
$
\textstyle
\Lambda(\infty,0)\neq
\lim_{T\to 0}\Lambda(\infty,T)$.

\begin{rem}\label{Lconstant}
In the case of time independent migration, for all $i$ we have $q_i=p_i$ where $p\gg 0$ is the positive eigenvector of $L$ such that $\sum_{i=1}^np_i=1$ and $Lp=0$. Hence, for all $T>0$, $\Lambda(\infty,0)=\Lambda(\infty,\infty)=\Lambda(\infty,T)=\sum_{i=1}^np_i\overline{r}_i$.
\end{rem}

\begin{rem}
In the case where the matrix of migrations $L(\tau)$ is symmetric we have $p_i(\tau)=1/n$, so that 
$\Lambda(\infty,0)=\Lambda(\infty,\infty)=\Lambda(\infty,T)
=\frac{1}{n}\sum_{i=1}^n\overline{r}_i.
$
 \end{rem} 

In the case where the migration matrix is constant and symmetric, and the local growth rates $r_i(\tau)$ are continuous functions, 
{the limits
\eqref{T=0} and \eqref{T=infini} were given in \cite[Theorem 2.1]{liu}, see also \cite[Lemma 8]{Katriel} and \cite[Lemma 5]{Katriel}}. The double limit 
\eqref{T=0m} was given in \cite[Lemma 9(i,ii)]{Katriel} and
\eqref{T=infini,m=0} was given in \cite[Theorem 1 and Lemma 7(ii)]{Katriel}. 
The results of items 3 (Slow migration) and 4 (Fast migration) of Theorem \ref{thm1} were not considered in \cite{Katriel}. 

\subsection{The DIG phenomenon}
Let $\chi$ be the DIG threshold defined by \eqref{chi}.
An immediate consequence of
\begin{equation}\label{chieq}
\Lambda(m,T)\leq \chi\quad\mbox{ and }\quad
\Lambda(0,\infty):=\lim_{m\to0}\lim_{T\to\infty}\Lambda(m,T)=\chi,
\end{equation}
proved in Theorem \ref{upperborneLambda}, and \eqref{T=infini,m=0}, respectively,  is the following result. 

\begin{theorem}\label{DIGthm}
Assume that Hypotheses \ref{H1} and \ref{H2} are satisfied.
We have  
$
\sup_{m,T}\Lambda(m,T)=\chi
.$
{Therefore, if $\overline{r}_i<0$ for all $i$, DIG occurs if and only if $\chi>0$}.   
\end{theorem}

This result is a first answer to the question posed in the title of the paper: a population spreading across sink habitats can grow exponentially if and only if $\chi > 0$.
As stated in Remark \ref{remDIG},
$\chi=\overline{\max_ir_i}$ is the average growth rate of an idealized habitat whose growth rate at any time is that of the habitat with maximal growth. Hence, the population can survive if and only it would survive in this idealized habitat. Moreover, thanks 
to Theorem \ref{thm1} and Proposition \ref{PropDL}, we can answer more precisely, as stated in the following remark.

\begin{rem} If $\chi>0$ and $\max_i\overline{r_i}<0$ then the population is growing exponentially if the environment is slowly varying and if the dispersal rate across the patch is small, but not too small. Indeed from $\lim_{m\to0}\Lambda(m,T)=\max_i\overline{r_i}<0$ we deduce that if $T$ is fixed and $m$ is very small then 
$\Lambda(m,T)<0$ and from the double limit in \eqref{chieq} we deduce that 
{
if $m$ is small enough and $T$ is large enough
(where `large enough' depends on the value of $m$)}, then 
$\Lambda(m,T)>0$.
\end{rem}
%%%%%%%%%%%%%%%%%%%%%

We can give a more precise description of the set of $m$ and $T$ for which {growth occurs}. 

%%%%%%%%%%%%%%%%%%
\begin{propo}\label{PropDIG}
Assume that $\chi>0$ and $\max_{1\leq i\leq n}\overline{r}_i<0$. 
Two cases must be distinguished.\\
1. If $\sum_{i=1}^n\overline{p_ir_i}<0$,  then the equation $\Lambda(m,\infty)=0$ has a unique solution $m=m^*>0$, and  we have	
	\begin{itemize}
	\item If $m\in (0,m^*)$ then for any $T$ sufficiently large (depending on $m$), we have 
		$\Lambda(m,T)>0$ (growth) and for any $T$ sufficiently small (depending on $m$), we have 
		$\Lambda(m,T)<0$ (decay).
	\item If $m\geq m^*$ then for any $T$ sufficiently small or sufficiently large (depending on $m$), we have $\Lambda(m,T)<0$ (decay).  
	\end{itemize}	
2. If $\sum_{i=1}^n\overline{p_ir_i}\geq 0$,  then  $\Lambda(m,\infty)>0$ for all $m> 0$ and for any $T$ sufficiently large (depending on $m$), we have 
		$\Lambda(m,T)>0$ (growth) and for any $T$ sufficiently small (depending on $m$), we have 
		$\Lambda(m,T)<0$ (decay).
\end{propo}
%%%%%%%%%%%%%%%%%%%

%%%%%%%%%%%%%%%%%%%
\begin{proof}
Assuming $\chi>0$ and $\overline{r}_i<0$, Proposition \ref{PropDL} tells us that 
$\Lambda(m,0)$ is a decreasing function from $\Lambda(0,0)=\max_{1\leq i\leq n}\overline{r}_i<0$ to 
$\Lambda(\infty,0)=\sum_{i=1}^nq_i\overline{r}_i<0$. We conclude that $\Lambda(m,0)<0$ for all $m>0$. 
Therefore, \eqref{T=0} tells us that for sufficiently small $T$ (depending on $m$) we have $\Lambda(m,T)<0$. 

On the other hand
$\Lambda(m,\infty)$ is a strictly decreasing function from $\Lambda(0,\infty)=\chi>0$ to $\Lambda(\infty,\infty)=\sum_{ii=1}^n\overline{p_ir_i}$. 
In the first case we have $\Lambda(\infty,\infty)<0$ so that
the equation  $\Lambda(m,\infty)=0$ has a unique solution $m=m^*>0$ and $\Lambda(m,\infty)>0$ for $m\in(0,m^*)$, 
$\Lambda(m,\infty)<0$ for $m>m^*$.  Therefore, \eqref{T=infini} tells us that if  $m\in(0,m^*)$, then for sufficiently large $T$ we have $\Lambda(m,T)>0$ and if $m\geq m^*$, then for sufficiently large $T$ we have $\Lambda(m,T)<0$.
In the second case we have $\Lambda(\infty,\infty)\geq 0$ so that
$\Lambda(m,\infty)>0$ for all $m>0$.  Therefore, \eqref{T=infini} tells us that for sufficiently large $T$ (depending on $m$) we have $\Lambda(m,T)>0$.
\end{proof}
%%%%%%%%%%%%%%%%%%%%%

\begin{rem}\label{Rem_mstar_fini}
If the migration matrix is time independent, then the second case in Proposition \ref{PropDIG} {never occurs} because
$\sum_{i=1}^n\overline{p_ir_i}=\sum_{i=1}^np_i\overline{r}_i<0$, if $\overline{r}_i<0$ for all $i$. 
{ Case 2 does not occur} either when $L(\tau)$ is symmetric since in this case we have $p_i(\tau)=1/n$, so that $\sum_{i=1}^n\overline{p_ir_i}=\frac{1}{n}\sum_{i=1}^n\overline{r}_i<0$, if $\overline{r}_i<0$ for all $i$. 
An example showing the behavior depicted in the second case of Proposition \ref{PropDIG} is provided in Section \ref{m_star_infini}.
\end{rem}

\subsection{Cooperative linear $T$-periodic systems}
The linear equation \eqref{eq3} is a special case of the apparently more general equation
\begin{equation}\label{A(omega)}
\frac{dx}{dt}=A(t/T)x,
\end{equation}
where $A(\tau)$ is an irreducible Metzler matrix not necessarily equal to $R(\tau)+mL(\tau)$ as in the system  \eqref{eq3}.  Since any irreducible Metzler matrix $A$ can be
decomposed as $R+M$, where $R$ is diagonal, and $M$ is Metzler and irreducible,
and has columns summing to 0 (indeed the elements of $R$ are simply the
column sums of $A$), the equation \eqref{A(omega)} is in fact not more
general than \eqref{eq3}. However, if we introduce the strength of migration  $m$ and put $M=mL$ in the decomposition $A=R+M$, the diagonal matrix will depend on $m$ in the equation \eqref{A(omega)}, whereas it does not in \eqref{eq3}. So, as long as we are interested in the growth rate of \eqref{eq3} and its behavior as a function of period $T$, the results concern  \eqref{A(omega)}.
However, if we are interested in the behavior as a function of $m$, we need to consider the equation \eqref{eq3}. 
 More precisely, we can make the following remark.
\begin{rem}\label{notmoregeneral}
Proposition \ref{Prop3} asserts that the growth rate $\Lambda(T)$ of  \eqref{A(omega)} is given by 
\begin{equation}\label{GR}
\Lambda(T)=\frac{1}{T}\ln(\mu(T))
\end{equation}
where $\mu(T)$ is the Perron root of the monodromy matrix $X(T)$ of \eqref{A(omega)}. Item 1 of Theorem \ref{thm1} asserts that
\begin{equation}\label{GR0}\lim_{T\to0}\Lambda(T)=
\lambda_{max}(\overline{A}),
\end{equation} 
where  $\lambda_{max}(\overline{A})$ is the spectral abscissa of the irreducible Metzler matrix $\overline{A}$.
Item 2 of Theorem \ref{thm1} asserts that 
 \begin{equation}\label{GRinfini}\lim_{T\to\infty}\Lambda(T)=
\overline{\lambda_{max}(A)},
\end{equation} 
where  $\lambda_{max}({A(\tau)})$ is the spectral abscissa of the irreducible Metzler matrix ${A(\tau)}$.
\end{rem}

{The limits \eqref{GR0} and \eqref{GRinfini} have already been proved  in the case where the matrix $A(\tau)$ is continuous, see \cite[Theorem 2.1]{Carmona} and \cite[Theorem 2.1]{liu}. Our results are therefore extensions of these results  to the case where the matrix $A(\tau)$ is piecewise continuous.}

\subsection{An integral formula for the
growth rate}\label{ReductionSimplex}  
\color{black}

A crucial step in the description and the proofs of our main results is to reduce the linear system \eqref{A(omega)} to the $n-1$ simplex 
$$\Delta:=\left\{x\in\mathbb{R}_+^n:\textstyle\sum_{i=1}^nx_i=1\right\}$$ 
of $\mathbb{R}_+^n$.
Indeed, the change of variables
\begin{equation}
\label{rhotheta}
\textstyle
\rho=\sum_{i=1}^nx_i,
\quad
\theta=\frac{x}{\rho},
\end{equation} 
transforms the system \eqref{A(omega)} into
\begin{align}\label{eqrhoT}
\frac{d\rho}{dt}&=\langle A(t/T)\theta,
{\bf 1}\rangle \rho,
\\
\label{eqthetaT}
\frac{d\theta}{dt}&=A(t/T)\theta- \langle A(t/T)\theta,
{\bf 1}\rangle \theta.
\end{align}
Here ${\bf 1}=(1,\ldots,1)^\top$ and $\langle x,{\bf 1}\rangle=\|x\|_1=\sum_{i=1}^nx_i$ is the usual Euclidean scalar product of vectors $x$ and ${\bf 1}$. 
The equation \eqref{eqthetaT} is a differential equation on $\Delta$.

\begin{rem}The variables $\rho$ and $\theta$ defined by \eqref{rhotheta} are interpreted as follows: $\rho$ is the total population present in all patches and $\theta_i=x_i/\rho$ is the  fraction of the population on patch $i$.
\end{rem}

By the Perron theorem the monodromy matrix $\Phi(T)$ of \eqref{A(omega)} has a positive eigenvector $\pi(T)\in\Delta$, called the \emph{Perron vector}, corresponding to the its Perron root $\mu(T)$, see Theorem \ref{Ptheorem} in Appendix \ref{PFT}.
We have the following result.

\begin{theorem}\label{thm2}
Let $\theta^*(t,T)$ be the solution of 
 \eqref{eqthetaT} with initial condition $\theta^*(0,T)=\pi(T)$, where $\pi(T)$ is the Perron vector of the monodromy matrix $X(T)$. Then $\theta^*(t,T)$ is a $T$-periodic solution, and is globally asymptotically stable. Moreover, the growth rate $\Lambda(T):=\frac{1}{T}\ln(\mu(T))$ of equation \eqref{A(omega)} satisfies 
\begin{equation} \label{Lambda=Lambda[rho]1}
 \Lambda(T)=\int_0^1\langle A(\tau)\theta^*(T\tau,T),{\bf 1}\rangle d\tau.
 \end{equation}
\end{theorem}

\begin{proof}
The proof is given in Section \ref{PT8}.
\end{proof}
\color{black}

Using the decomposition $A=R+mL$ in \eqref{eq3}, the system (\ref{eqrhoT},\ref{eqthetaT}) is written
\begin{equation}
\textstyle
\frac{d\rho}{dt}=\rho\sum_{i=1}^nr_i(t/T)\theta_i,
\quad
\textstyle
\label{eqzitheta0}
\frac{d\theta}{dt}=F(t/T,\theta)
\end{equation}
where, 
for $1\leq i\leq n$, $F_i(\tau,\theta)$ is given by
$$F_i(\tau,\theta)=r_i(\tau)\theta_i+m
\sum_{j=1}^n\ell_{ij}(\tau)\theta_j-
\theta_i\sum_{j=1}^nr_j(\tau)\theta_j.$$ 
From Theorem \ref{thm2} we deduce that the second equation in \eqref{eqzitheta0} has a $T$-periodic solution which is globally asymptotically stable. We denote it by 
$\theta^*(t,m,T)$, to recall its dependence on the  parameters $m$ and $T$. It is the solution, with initial condition  $\theta^*(0,m,T)=\pi(m,T)$, where $\pi(m,T)\in\Delta$, is the {Perron vector}, corresponding to the its Perron root $\mu(m,T)$, which was used in \eqref{Lambda} to define $\Lambda(m,T)$.
Moreover, from \eqref{Lambda=Lambda[rho]1} we deduce that 
\begin{align}
\label{Lambdatheta*0}
\Lambda(m,T)&=
\int_0^{1}
\sum_{i=1}^nr_i(\tau)\theta^*_i(T\tau,m,T) 
d\tau.
\end{align}

The formula \eqref{Lambda=Lambda[rho]1}, or its form \eqref{Lambdatheta*0} when the parameter $m$ is taken into account, gives us an integral representation of the growth rate $\Lambda(m,T)$. It will play a major role when we study the limits of $\Lambda(m,T)$ when $T$ is small or large, or when $m$ is large, see Sections \ref{HFL}, \ref{LFL} and \ref{FM}.
\color{black}

\subsection{Time independent migration}
{When the migration matrix is time independent, we have
\begin{equation}\label{infLambda}
\Lambda(\infty,0)=\Lambda(\infty,\infty)=\Lambda(\infty,T)=\sum_{i=1}^np_i\overline{r}_i,
\end{equation}
where $p$ is the positive eigenvector of $L$ such that $Lp=0$ and $\sum_{i=1}^np_i=1$, see Remark \ref{Lconstant}. The aim of this section, see Theorem \ref{infimum}, is
 to prove that the limit \eqref{infLambda} is actually the infimum of $\Lambda(m,T)$. For this purpose, we first establish a new formula for the growth rate $\Lambda(m,T)$.}

\begin{propo}\label{thm2bis}
Assume that the migration matrix $L=(\ell_{ij})$ is time independent. Let $p$ be the positive eigenvector of $L$ such that $Lp=0$ and $\sum_{i=1}^np_i=1$. We have 
\begin{align}
\label{Lambdah(theta)0}
\Lambda(m,T)&=\sum_{i=1}^n
p_i\overline{r}_i+
{m}\int_0^1h\left(\theta^*(T\tau,m,T)\right)d\tau,
\end{align}
where 
$h(x)=\sum_{i=1}^n\left(\sum_{j=1}^n\ell_{ij}x_j\right)\frac{p_i}{x_i}$.
\end{propo}

\begin{proof}
{Let $x(t)$ be a solution of \eqref{eq3}}.
We use the  variable 
$$U=\ln(x_1^{p_1}\cdots x_n^{p_n})=\sum_{i=1}^np_i\ln{x_i}.$$ 
We have
\begin{equation*}
\frac{dU}{dt}=\sum_{i=1}^np_ir_i(t/T)+
mh(x),
\end{equation*}
where $h(x)=\langle Lx,p/x\rangle$, and $p/x=(p_1/x_1,\cdots,p_n/x_n)^\top$.
We have
$$h(x)=\langle Lx,p/x\rangle=
\langle L\rho\theta,p/(\rho\theta)\rangle=
\langle L\theta,p/\theta\rangle=h(\theta).
$$
Therefore 
\begin{equation}\label{dU/dt}
\frac{dU}{dt}=\sum_{i=1}^np_ir_i(t/T)+
mh(\theta).
\end{equation}
Let $x(t)$ the solution of \eqref{eq3} with initial condition $x(0)=\pi(m,T)$, where $\pi(m,T)$ is the Perron vector of the monodromy matrix $X(T)$ of \eqref{eq3}.
Since $\rho(0)=1$, the corresponding solution of 
\eqref{eqzitheta0} has initial condition $\theta(0)=\pi(m,T)$. Hence, it is the periodic solution $\theta^*(t,m,T)$. Consider now
$U(t)=\sum_{i=1}^np_i\ln{x_i(t)}.$ 
We have
$$
{
\lim_{t\to\infty}\frac{U(t)}{t}=
\sum_{i=1}^np_i\lim_{t\to\infty}\frac{1}{t}\ln x_i(t)=\sum_{i=1}^np_i\Lambda[x_i]=\Lambda(m,T)\sum_{i=1}^np_i=\Lambda(m,T)}.$$
Since $U(t)$ is a solution of \eqref{dU/dt}, we have the following integral representation of 
$U(t)$
$$U(t)=U(0)+
\sum_{i=1}^np_i
\int_0^tr_i(s/T)ds
+m\int_0^th(\theta^*(s,m,T))ds.
$$ 
Therefore
$\Lambda(m,T)=
\lim_{t\to\infty}\frac{U(t)}{t}=
\sum_{i=1}^np_i\overline{r}_i+
m\int_0^1h(\theta^*(T\tau,m,T))d\tau.
$
\end{proof}

The formula \eqref{Lambdah(theta)0}  is given in \cite{Katriel} uniquely in the two-patch case, when the migration is  symmetric and the growth rates $r_1(\tau)$ and $r_2(\tau)$ are continuous, see \cite[Lemma 4]{Katriel}.

%%%%%%%%%%%%%%%%
\begin{theorem}\label{infimum} 
If the migration matrix is time independent then
$$\inf_{m,T}\Lambda(m,T)=\sum_{i=1}^np_i\overline{r}_i.$$   
\end{theorem}
\begin{proof}
In the time independent migration case, we have the formula \eqref{Lambdah(theta)0} for $\Lambda(m,T)$. If we prove that the function $h$ is non negative on the positive cone then, for all $m>0$ and $T>0$ 
$\Lambda(m,T)\geq \sum_{i=1}^np_i\overline{r}_i.$
Using \eqref{infLambda} we deduce then that the infimum of $\Lambda(m,T)$ is equal to 
$\sum_{i=1}^np_i\overline{r}_i$.

Let us prove that the function $h$ is non negative on the positive cone.
Observe that $h(x)=\langle Lx,p/x\rangle$, where $p/x=(p_1/x_1,\cdots,p_n/x_n)^\top$ and $\langle \cdot,\cdot\rangle$ is the usual Euclidean scalar product in $\mathbb{R}^n$.
Let $R$ denote the transpose of $L$. Then, 
for all $x \in \mathbb{R}^n,$
$$(R x)_i = \sum_j R_{ij} x_j = \sum_j R_{ij}(x_j - x_i).$$
Observe that for all $x \in \mathbb{R}^n$
\begin{equation}
\label{pinvariant}
\langle p, R x \rangle =\langle Lp, x \rangle= 0
\end{equation}
because $p$ is in the kernel of $L.$
By convexity, for all $x \in \mathbb{R}^n$ with positive entries,
$$\ln(x_j) - \ln(x_i) \leq \frac{x_j - x_i}{x_i}.$$ 
Thus
$${(R \ln(x))_i} 
\leq \sum_j R_{ij} \frac{x_j -x_i}{x_i} = \frac{(R x)_i}{x_i}.$$ That is
$R\ln(x) \leq \frac{R x}{x}$ componentwise, where  $\ln(x)$ stands for the  vector defined as $\ln(x)_i = \ln(x_i).$
Hence, using (\ref{pinvariant}),
$$\textstyle
0 = \langle p, R\ln(x) \rangle \leq \langle p, \frac{Rx}{x} \rangle.$$
Let $y = \frac{p}{x}$. We have
$\textstyle
h(y)=\langle Ly, \frac{p}{y} \rangle=\langle L\frac{p}{x}, x \rangle
=\langle \frac{p}{x}, Rx \rangle=\langle p, \frac{Rx}{x}\rangle\geq 0.$
 This proves that $h(y) \geq 0$ whenever $y$ has positive entries.
\end{proof}

\subsection{The monotonicity of $\Lambda(m,T)$ with respect to $T$}\label{Monotonie}
When the matrix migration is time independent and symmetric,  
{
it follows from \cite{liu} that the growth rate $\Lambda(m,T)$ is strictly increasing with respect to $T$, see Lemma 2 in \cite{Katriel} and the comments following this lemma}. When the migration is time dependent, the monotonicity is no longer true, see Sections \ref{NotIncr}, \ref{ExempleCL} and \ref{ExempleESL}.
We were not able to prove the monotonicity of the function $T\mapsto\Lambda(m,T)$ in the non symmetric time independent migration case. However, owing to the property $\Lambda(m,0)\leq\Lambda(m,T)$ given in  
\eqref{Lambda(m,0)<Lambda(m,T)} and the numerous numerical simulations we have done with constant non-symmetric matrices (see Sections \ref{TIM}), we  formulate the following conjecture.
%%%%%%%%
\begin{conj}\label{Conj1}
If the matrix migration is time independent, then
the function $T\mapsto \Lambda(m,T)$ is strictly increasing for all $m$ (except in the case where all $r_i(\tau)$ are equal).
\end{conj}

{It should be mentioned that this conjecture was already stated as an open question in the last section of \cite{liu}.}
If the migration matrix is time independent, $\chi>0$ and $\max_{1\leq i\leq n}\overline{r}_i<0$ then as shown in Remark \ref{Rem_mstar_fini}, the value $m^*$ for which $\Lambda(m,\infty)=0$ exists.  We have the following result.

\begin{propo}\label{PropT_c}
Assume that the migration matrix is time independent.
Assume that $\chi>0$ and $\max_{1\leq i\leq n}\overline{r}_i<0$. Let $m^*$ be the unique solution of $\Lambda(m,\infty)=0$. If Conjecture \ref{Conj1} is true (which thanks to \cite{Katriel,liu} is the case when the migration is symmetric),
there exist an analytic function $T_c:(0,m^*)\to (0,\infty)$ such that
$\lim_{m\to0}T_c(m)=\lim_{m\to m^*}T_c(m)=\infty$ and $\Lambda(m,T)>0$ if and only if $T>T_c(m)$. In other words the critical curve $T=T_c(m)$ separates the parameters plane $(m,T)$ in two regions : above it, {growth occurs}, below it, {decay occurs}.
\end{propo} 
%%%%%%%%%%%%%%%%%%
\begin{proof}
For all $m>0$ we have $\Lambda(m,0)<0$.  
If $m\geq m^*$, then $\Lambda(m,\infty)<0$ and by the monotonicity of $\Lambda(m,T)$ with respect to $T$, for all $T>0$ we have $\Lambda(m,T)<0$. If $m\in(0,m^*)$, then $\Lambda(m,\infty)>0$ and 
by the monotonicity of $\Lambda(m,T)$ with respect to $T$, there exists a unique value $T=T_c(m)$ such that $\Lambda(m,T)>0$ if $T>T_c(m)$
and $\Lambda(m,T)<0$ if $T<T_c(m)$.

Using Proposition \ref{Prop3}, the function
$\Lambda(m,T)$ is analytic in $m$ and $T$. The implicit function theorem implies that $T=T_c(m)$ is analytic in $m$ since it is the solution of equation $\Lambda(m,T)=0$.
\end{proof}

Examples where the critical curve $T=T_c(m)$ exists are provided in Section \ref{TIM}. The time independence of the migration matrix is not a necessary condition for the existence of the critical curve $T=T_c(m)$, see Sections \ref{NotIncr} and \ref{m_star_infini}.
Note that in the symmetric case we  have $T_c(m)=1/\nu_c(m)$, where $\nu_c(m)$ is the critical curve of Katriel, see \cite[Theorem 1(II)]{Katriel}.

\begin{rem}\label{DIGNew} For time independent migration, and if Conjecture \ref{Conj1} is true, as a consequence of Proposition \ref{PropT_c},  growth can occur only if $m<m^*$. In contrast, in the time dependent migration case,  then we can have $\Lambda(m,T)>0$ (growth) for some $m\in(m^*,\infty)$. Examples showing this behavior are provided in Sections \ref{ExempleCL} and \ref{ExempleESL}. 
\end{rem}

%%%%%%%%%%%%%%%%%%%%%%%
\subsection{Stochastic environment}\label{SE}
In this section, we briefly explain why most of our results remain valid if the growth rates are stochastic. More precisely, we consider a Markov Feller process $(\omega_t)_{t \geq 0}$ on a compact state $S$. {For precise definition, the reader is referred to \cite{BLSSarXiv}}. For each $1 \leq i \leq n$ we consider  a continuous function $r_i : S \to \mathbb{R}$. We also consider for each $s \in S$ a matrix $L(s) = (l_{ij}(s))_{ij}$ which satisfies \eqref{Lii} and we assume that $s \mapsto L(s)$ is continuous on $S$. We then have the system of differential equations
\begin{equation}\label{eq1:sto}
	\frac{dx_i}{dt}=r_i(\omega_t)x_i+m\sum_{j\neq i}\left(\ell_{ij}(\omega_t)x_j-\ell_{ji}(\omega_t)x_i\right),\qquad 1\leq i\leq n, 
\end{equation}
We assume that $(\omega_t)_{t \geq 0}$ has a unique stationary distribution $\mu$ on $S$. {This is a consequence of the fact that}, for all bounded continuous function $f: S \to \mathbb R$, with probability one,
\begin{equation}
    \label{eq:ergodic}
    \lim_{t \to \infty} \frac{1}{t} \int_0^t f(\omega_u) du = \int_S f(s) \mu(ds)
\end{equation}
In particular, in analogy with the periodic case, we let 
\[
\overline r_i = \int_S r_i(s) \mu(ds)
\]
be the local average growth rate in each patch in the absence of migration ($m=0$). Formula \eqref{eq:ergodic} implies that when $m=0$, for each $1 \leq i \leq n$, with probability one,
\[
\lim_{t \to \infty} \frac{1}{t} \ln(x_i(t)) = \lim_{t \to \infty} \frac{1}{t} \int_0^t r_i(\omega_u)du = \overline r_i.
\]
For $s\in S$, we let $R(s) = \mathrm{diag}(r_1(s), \ldots, r_n(s))$, $A(s) = R(s) + mL(s)$ and for a function $f$ defined on $S$ and with values in $\mathbb{R}$ or in the set of matrices, we let $\overline f = \int_S f(s) \mu(ds)$ . 
{Setting ${x(t)}=\left(
	x_1(t),\cdots,
	x_n(t)\right)^\top$, Equation \eqref{eq1:sto} can be rewritten as 
 \begin{equation*}
\frac{dx(t)}{dt}=A(\omega_t){x},
\end{equation*}}
By  Proposition 1 in
\cite{BLSSarXiv}, we have:
%%%%%%%%%%%%%%%
\begin{propo}
There exists $\Lambda \in \mathbb{R}$ such that, for all $x(0) > 0$, with probability one, 
\begin{equation}
\label{eq:LyapFeller}
\lim_{t \to \infty} \frac{\ln(\|x(t)\|)}{t} = \Lambda.
\end{equation}
\end{propo}
%%%%%%%%%%%%%%%%%%%

For all $T > 0$, we let $\omega_t^T = \omega_{t / T}$. We let $\Lambda(m,T)$ denote the Lyapunov exponent given by \eqref{eq:LyapFeller}, when $(\omega_t)_{t \geq 0}$ is replaced by $(\omega_t^T)_{t \geq 0}$ in \eqref{eq1:sto}. We can prove the following results:

\begin{theorem}
\label{thm:stochastic}
The Lyapunov exponent $\Lambda(m,T)$ satisfies the following properties
\begin{enumerate}
\item 
For all $m, T > 0$, 
\begin{equation}
\label{bounds}
 \Lambda(m,T) \leq \chi:= \int_S \max(r_i(s)) \mu(ds).
\end{equation}
\item 
For all $m > 0$, 
\begin{equation}
\label{limT0}
\lim_{T \to 0} \Lambda(m,T) = \lambda_{\max}( \overline{R + m L}).
\end{equation}
\item 
For all $m > 0$, 
\begin{equation}
\label{limTinfty}
\lim_{T \to \infty} \Lambda(m,T) = \int_S \lambda_{\max}( R(s) + mL(s)) \mu(ds).
\end{equation}
\item 
We have
\begin{equation}
\label{limTinftym0}
\lim_{m \to 0}\lim_{T \to \infty} \Lambda(m,T) = \chi
\end{equation}
In particular, $\sup_{m,T} \Lambda(m,T) = \chi$.
\item 
We have
\begin{equation}
\label{eq:limT0m0}
\lim_{m \to 0} \lim_{T \to 0} \Lambda(m,T) = \max(\overline r_i).
\end{equation}
\item
We have
\begin{equation}
\label{eq:limminfty}
\lim_{m \to \infty} \lim_{T \to \infty} \Lambda(m,T)  = \sum_i \overline{p_i  r_i}, \quad \lim_{m \to \infty} \lim_{T \to 0} \Lambda(m,T) = \sum_i q_i \overline{r_i}
\end{equation}
\end{enumerate}
\end{theorem}

\begin{proof}
The proof of the upper bound \eqref{bounds} is exactly the same as in the periodic case, except that we use \eqref{eq:ergodic} to justify the convergence of the temporal mean. The limits \eqref{limT0} 
and \eqref{limTinfty} are consequence, respectively of Propositions 4 and 5 in \cite{BLSSarXiv}. The double limits \eqref{limTinftym0}, \eqref{eq:limT0m0} and \eqref{eq:limminfty} are proven as in the periodic case, using Lemma \ref{lem:regularity-lambdamax}. 
\end{proof}

\begin{example}{\bf (Periodic case)}
The continuous periodic case corresponds to 
$S = \mathbb{R}/\mathbb{Z}$  identified with the unit circle and 
$\omega_t = s + t \, ({\sf mod} \, 1)$ for some $s \in S.$ 
\end{example}

\begin{example}{\bf (PDMP case)}
Let $S= \{1, \ldots, N\}$ a finite set, and $(\omega_t)_{ t \geq 0}$ a continuous time Markov chain on $S$. Then, $(\omega_t)_{ t \geq 0}$ is a Markov Feller process. The process $(x_t,\omega_t^T)_{t \geq 0}$ is a Piecewise Deterministic Markov Process (PDMP). The case where $N = 2$ and $n = 2$ has been investigated in \cite{benaim}. 
{Theorem \ref{thm:stochastic} extends the results found in \cite{benaim} to the
case of general $N$ and $n$}.
\end{example}

\begin{rem}
{The results given here all rely on results proved in \cite{BLSSarXiv}. In this paper, we have used the fact that the couple $(x_t, \omega_t)_{t \geq 0}$ is a Feller Markov process (see \cite[Lemma 7]{BLSSarXiv}). This is the reason why we assumed that  $s\to R(s)$ and $s\mapsto L(s)$ are continuous functions, in contrast with Section \ref{PE}, where these functions can have discontinuities. }
\end{rem}

{
The proofs of the asymptotic formulas for $\Lambda(m,T)$ when $T$ tends to 0 or $T$ tends to infinity are done in quite different ways, in the {piecewise continuous periodic case} (Theorem \ref{thm1}) and the random case (Theorem \ref{thm:stochastic}). In the random case, these formulas are special cases of the results given in \cite{BLSSarXiv} for a general irreducible cooperative linear system.

In the periodic case, these are general results that deal with an irreducible linear cooperative system, see Remark \ref{notmoregeneral}. 
\color{black}
Moreover, as for the periodic case, one could also give in the random case asymptotic formulas for $\Lambda(m,T)$ when $m$ tends to 0 or $m$ tends to infinity, which are similar to those given in items 3 and 4 of Theorem \ref{thm1}.}

%%%%%%%%%%%%%%%%%%%%%%%%
\section{Examples and numerical illustrations}\label{Numerical}

\subsection{Explicit formulas for the two-patch case}\label{2pc}

In the simplest two-patch case ($n=2$) the system \eqref{eq1} is
\begin{equation}\label{2eq1}
\begin{array}{l}
\frac{dx_1}{dt}= r_1(t/T)x_1 +m (\ell_{12}(t/T)x_2- \ell_{21}(t/T)x_1),\\[1mm]
\frac{dx_2}{dt}= r_2(t/T)x_2 +m (\ell_{21}(t/T)x_1- \ell_{12}(t/T)x_2).
\end{array}
\end{equation}
%%%%%%%%%%%%%%%%%%%%%%

\begin{propo}\label{Prop2pc}
We obtain the following explicit formulas for $\Lambda(m,0)$ and $\Lambda(m,\infty)$:
$$
\begin{array}{rcl}
\Lambda(m,0)&=&
\frac{1}{2}\left(\overline{r}_1+\overline{r}_2
+\sqrt{D(\overline{r}_1,\overline{r}_2,\overline{\ell}_{21},\overline{\ell}_{12})}\right)
-m\frac{\overline{\ell}_{12}+\overline{\ell}_{21}}{2},\\[2mm]
\Lambda(m,\infty)&=&
\frac{1}{2}\left(\overline{r}_1+\overline{r}_2
+\int_0^1\sqrt{D(r_1(\tau),
r_2(\tau),\ell_{21}(\tau),\ell_{12}(\tau))}d\tau\right)-m\frac{\overline{\ell}_{12}+\overline{\ell}_{21}}{2}.
\end{array}
$$
where 
\begin{equation}
    \label{Dr1r2}
    D(r_1,r_2,\ell_{21},\ell_{12})=(r_1-r_2+m(\ell_{12}-\ell_{21}))^2+4m^2\ell_{12}\ell_{21}
\end{equation}
When the migration is constant the growth rate of \eqref{2eq1} is given by
$$
\Lambda(m,T)=
\frac{\ell_{12}\overline{r}_1+\ell_{21}\overline{r}_2}{\ell_{12}+\ell_{21}}+
\frac{m}{T}\int_0^T
\frac{(\ell_{12}\theta_2^*(t,T)-\ell_{21}\theta_1^*(t,T))^2}{\theta_1^*(t,T)\theta_2^*(t,T)(\ell_{12}+\ell_{21})}dt,
$$
where $(\theta_1^*,\theta_2^*)$ is the $T$-periodic solution of the differential system equation
$$
\begin{array}{lcl}
\frac{d\theta_1}{dt}&=&
(r_1(t/T)-m\ell_{21}(t/T))\theta_1+ m\ell_{12}(t/T)\theta_2-(r_1(t/T)\theta_1+r_2(t/T)\theta_2)\theta_1,\\[2mm]
\frac{d\theta_2}{dt}&=&
(r_2(t/T)-m\ell_{12}(t/T))\theta_2+ m\ell_{21}(t/T)\theta_1-(r_1(t/T)\theta_1+r_2(t/T)\theta_2)\theta_2.
\end{array}
$$
Since $\theta_2=1-\theta_1$, this equation reduces to a scalar equation on $[0,1]$.
\end{propo}
\begin{proof}
The formulas for $\Lambda(m,0)$ and $\Lambda(m,\infty)$ follow from \eqref{T=0} and \eqref{T=infini}, respectively, and the fact that
the dominant eigenvalue $\lambda_{max}(R(\tau)+mL(\tau))$ is given by
$$\lambda_{max}(R(\tau)+mL)\!=\!\textstyle
\frac{1}{2}
\left(r_1(\tau)\!+\!r_2(\tau)\!+\!\sqrt{D(r_1(\tau),r_2(\tau),\ell_{21}(\tau),\ell_{12}(\tau))}\right)-m\frac{\ell_{12}(\tau)+\ell_{21}(\tau)}{2},$$
where $D$ is given by \eqref{Dr1r2}.
The formula for $\Lambda(m,T)$ follows from  \eqref{Lambdah(theta)0}.
\end{proof}

In the symmetric constant migration case  ($\ell_{12}(\tau)=\ell_{21}(\tau)=1$) we obtain the formulas 
$$
\begin{array}{rcl}
\Lambda(m,0)&=&
\frac{1}{2}\left(\overline{r}_1+\overline{r}_2
+\sqrt{(\overline{r}_1-\overline{r}_2)^2+4m^2}
\right)-m,\\
\Lambda(m,\infty)&=&
\frac{1}{2}\left(\overline{r}_1+\overline{r}_2
+\int_0^1\sqrt{({r}_1(\tau)-{r}_2(\tau))^2+4m^2}d\tau
\right)-m.
\end{array}
$$
These formulas were given by Katriel \cite{Katriel}, see the formula (12) and the formula preceding (17) in \cite{Katriel}.
In the symmetric case ($\ell_{12}=\ell_{21}=1$) we obtain the formula
$$
\Lambda(m,T)=
\frac{\overline{r}_1+\overline{r}_2}{2}+
m\left(\frac{1}{2T}\int_0^T
\left(\frac{\theta_2^*(t,T)}{\theta_1^*(t,T)}+
\frac{\theta_1^*(t,T)}{\theta_2^*(t,T)}\right)dt-1\right).
$$
Using, as in \cite{Katriel}, the variable $z=x_2/x_1=\theta_2/\theta_1$, we obtain 
$$
\Lambda(m,T)=
\frac{\overline{r}_1+\overline{r}_2}{2}+
m\left(\frac{1}{2T}\int_0^T
\left(z^*(t,T)+
\frac{1}{z^*(t,T)}\right)dt-1\right).
$$
which is the same formula as \cite[Formula (26)]{Katriel}.

\begin{example}
{\bf{The $\pm1$ model}}. \label{pm1model}
%%%%%%%%%%%%%%%%%%%%
This model corresponds to the two-patch case \eqref{2eq1}, with constant symmetric migration $\ell_{12}=\ell_{21}=1$ and piecewise constant growth rates 
\begin{equation}\label{r1r2poum1}
r_1(\tau)=\left\{\begin{array}{l}
a\mbox{ if }0\leq \tau< 1/2\\
b\mbox{ if }1/2\leq \tau< 1
\end{array}
\right.,\qquad
r_2(\tau)=\left\{\begin{array}{l}
b\mbox{ if }0\leq \tau< 1/2\\
a\mbox{ if }1/2\leq \tau< 1
\end{array}
\right.,
\end{equation}
with $a>0$ and $a+b<0$.
Therefore, $\overline{r}_1=\overline{r}_2=\frac{a+b}{2}<0$ and $\chi=a>0$. Hence, DIG occurs. Note that we have two identical sinks, that are in phase opposition. This model can be reduced to the form 
$a=1-\varepsilon$, $b=-1-\varepsilon$, with $0<\varepsilon<1$, see \cite[Remark 3]{benaim}.  
Using the theoretical formulas in Figure \ref{fig1} and Proposition \ref{Prop2pc}, we have 
$$
\begin{array}{l}
\Lambda(0,T)=\Lambda(m,0)=\Lambda(\infty,T)=-\varepsilon,\quad 
\Lambda(0,\infty)=1+\varepsilon,\\
\Lambda(m,\infty)=-\varepsilon+\sqrt{1+m^2}-m.
\end{array}
$$
\end{example}

These formulas for the $\pm1$ model where already obtained in \cite[Proposition 2.8]{benaim} by using explicit computation of $\Lambda(m,T)$. Indeed, the monodromy matrix is given by $\Phi(T)=e^{TB/2}e^{TA/2}$ where the matrices $A$ and $B$  are given by
$$
A=\left[
\begin{array}{cc}
1-\varepsilon-m&m\\
m&-1-\varepsilon-m
\end{array}
\right],
\quad
B=\left[
\begin{array}{cc}
-1-\varepsilon-m&m\\
m&1-\varepsilon-m
\end{array}
\right].
$$
We can compute explicitly the Perron root of the monodromy matrix, see \cite[Proposition 2.7]{benaim}. 
In this special case we gave in \cite{benaim} some properties that are not extended to the general case considered in the present work. We proved in particular
that the threshold of the dispersal rate at which {growth appears is
exponentially small with respect to the period $T$ for $T$ large}, see \cite[Proposition 2.9]{benaim}. For a more detailed discussion on this issue, see Section \ref{DIGnuT}.  

\subsection{The piecewise constant case}
We consider the system \eqref{eq1}, with piecewise constant 1-periodic growth rates given by 
\begin{equation}
\label{PC}
r_i(\tau)=\left\{\begin{array}{l}
a_i\mbox{ if }0\leq \tau< \alpha\\
b_i\mbox{ if }\alpha\leq \tau< 1
\end{array}
\right.,
\quad 1\leq i\leq n,
\end{equation}
where $\alpha\in(0,1)$ and $a_i$ and $b_i$,  are real numbers. Thus, during a time of duration $\alpha T$ the $i$th population grows with a rate $a_i$, then, during a time of duration $(1-\alpha)T$ the population grows with a rate $b_i$. 
We also assume that the migration terms $\ell_{ij}(\tau)$ are piecewise constant. For simplicity we assume the discontinuity arises at the same value of time as the growth rates \eqref{PC}:
\begin{equation}
\label{PCL}
\ell_{ij}(\tau)=\left\{\begin{array}{l}
h_{ij}\mbox{ if }0\leq \tau< \alpha\\
k_{ij}\mbox{ if }\alpha\leq \tau< 1\\
\end{array}
\right.,
\quad
1\leq i\neq j\leq n,
\end{equation}
where $h_{ij}$ and $k_{ij}$, $i\neq j$, are non negative real numbers such that the matrices $H=(h_{ij})$ and $K=(k_{ij})$, whose diagonal elements are defined as in \eqref{Lii}, are irreducible. This simplest case
is already of much interest, since it illustrates all behaviors depicted in the preceding section. 
The monodromy matrix is given by
\begin{equation}
\label{PhiExplicite}
\Phi(T)=e^{(1-\alpha){T}B}e^{\alpha{T}{A}},
\end{equation}
where the matrices 
$A={\rm diag}(a_i)+mH$ and $B={\rm diag}(b_i)+mK$, 
are time independent matrices.
Hence, we can compute the Perron root $\mu(m,T)$ of the matrix $\Phi(T)$ and use the formula \eqref{Lambda} to compute 
$\Lambda(m,T)=\frac{1}{T}\ln(\mu(m,T))$.

In \eqref{PC} and \eqref{PCL} we have only two discontinuities on each period of time. In Section \ref{ExempleCL} we will consider a case with three discontinuities.

{Note that in the two-patch case, a computer program like Maple is able to compute analytically the monodromy matrix \eqref{PhiExplicite} for any constant matrices $A$ and $B$, and then determine its Perron root, since this computation requires only the solution of a second degree algebraic equation. However, the formula obtained for $\Lambda(m,T)$ is so complicated that we cannot exploit it mathematically, as it was the case in the particular case of the $\pm1$ model considered in Example \ref{pm1model}. Nevertheless, if the values of the model parameters (i.e. the coefficients $a_i$, $b_i$, $h_{ij}$ and $k_{ij}$) are fixed, we can plot the graphs of the functions $(m, T)\mapsto \Lambda(m, T)$, see Figure \ref{fig2}(a), $T\mapsto \Lambda(m,T)$, for $m$ fixed, see Figure \ref{fig2}(d), $m\mapsto \Lambda(m,T)$, for $T$ fixed, see Figure \ref{fig2}(c), and also the critical set where $\Lambda(m,T)=0$, see Figure \ref{fig2}(b). The aim of this section is to consider various examples and draw the corresponding graphs in order to illustrate our main results, to support Conjecture \ref{Conj1}, to show that this conjecture is not true in the case of time-dependent migration, and to show the new features that appear in this case.}

 %%%%%%%%%%%%%%%%
\begin{figure}[ht]
\begin{center}
\includegraphics[width=10cm,
viewport=160 350 440 700]{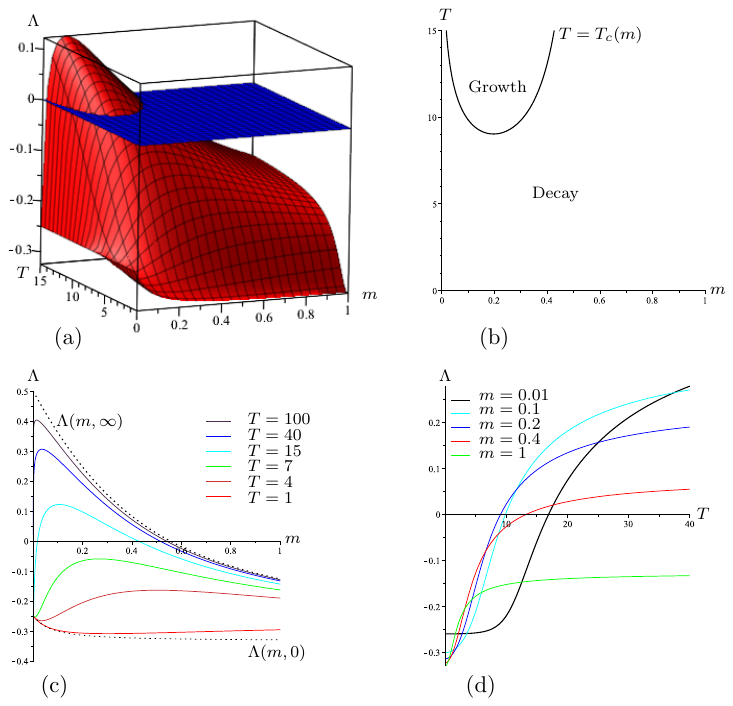}
\caption{(a) The graph of $(m,T)\mapsto \Lambda(m,T)$.  (b) The set $\Lambda(m,T)=0$. (c) Graphs of $m\mapsto \Lambda(m,T)$ with the indicated values of $T$. (d) Graphs of $T\mapsto \Lambda(m,T)$ with the indicated values of $m$. Here we used the two patch model corresponding to the matrices \eqref{AB1} and $\alpha=0.5$.}	\label{fig2}	
\end{center}
\end{figure}
%%%%%%%%%%%%%%%%%%%%

%%%%%%%%%%%%%%%%%%%%%%%%%%%
\begin{table}
\caption{Limits of $\Lambda(m,T)$ for the parameters values used in Figure \ref{fig2}}\label{Ex1}
\begin{center}
\begin{tabular}{l}
\hline
$\Lambda(0,\infty)=1/2$,\quad
$\Lambda(0,0)=\Lambda(0,T)=-1/4$\\ 
$\Lambda(\infty,0)=\Lambda(\infty,T)=
\Lambda(\infty,\infty)=-1/3$
\\
\hline
$\Lambda(m,0)=-\frac{3}{8}-\frac{3}{2}m+\frac{1}{8}\sqrt{1+8m+144m^2}$\\[1mm]
$\Lambda(m,\infty)=-\frac{3}{8}-\frac{3}{2}m+\frac{1}{4}\sqrt{4+4m+9m^2}+\frac{1}{8}\sqrt{9-12m+36m^2}$\\[1mm]
$\Lambda(m,\infty)=0$ for $m=m^*=5/9$.\\%[1mm]
\hline
\end{tabular}
\end{center}
\end{table}
%%%%%%%%%%%%%%%%%%%%%%%%

\subsection{Time-independent asymmetric migration} 
\label{TIM} 
Our objective in this section is to show numerical simulations with two and three patches that illustrate all our findings and also corroborate  Conjecture \ref{Conj1}.
\color{black}

We consider the two-patch case with time-independent migration given by $\ell_{12}=2$ and $\ell_{21}=1$ and growth rates
\begin{equation}
\label{r1r2}
r_1(\tau)=\left\{\begin{array}{rcl}
1/2&\mbox{if}&0\leq \tau< 1/2\\
-1&\mbox{if}&1/2\leq \tau< 1
\end{array}
\right.,\qquad
r_2(\tau)=\left\{\begin{array}{rcl}
-3/2&\mbox{if}&0\leq \tau< 1/2\\
1/2&\mbox{if}&1/2\leq \tau< 1
\end{array}
\right..
\end{equation}
Hence $\alpha=1/2$ and the matrices 
$A$ and $B$ in \eqref{PhiExplicite} are given by
\begin{equation}\label{AB1}
A=\left[
\begin{array}{cc}
1/2-m&2m\\
m&-3/2-2m
\end{array}
\right],
\quad
B=\left[
\begin{array}{cc}
-1-m&2m\\
m&1/2-2m
\end{array}
\right].
\end{equation}
We have $\chi=1/2$, 
$\overline{r}_1=-1/4$ and 
 $\overline{r}_2=-1/2$. 
Therefore, the patches are sinks and DIG occurs. 
Using Remark \ref{delta}, we 
have $p_1=2/3$ and $p_2=1/3$.
Using the theoretical formulas in Figure \ref{fig1} and Proposition \ref{Prop2pc}, we obtain the expressions shown in Table \ref{Ex1}. We show in Figure \ref{fig2} several plots of $\Lambda(m,T)$ (obtained by the Maple software).

\noindent
{\bf  Comments on Figure \ref{fig2}}. The strictly decreasing functions $m\mapsto \Lambda(m,0)$ and 
$m\mapsto \Lambda(m,\infty)$, are depicted in dotted line on panel (c) of the figure. 
Panel (d) of the figure shows that 
for all $m>0$, the functions $T\mapsto\Lambda(m,T)$ are strictly increasing, supporting Conjecture \ref{Conj1}. Hence, there exists a critical curve $T=T_c(m)$ defined for $0<m<m^*$ such that $T_c(0)=T_c(m^*)=\infty$ and DIG occurs if and only if $T>T_c(m)$, as depicted in panel (b) of the figure. Panel (c) of the Figure shows the graphs of functions $m\mapsto\Lambda(m,T)$ and illustrates their convergence 
toward $\Lambda(m,0)$ and $\Lambda(m,\infty)$ as $T$ tends to 0 and $\infty$, respectively.
Notice that for $0<T<\infty$, the functions $m\mapsto\Lambda(m,T)$ are not monotonic.

In the supplementary material Section S1, we give numerical simulations in a three-patch case showing that the function $T\mapsto \Lambda(m,T)$ is increasing, and supporting Conjecture \ref{Conj1}.

\subsection{Time-dependent irreducible migration} \label{NCM}
Our objective in this section is to show numerical simulations with two and three patches that illustrate all our findings and show that  Conjecture \ref{Conj1} is not true when the migration is not constant.
\color{black}

\subsubsection{The function $T\mapsto \Lambda(m,T)$ is not always increasing}\label{NotIncr}
In the supplementary material Section S2.1, we consider the two-patch case with time-dependent migration 
$$
\ell_{12}(\tau)=\left\{\begin{array}{rcl}
1&\mbox{if}&0\leq \tau< 1/2\\
2&\mbox{if}&1/2\leq \tau< 1
\end{array}
\right.,\qquad
\ell_{21}(\tau)=\left\{\begin{array}{rcl}
2&\mbox{if}&0\leq \tau< 1/2\\
1&\mbox{if}&1/2\leq \tau< 1
\end{array}
\right..
$$
The growth rates are given by \eqref{r1r2}. Hence, $\alpha=1/2$ and  
the matrices $A$ and $B$ in \eqref{PhiExplicite} are given by 
\begin{equation}\label{AB2S}
A=\left[
\begin{array}{cc}
1/2-2m&m\\
2m&-3/2-m
\end{array}
\right],
\quad
B=\left[
\begin{array}{cc}
-1-m&2m\\
m&1/2-2m
\end{array}
\right].
\end{equation}
We have
$\chi=1/2$, $\overline{r}_1=-1/4$ and $\overline{r}_2=-1/2$.
Therefore, the patches are sinks and DIG occurs.
Note that $\Lambda(\infty,T)=-2/3$ and $\Lambda(\infty,0)=-3/8$, as depicted in the in the supplementary material Table S1. Therefore $\Lambda(\infty,T)<\Lambda(\infty,0)$ and,
for $m$ large enough, the property \eqref{Lambda(m,0)<Lambda(m,T)}, which is true in the case of time independent migration, is not satisfied. Therefore,  Conjecture \ref{Conj1} is not true in general for time dependent migration. 

The plot of the graphs of the functions $(m, T)\mapsto \Lambda(m, T)$, $T\mapsto \Lambda(m,T)$, for $m$ fixed,  $m\mapsto \Lambda(m,T)$, for $T$ fixed, and also the critical set where $\Lambda(m,T)=0$ are shown in the in the supplementary material Figure S2 . 
We can make the same comments on this figure, as those made in the previous section on Figure \ref{fig2}, except that, in contrast with Figure \ref{fig2}(c), for $m$ large enough, we have $\Lambda(m,T)<\Lambda(m,0)$ and the function $T\mapsto \Lambda(m,T)$ is decreasing instead of increasing.

\subsubsection{Growth can occur for all $m>0$} \label{m_star_infini}
In the supplementary material Section S2.2, we consider the two-patch case with time-dependent migration 
$$
\ell_{12}(\tau)=\left\{\begin{array}{rcl}
5&\mbox{if}&0\leq \tau< 1/2\\
1&\mbox{if}&1/2\leq \tau< 1
\end{array}
\right.,\qquad
\ell_{21}(\tau)=\left\{\begin{array}{rcl}
1&\mbox{if}&0\leq \tau< 1/2\\
5&\mbox{if}&1/2\leq \tau< 1
\end{array}
\right..
$$
The growth rates are given by \eqref{r1r2}. Hence, $\alpha=1/2$ and
the matrices $A$ and $B$ in \eqref{PhiExplicite} are given by
\begin{equation}\label{ABm_star_infini}
A=\left[
\begin{array}{cc}
1/2-m&5m\\
m&-3/2-5m
\end{array}
\right],
\quad
B=\left[
\begin{array}{cc}
-1-5m&m\\
5m&1/2-m
\end{array}
\right].
\end{equation}
We have
$\chi=1/2$, $\overline{r}_1=-1/4$ and $\overline{r}_2=-1/2$.
Therefore, the patches are sinks and DIG occurs.

The plot of the graphs of the functions $(m, T)\mapsto \Lambda(m, T)$, $T\mapsto \Lambda(m,T)$, for $m$ fixed,  $m\mapsto \Lambda(m,T)$, for $T$ fixed, and also the critical set where $\Lambda(m,T)=0$ are shown in the supplementary material Figure S3. 
We can make the same comments on this figure, as those made in the previous section on Figure \ref{fig2}, except that, in contrast with Figure \ref{fig2}(c), the critical curve $T=T_c(m)$ is defined for all $m>0$. Indeed, $\Lambda(\infty,T)>0$, for any $m>0$, as depicted in the in the supplementary material Table S2.

\begin{figure}[ht]
\begin{center}
\includegraphics[width=10cm,
viewport=160 580 440 710]{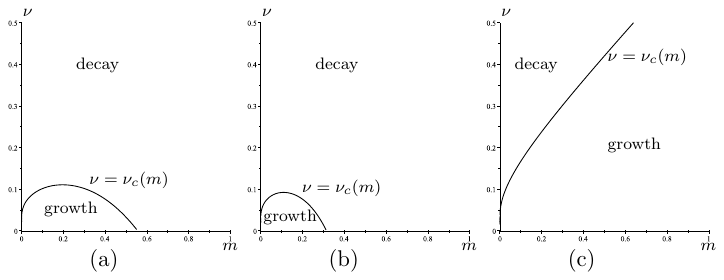}
\caption{The set where $\Lambda(m,1/\nu)>0$ in the $(m,\nu)$ parameter-plane. (a) Parameters values of \eqref{AB1}. (b) Parameters values of \eqref{AB2S}. (c) Parameters values of \eqref{ABm_star_infini}}	\label{fig5}	
\end{center}
\end{figure}

%%%%%%%%%%%%%%%%%%%%
\subsubsection{The set where growth occurs in the $(m,\nu)$ parameter-plane}\label{DIGnuT}

The difference between the model \eqref{AB1}, considered in Section, \ref{TIM}, \eqref{AB2} considered in Section  \ref{NotIncr} and \eqref{ABm_star_infini} considered in Section \ref{m_star_infini} is in the migration matrix which is assumed to be independent of time in the first model whereas it depends on it in the two following ones.
In Figure \ref{fig5} we display the set where growth occurs in the $(m,\nu)$ parameter-plane, where $\nu=1/T$ is the frequency. The figure shows that the critical curve $\nu=\nu_c(m)$, where $\nu_c(m):=1/T_c(m)$, is tangent to the $\nu$-axis at the origin, i.e. $\lim_{m\to0}\nu'_c(m)=-\infty$. This property was already numerically observed in the symmetric migration case by Katriel \cite{Katriel}. This  property was established in \cite{benaim}, for the $\pm1$ model considered in Example \ref{pm1model}. Using the explicit expression of $\Lambda(m,T)$ we showed that when $\nu\to 0$, the threshold
$m^*(\nu)=\inf_{\nu>0}\{m:\Lambda(m,1/\nu)>0\}$
at which growth occurs is of order 
$e^{-(1-\varepsilon)/\nu}$, see \cite[Proposition 2.9]{benaim}. Therefore the critical curve
$\nu=\nu_c(m)$ has asymptotic behavior of the form 
$m\sim e^{-k/\nu}$, that is $m$ becomes
exponentially small in $1/\nu$ near the origin.

Note that in \eqref{AB2} the migration is always stronger towards the most unfavorable patch. As expected, and as illustrated in Figure \ref{fig5}(b), the region of the $(m,\nu)$ for which growth occurs is narrowed, but it still remains present. 
In \eqref{ABm_star_infini} the migration is always stronger towards the most favorable patch. As expected, and as illustrated in Figure \ref{fig5}(c), the region of the $(m,\nu)$ for which growth occurs is bigger. In this last case, growth can occur for all $m>0$.

\begin{figure}[ht]
\begin{center}
\includegraphics[width=10cm,
viewport=160 540 440 740]{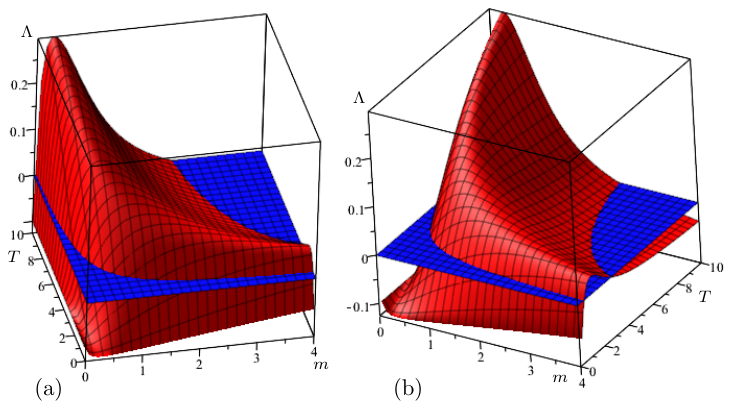}
\caption{The graph of $(m,T)\mapsto \Lambda(m,T)$ corresponding to the matrices \eqref{ABCmatrix}, seen from left (a) and right (b), showing the non monotonicity of $T\mapsto \Lambda(m,T)$.}	\label{fig7}	
\end{center}
\end{figure}

\subsubsection{Growth can also occur for $m>m^*$: a two-patch example}
\label{ExempleCL}
In this section we consider the following example of two-patches \eqref{2eq1}, with piecewise constant 1-periodic growth rates and migration terms, having three discontinuities on each period of time, and given by 
\begin{equation}
\label{PCCL}
r_1(\tau)=\left\{\begin{array}{r}
0\mbox{ if }0\leq \tau< \frac{1}{3}\\[1mm]
-\frac{4}{5}\mbox{ if }\frac{1}{3}\leq \tau<\frac{2}{3}\\[1mm]
\frac{1}{2}\mbox{ if }\frac{2}{3}\leq \tau<1
\end{array}
\right.,
\quad
r_2(\tau)=\left\{\begin{array}{r}
-\frac{1}{10}\mbox{ if }0\leq \tau< \frac{1}{3}\\[1mm]
\frac{3}{2}\mbox{ if }\frac{1}{3}\leq \tau<\frac{2}{3}\\[1mm]
-2\mbox{ if }\frac{2}{3}\leq \tau<1
\end{array}\right.,
\end{equation}
\begin{equation}
\label{PCMCL}
\ell_{12}(\tau)=\left\{\begin{array}{r}
\frac{1}{10}\mbox{ if }0\leq \tau< \frac{1}{3}\\[1mm]
2\mbox{ if }\frac{1}{3}\leq \tau<\frac{2}{3}\\[1mm]
\frac{1}{100}\mbox{ if }\frac{2}{3}\leq \tau<1
\end{array}
\right.,
\quad
\ell_{21}(\tau)=\left\{\begin{array}{r}
1\mbox{ if }0\leq \tau< \frac{1}{3}\\[1mm]
\frac{1}{5}\mbox{ if }\frac{1}{3}\leq \tau<\frac{2}{3}\\[1mm]
\frac{1}{100}\mbox{ if }\frac{2}{3}\leq \tau<1
\end{array}\right..
\end{equation}
The monodromy matrix is given by
\begin{equation}
\label{PhiCL}
\Phi(T)=e^{\frac{T}{3}C}e^{\frac{T}{3}B}e^{\frac{T}{3}A},
\end{equation}
where the matrices $A$, $B$ and $C$ are defined by
\begin{equation}
\label{ABCmatrix}
\begin{array}{c}
A\!=\!\left[
\begin{array}{rr}
-m&\frac{m}{10}\\
m&-\frac{1}{10}-\frac{m}{10}
\end{array}
\right],
~
B\!=\!\left[
\begin{array}{rr}
-\frac{4}{5}-\frac{m}{5}&{2m}\\[1mm]
\frac{m}{5}&\frac{3}{2}-2m
\end{array}
\right],
~
C\!=\!\left[
\begin{array}{rr}
\frac{1}{2}-\frac{m}{100}&\frac{m}{100}\\
\frac{m}{100}&-2-\frac{m}{100}
\end{array}
\right].
\end{array}
\end{equation}
%%%%%%%%%%%%%%%%%%%%%%%%
We have
$\chi=\frac{2}{3}$, $\overline{r}_1=-\frac{1}{10}$ and 
 $\overline{r}_2=-\frac{1}{5}$. Therefore, the patches are sinks and DIG occurs.

%%%%%%%%%%%%%%%%%%
\begin{figure}[ht]
\begin{center}
\includegraphics[width=10cm,
viewport=160 510 440 670]{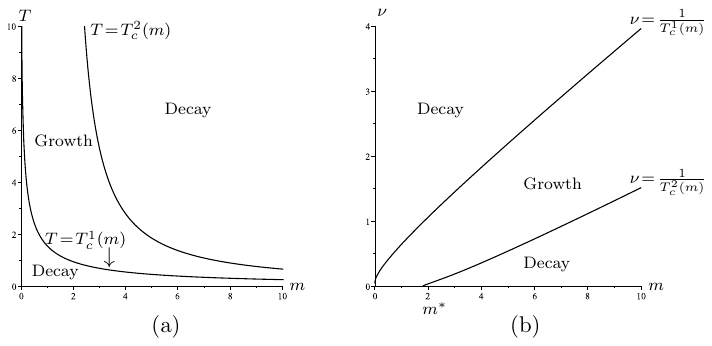}
\caption{(a) The set $\Lambda(m,T)=0$. (b)  The set $\Lambda(m,\nu)=0$. Here we use  the parameter values of \eqref{ABCmatrix} ($m^*=1.764$)}	\label{fig9}	
\end{center}
\end{figure} 
%%%%%%%%%%%%%%%%%%
Note that $\Lambda(\infty,T)<\Lambda(\infty,0)$, as depicted in the supplementary material Table S3. Therefore, 
for $m$ large enough, the condition 
$\Lambda(m,T)>\Lambda(m,0)$, proved in \eqref{Lambda(m,0)<Lambda(m,T)} for time-independent migration, is not satisfied in this model with time-dependent migration.

We can compute the Perron root $\mu(m,T)$ of the matrix $\Phi(T)$ and plot the graph of  
$\Lambda(m,T)=\frac{1}{T}\ln(\mu(m,T))$, which is shown in Figure \ref{fig7}. For $m$ fixed, the function $T\mapsto \Lambda(m,T)$ can be increasing then decreasing. This behavior is more explicitly depicted in the in the supplementary material Figure S4. This figure shows the plot of the graphs of the  functions $T\mapsto \Lambda(m,T)$, for $m$ fixed and  also the graphs of the functions $m\mapsto \Lambda(m,T)$, for $T$ fixed. Hence we do not have a critical curve  $T=T_c(m)$, defined for $0<m<m^*$, such that growth occurs if and only if $T>T_c(m)$. 

The set of parameter values where {growth occurs} behaves as in Remark \ref{DIGNew}.
Indeed, there exist functions $T_c^1:(0,\infty)\to(0,\infty)$ and 
$T_c^2:(m^*,\infty)\to(0,\infty)$ such that 
growth occurs if and only if
$T_c^1(m)<T<T_c^2(m)$, see Figure \ref{fig9}(a). 
Therefore, DIG can occur for $m>m^*$. In Figure \ref{fig9}(b) we display the set where growth occurs in the $(m,\nu)$, where $\nu=1/T$. We observe that the critical curve $\nu=1/T_c^1(m)$ is tangent to the $\nu$-axis at the origin.

\subsubsection{Growth can also occur for $m>m^*$: a three-patch example}\label{ExempleESL}
Our objective in this section is to show numerical simulations with three patches that illustrate the non monotonicity of $\Lambda(m,T)$ with respect to $T$.
We consider the case where $\alpha=1/2$ and the matrices $A$ and $B$ in \eqref{PhiExplicite} are given by:   
\begin{equation}\label{AESL}
\begin{array}{c}
A_\varepsilon(m)=\left[
\begin{array}{ccc}
9-(10+\varepsilon_1) m&\varepsilon_2m&\varepsilon_3 m\\
10 m&-1-(\varepsilon_2+\varepsilon_4) m&\varepsilon_5 m\\
\varepsilon_1m&\varepsilon_4 m&-10-(\varepsilon_3+\varepsilon_5) m
\end{array}
\right],
\end{array}
\end{equation}
\begin{equation}\label{BESL}
\begin{array}{c}
B_\delta(m)=\left[
\begin{array}{ccc}
-10-(\delta_1+\delta_4)m&\delta_2m&10 m\\
\delta_1 m&-(10+\delta_2) m&\delta_3m\\
\delta_4m&10 m&9-(10+\delta_3) m
\end{array}
\right],
\end{array}
\end{equation}
with $\varepsilon=(\varepsilon_1,\varepsilon_2,\varepsilon_3,\varepsilon_4,\varepsilon_5)\geq 0$ and 
$\delta=(\delta_1,\delta_2,\delta_3,\delta_4)\geq 0$ such that the corresponding migration matrices  are irreducible.
Our theory applies to this example.
For these parameter values, we have
$\chi=9$ and $\overline{r}_1=\overline{r}_2= \overline{r}_3=-1/2$. Therefore all patches are sinks and DIG occurs. We have the following result

\begin{propo} \label{PropES} 
Let 
$\Lambda_{\varepsilon,\delta}(m,T)$ the growth rate corresponding to the piecewise constant system defined by the matrices \eqref{AESL} and \eqref{BESL}. For $\varepsilon$ and $\delta$ small enough, we have 
$$\Lambda_{\varepsilon,\delta}(1,0)<0,
\quad 
\Lambda_{\varepsilon,\delta}(1,\infty)<0,
\quad
\Lambda_{\varepsilon,\delta}(1,2)>0,$$
so that the function $T\mapsto \Lambda_{\varepsilon,\delta}(1,T)$ is not monotonous.
\end{propo}
%%%%%%%%%%
\begin{proof}
For $m=1$, $\epsilon=0$ and $\delta=0$ we get $A_0(1)=A$ and $B_0(1)=B$, where $A$ and $B$ are given by
$$
\begin{array}{l}
A=\left[
\begin{array}{ccc}
-1 &0&0\\
10&-1&0\\
0&0&-10
\end{array}
\right],
\quad
B=\left[
\begin{array}{ccc}
-10&0&10\\
0&-10&0\\
0&10&-1
\end{array}
\right].
\end{array}
$$
These matrices have been proposed in \cite{Fainshil} as a counterexample to the conjecture that a PLS (\emph{positive linear switched system}) is GUAS (\emph{globally uniformly asymptotically
stable}) if every matrix in the convex hull of the matrices defining the subsystems of the PLS is Hurwitz, 
{i.e. its spectral abscissa is negative}. 
Indeed, it is proved in \cite{Fainshil} that every matrix in 
${\rm co}(A,B)=\{kA+(1-k)B:k\in[0,1]\}$
is Hurwitz and a calculation reveals that the matrix
$e^Ae^B$ has one real eigenvalue $\mu\approx  1.669>1$. Thus the PLS defined by the matrices $A$ and $B$ is not GUAS.
From these observations we deduce that 
\begin{align*}
&\lambda_{max}(A)<0,
\quad 
\lambda_{max}(B)<0,
\quad 
\textstyle
\lambda_{max}\left(\frac{A+B}{2}\right)<0,
\\
&\mbox{and the  Perron root of } e^Ae^B 
\mbox{ is stricly greater than 1}.
\end{align*}
 Using the continuity of the spectral abscissa and the continuity of the Perron root we deduce that
for $\varepsilon$ and $\delta$ small enough, we have
\begin{align*}
&\lambda_{max}\left(A_\varepsilon(1)\right)<0,
\quad 
\lambda_{max}\left(B_\delta(1)\right)<0,
\quad 
\textstyle
\lambda_{max}\left(\frac{A_\varepsilon(1)+B_\delta(1)}{2}\right)<0,
\\
&\mbox{and the  Perron root of } e^{A_\varepsilon(1)}e^{B_\delta(1)} 
\mbox{ is stricly greater than 1}.
\end{align*}
 Therefore, using \eqref{T=0} and \eqref{T=infini} we have 
\begin{align*}
&\Lambda_{\varepsilon,\delta}(1,0)=\textstyle
\lambda_{max}\left(\frac{A_\varepsilon(1)+B_\delta(1)}{2}\right)<0,\\
&
\Lambda_{\varepsilon,\delta}(1,\infty)=\textstyle
\frac{1}{2}\left(\lambda_{max}\left(A_\varepsilon(1)\right)+
\lambda_{max}\left(B_\delta(1)\right)\right)<0.
\end{align*}
Let $\mu$ be the Perron root of $e^{A_\varepsilon(1)}e^{B_\delta(1)}$. Using the definition \eqref{Lambda} of $\Lambda(m,T)$, we have
$\Lambda_{\varepsilon,\delta}(1,2)
=\frac{1}{2}\ln\left(\mu\right)>0.$
%%%%%%%%%%%%%%%%%%%%%%%
\end{proof}

%%%%%%%%%%%%%%%%%%
\begin{figure}[ht]
\begin{center}
\includegraphics[width=10cm,
viewport=160 505 440 660]{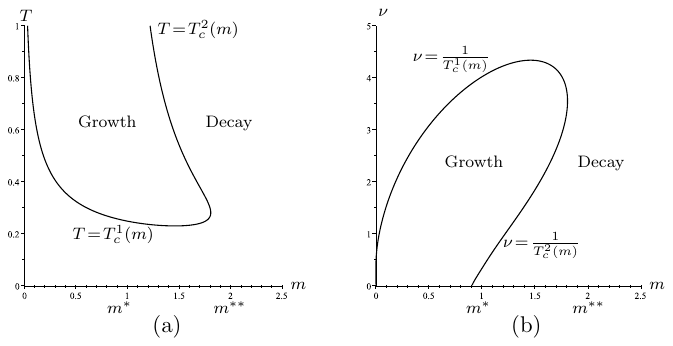}
\caption{(a) The set $\Lambda(m,T)=0$. (b)  The set $\Lambda(m,\nu)=0$. The parameter values are those of \eqref{ABESL0}. We have $m^*\approx0.904$ and $m^{**}\approx1.807$.}	\label{fig11}	
\end{center}
\end{figure} 
%%%%%%%%%%%%%%%%%%

We show in the supplementary material Figure S5 the plot of $\Lambda_{\varepsilon,\delta}(m,T)$ for the particular choice of $\varepsilon$ and $\delta$ given by  
$$ \varepsilon_3=\varepsilon_4=0.1,\quad\varepsilon_1=\varepsilon_2=\varepsilon_5=0,\quad \delta_1=0.1,\quad\delta_2=\delta_3=\delta_4=0.$$
For this choice of $\varepsilon$ and $\delta$ the matrices $A$ and $B$ are written
\begin{equation}\label{ABESL0}
\begin{array}{c}
A=\left[
\begin{array}{ccc}
9-10 m&0&\frac{m}{10}\\
10 m&-1-\frac{m}{10}& 0\\
0&\frac{m}{10}&-10-\frac{m}{10}
\end{array}
\right],
\quad
B=\left[
\begin{array}{ccc}
-10-\frac{m}{10}&0&10 m\\
\frac{m}{10}&-10 m&0\\
0&10 m&9-10 m
\end{array}
\right].
\end{array}
\end{equation}
Therefore, the
migration matrix is irreducible and corresponds to the circular migrations
$$
\left\{
\begin{array}{l}
1 \stackrel{10}{\longrightarrow} 2
\stackrel{0.1}{\longrightarrow} 3
\stackrel{0.1}{\longrightarrow} 1,\qquad\mbox{for }\tau\in[0,1/2),\\
1 \stackrel{0.1}{\longrightarrow} 2
\stackrel{10}{\longrightarrow} 3
\stackrel{10}{\longrightarrow} 1,\qquad\mbox{for }\tau\in[1/2,1).
\end{array}
\right. 
$$
As shown in the supplementary material Table S3, we have $\Lambda(\infty,T)<\Lambda(\infty,0)$ and hence, for $m$ large enough we should have $\Lambda(m,T)<\Lambda(m,0)$, so that  $\Lambda(m,T)$ is not increasing with respect to $T$. Actually, the behavior predicted by Proposition \ref{PropES} occurs in this case since for $m=1$ the map $T\mapsto\Lambda(m,T)$ is increasing and then decreasing, 
see the supplementary material Figure S4(d).

The set of parameter values where growth occurs behaves as in Remark \ref{DIGNew}: there exists $m^{**}>m^*$ and functions $T_c^1:(0,m^{**})\to(0,\infty)$ and 
$T_c^2:(m^*,m^{**})\to(0,\infty)$ such that
$T_c^1(m^{**})=T_c^2(m^{**})$ and 
growth occurs if and only if
$T_c^1(m)<T<T_c^2(m)$, see Figure  \ref{fig11}(a). 
Therefore, growth can occur for $m>m^*$. In Figure \ref{fig11}(b) we display the set where growth occurs in the $(m,\nu)$, where $\nu=1/T$. We observe that the critical curve $\nu=1/T_c^1(m)$ is tangent to the $\nu$-axis at the origin.

\subsection{Time-dependent reducible migration}\label{INNsec}
The aim of this section is to relax Hypothesis \ref{H2}, i.e. the assumption that the migration matrix $L(\tau)$ is irreducible
for any $\tau\in[0,1]$. We have the following result.

\begin{propo}\label{Prop12}
Assume that Hypothesis \ref{H1} is satisfied and the monodromy matrix $\Phi(T)$ is positive. Suppose $m>0$ and $T>0$. Then \eqref{eq3} has a growth rate $\Lambda(m,T)$ which is given by 
$\Lambda(m,T)=\frac{1}{T}\ln(\mu(m,T)),$
where $\mu(m,T)$ is the Perron root of the monodromy matrix $\Phi(T)$. Moreover, for all $m>0$ and $T>0$ we have $\Lambda(m,T)\leq \chi$ where $\chi$ is defined by \eqref{chi} and for all $T>0$ we have $\lim_{m\to 0}\Lambda(m,T)=\max_i\overline{r}_i$.
\end{propo}

\begin{proof}
The result on the existence of the growth rate (Proposition \ref{Prop3}) only uses the fact that the monodromy matrix  $\Phi(T)$ is positive. The bound of $\Lambda(m,T)$ by $\chi$ (Theorem \ref{upperborneLambda}) does not use any additional assumption. The limit \eqref{m=0} in Theorem \ref{thm1} follows from the continuous dependence of the solutions in the $m$ parameter. 
\end{proof}

Now we give examples showing that the DIG phenomenon occurs with $\chi>0$ as well as examples showing that the DIG phenomenon does not occur, even if $\chi>0$.

\begin{figure}[ht]
\begin{center}
\includegraphics[width=10cm,
viewport=155 550 425 740]{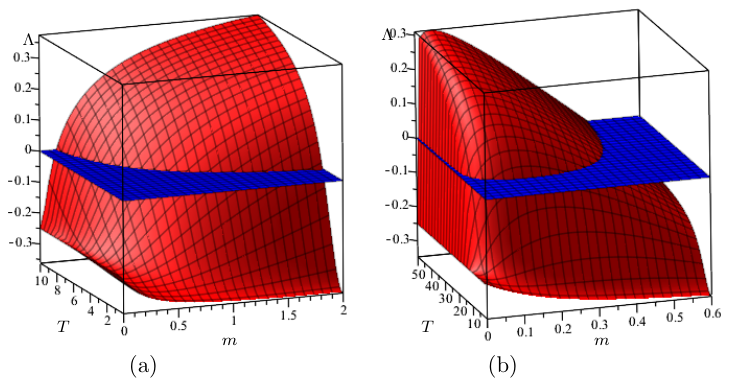}
\caption{(a) Unidirectional migration toward the favorable patch. (b) Unidirectional migration toward the unfavorable patch.
\label{fig17}}
\end{center}
\end{figure} 

\subsubsection{Unidirectional migration in a two patch case}\label{UDMig}
Let us consider again the example of two patches studied in Sections \ref{TIM} and \ref{NCM}, whose growth rates are given by \eqref{r1r2}. We have $\chi=1/2$, 
$\overline{r}_1=-1/4$ and 
 $\overline{r}_2=-1/2$. 
Therefore, the patches are sinks and, according to Theorem  \ref{DIGthm}, 
if the migration is in both directions, DIG occurs. Our aim is to consider a unidirectional migration and see if the DIG phenomenon continues to occur. 
First, let us consider the migration, defined by the matrix
$$L(\tau)=\left[\begin{array}{cc}
0&1\\
0&-1
\end{array}\right] 
\mbox{ for }\tau\in[0,1/2)\quad\mbox{ and }\quad  
L(\tau)=\left[\begin{array}{cc}
-1&0\\
1&0
\end{array}\right]
\mbox{ for }\tau\in[1/2,1).$$ 
%%%%%%%%%%%%
The monodromy matrix is $\Phi(T)=e^{TA_2/2}e^{TA_1/2}$, where $A_1$ and $A_2$ are the matrices
$$
A_1=\left[
\begin{array}{cc}
1/2&m\\
0&-3/2-m
\end{array}
\right],
\quad
A_2=\left[
\begin{array}{cc}
-1-m&0\\
m&1/2
\end{array}
\right].
$$ 
The matrix $\Phi(T)$ is positive. From Proposition \ref{Prop12} we deduce that the growth rate exists and is given by  $\Lambda(m,T)=\frac{1}{T}\ln(\mu(m,T)),$
where $\mu(m,T)$ is the Perron root of the monodromy matrix $\Phi(T)$. Figure \ref{fig17}(a) shows the graph of $\Lambda(m,T)$: we see that DIG occurs.

In the previous example the migration is unidirectional from the patch where the local growth rate $r_i(\tau)$ is negative toward the patch where it is positive. Let us consider the opposite situation where the migration is unidirectional toward the patch where the local growth rate $r_i(\tau)$ is negative. For this purpose we use the migration matrix
$$L(\tau)=\left[\begin{array}{cc}
-1&0\\
1&0
\end{array}\right] 
\mbox{ for }\tau\in[0,1/2)\quad\mbox{ and }\quad  
L(\tau)=\left[\begin{array}{cc}
0&1\\
0&-1
\end{array}\right]
\mbox{ for }\tau\in[1/2,1).$$ 
The monodromy matrix is $\Phi(T)=e^{TA_2/2}e^{TA_1/2}$, where $A_1$ and $A_2$ are given by
$$
A_1=\left[
\begin{array}{cc}
1/2-m&0\\
m&-3/2
\end{array}
\right],
\quad
A_2=\left[
\begin{array}{cc}
-1&m\\
0&1/2-m
\end{array}
\right].
$$ 
The matrix  $\Phi(T)$
is positive. From Proposition \ref{Prop12} we deduce that the growth rate exists and is given by $\Lambda(m,T)=\frac{1}{T}\ln(\mu(m,T)),$
where $\mu(m,T)$ is the Perron root of the monodromy matrix $\Phi(T)$. Figure \ref{fig17}(b) shows the graph of $\Lambda(m,T)$: we see that DIG occurs.

\subsubsection{The threshold $\chi$ is positive, but DIG does not occur}\label{ExLobryNoDIG}

We consider the three-patch model described in Figure \ref{fig15}, where the growth rate is indicated in each patch. The migration is symmetric and is only between the patches where the growth rate is $b$. We assume that $a>0>b$ and $a+2b<0$. Therefore $\overline{r}_1=\overline{r}_2=\overline{r}_3=\frac{a+2b}{3}<0$ and $\chi=a>0$. Does the DIG phenomenon occurs for this system ?
%%%%%%%%%%%%%
\begin{figure}[ht]
\begin{center}
\includegraphics[width=10cm,
viewport=190 580 410 690]{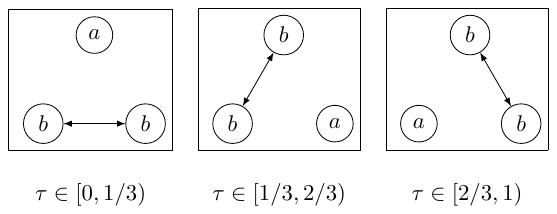}
\caption{The three-patch model.
\label{fig15}}
\end{center}
\end{figure}
%%%%%%%%%%%% 
\begin{figure}[ht]
\begin{center}
\includegraphics[width=10cm,
viewport=150 550 420 730]{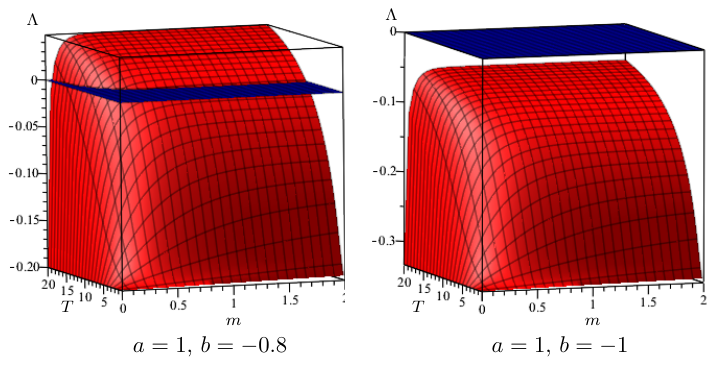}
\caption{If $a=1$ then DIG occurs for $-0.8\leq b<-0.5$ and does not occur for $b\leq -1$.
\label{fig16}}
\end{center}
\end{figure} 
%%%%%%%%%%% 

The monodromy matrix is $\Phi(T)=e^{TA_3/3}e^{TA_2/3}e^{TA_1/3}$, where $A_1$, $A_2$  and $A_3$ are the matrices
$$
A_1=\left[\begin{array}{ccc}
a-m&m&0\\
m&a-m&0\\
0&0&b
\end{array}\right],
\quad
A_2=\left[\begin{array}{ccc}
a-m&0&m\\
0&b&0\\
m&0&a-m
\end{array}\right],
$$
$$
A_3=\left[\begin{array}{ccc}
b&0&0\\
0&a-m&m\\
0&m&a-m
\end{array}\right]
$$
The matrix  $\Phi(T)$
is positive. From Proposition \ref{Prop12} we deduce that the growth rate exists and is given by $\Lambda(m,T)=\frac{1}{T}\ln(\mu(m,T)),$
where $\mu(m,T)$ is the Perron root of the monodromy matrix $\Phi(T)$. Figure \ref{fig16} shows the graph of $\Lambda(m,T)$ for $a=1$ and $b=-1$ (right) in which DIG does not occur and $a=1$ $b=-0.8$ (left) in which DIG occurs. 
 Therefore, in this case $\chi$ is not equal to the upper limit of $\Lambda(m,T)$ as in the irreducible case, so that we can have $\chi>0$ but DIG does not occur. Determining the upper limit of $\Lambda(m,T)$ is an open problem, whose resolution would enable us to determine the value of the DIG threshold for this example.

\section{Proofs of the main results}
\label{IP}

We'll first give the proof of Proposition \ref{Prop3}, on the existence of the growth rate, followed by the proof of Theorem \ref{thm2}, which gives the integral formula for calculating this rate. This formula is then used in the proof of Theorem \ref{thm1}.

The existence of the growth rate  and its integral formula  follow from similar results on linear cooperative 1-periodic systems.  
Indeed, by changing time, we can return to the case where the period is 1 and make $T$ appear as a parameter of the system. 
The change of variables
$$t/T=\tau,\quad y(\tau)=x(T\tau),\quad Y(\tau)=X(T\tau)$$
transforms \eqref{eq3} and the corresponding matrix equation \eqref{FMA} into the equations
\begin{equation}\label{A(tau)}
  \frac{dy}{d\tau}=TA(\tau)y,\qquad \frac{dY}{d\tau}=TA(\tau)Y.  
\end{equation}
These equations are 1-periodic.
In order not to burden the presentation, the results on the existence of the growth rate for the system \eqref{A(tau)}, and the formulas that give it, are discussed in Appendix \ref{CLDE}. 
The main result of Appendix \ref{CLDE} is the Theorem \ref{Lambda(x)} which shows that the system \eqref{A(tau)} admits a growth rate. This theorem also provides two formulas to compute the growth rate, one using the Perron root  of the monodromy matrix $Y(1)=X(T)$ and the other using a periodic GAS solution of the equation associated with system \eqref{A(tau)} in the simplex $\Delta$. The existence and global asymptotic stability of this periodic solution is given in Proposition \ref{periodicPsi} in Appendix \ref{CLDE}. 
 %%%%%%%%%%%%
 
\subsection{Proof of Proposition \ref{Prop3}}\label{IP1}
We give the proof for the system \eqref{A(omega)}, see Remark \ref{notmoregeneral}.
\color{black}
Using Theorem \ref{Lambda(x)} of Appendix \ref{CLDE}, for any solution $y(\tau)$ of \eqref{A(tau)}, such that $y(0)>0$, we have
$$\lim_{\tau\to\infty}\frac{1}{\tau}\ln(y_i(\tau))=\ln(\mu(T)),$$
where $\mu(T)$ is the Perron root of the monodromy matrix $Y(1)=X(T)$. 
We use the notation $\mu(T)$ to recall the dependence of the Perron root on the parameter $T$ in the system \eqref{A(tau)}. Since
$$\lim_{\tau\to\infty}\frac{1}{\tau}\ln(y_i(\tau))=
\lim_{t\to\infty}\frac{1}{t/T}\ln(x_i(t))=T\lim_{t\to\infty}\frac{1}{t}\ln(x_i(t)),$$
we deduce that for all $i$, 
$\lim_{t\to\infty}\frac{1}{t}\ln(x_i(t))=\frac{1}{T}\ln(\mu(T)).$
Therefore, for all $i$, 
$\Lambda[x_i]=\Lambda(T)$, where $\Lambda(T)=\frac{1}{T}\ln(\mu(T))$. This proves \eqref{GR} or, equivalently, \eqref{Lambda}.

We prove now that
$\Lambda(T)$ is analytic in $T$. The monodromy matrix  $X(T)$ is analytic in $T$. Indeed, the solutions of the differential equation \eqref{A(tau)}, which is analytic in the parameter $T$ are analytic in this parameter. Hence $X(T)$ is analytic in $T$. So is its Perron root $\mu(T)$, since it is a simple root of the characteristic polynomial, see \cite{Brillinger}. Therefore $\Lambda(T)$ is analytic in $T$.

In the special case where $A(\tau)=R(\tau)+mL(\tau)$, the growth rate is also analytic in $m$, since the differential equation is analytic in $m$

\subsection{Proof of Theorem \ref{thm2}}\label{PT8}
\color{black}
Using the change of variable 
$$\tau=t/T,
\quad\sigma(\tau)=\rho(T\tau),\quad \eta(\tau)=\theta(T\tau),$$ 
the system (\ref{eqrhoT},\ref{eqthetaT}) becomes
\begin{equation}\label{eqrhoTtau}
\begin{array}{lcl}
\frac{d\sigma}{d\tau}&=&T\langle A(\tau)\eta,
{\bf 1}\rangle \sigma
\\
\frac{d\eta}{d\tau}&=&TA(\tau)\eta- T\langle A(\tau)\eta,
{\bf 1}\rangle \eta
\end{array}
\end{equation}
According to Proposition \ref{periodicPsi} in Appendix \ref{CLDE}, the second equation in \eqref{eqrhoTtau}, has a periodic solution, denoted $\eta^*(\tau,T)$ to emphasize its dependence on the parameter $T$, which is globally asymptotically stable. Recall that $\eta^*(\tau,T)$ is the solution of initial condition $\eta^*(0,T)=\pi(T)$, where $\pi(T)$ is the Perron vector of the monodromy matrix $X(T)$. Therefore, $\theta^*(t,T):=\eta^*(t/T,T)$ is a $T$-periodic solution of \eqref{eqthetaT}. It is globally asymptotically stable. As a consequence of Theorem \ref{Lambda(x)} of Appendix \ref{CLDE}, we have the formula \eqref{Lambda=Lambda[rho]1} for the growth rate $\Lambda(T)$ of the equation \eqref{A(omega)}.

\subsection{Proof of Theorem \ref{thm1}, item 1}\label{HFL}
We give the proof for the system \eqref{A(omega)}, see Remark \ref{notmoregeneral}.
Our aim is to determine the limit of $\Lambda(T)$ as $T\to 0$. We use the averaging method. Let us briefly recall the principle of this method. Let $x(t,a,\varepsilon)$ be the solution of the initial value problem 
\begin{equation}\label{av0}
\frac{dx}{dt}=f(t/\varepsilon,x),\quad x(0)=a,
\end{equation}
where $f(\tau,x)$ is 1-periodic in $\tau$, Lipschitz continuous in $x$ on some domain $D$, and continuous in $\tau$. Let $\overline{f}(x)=\int_0^1f(\tau,x)d\tau$ be the average of $f$. If $y(t,a)$ is the solution of the initial value problem
\begin{equation}\label{av1}
\frac{dy}{dt}=\overline{f}(y),\quad y(0)=a,
\end{equation}
then, as $\varepsilon\to 0$, we have $x(t,a,\varepsilon)=y(\varepsilon t,a)+o(1)$ uniformly for $t\in[0,1/\varepsilon]$, see \cite[Section 4.2]{SVM}. This result, known as the averaging theorem, was extended from the periodic case to the so called KBM vector fields (KBM stands for Krylov, Bogoliubov and Mitropolsky) for which the average $\overline{f}(x)=\lim_{T\to\infty}\frac{1}{T}\int_0^Tf(\tau,x)d\tau$ exists, see \cite{bogolioubov,mitropolsky,RoseauED}, \cite[Chapter 10]{Khalil} or \cite[Section 4.3]{SVM}. It was also extended in the random case, see \cite[Chapter 7]{freidlin}. More information on the development of the theory can be found in \cite[Section 4.1]{SVM} and \cite[Section 7.1]{freidlin}.
See also \cite[Chapter 12]{RoseauVib} where the averaging method is known as \emph{synchronization theory}. This book presents an interesting extension of the method to time-periodic vector fields with discontinuities in the time variable, which is usefull in our context, see \cite[Page 114]{RoseauVib}. This extension to the discontinuous case was proposed in \cite{haag}. Note also that the method of averaging applies with Caratheodory conditions, i.e. the vector field $f(\tau,x)$ is measurable in $\tau$ for all $x$ and continuous in $x$ for almost all $\tau$, and for differential equations in Banach spaces, see \cite[Chapter 7]{RoseauED}. In the following, we will give an averaging result with weaker assumptions than those of \cite{haag,RoseauVib,RoseauED}, but which are sufficient in our context.

\begin{theorem}\label{avthm}
Assume that $f(\tau,x)$ is 1-periodic in $\tau$, Lipschitz continuous in $x$, and piecewise continuous in $\tau$, with a finite number of discontinuities on $[0,1]$ and has left and right limits at the discontinuity points. Let $x(t,a,\varepsilon)$ be the solution of \eqref{av0}. Let $y(t,a)$ be the solution of the averaged system \eqref{av1}, which is assumed to exist at least on $[0,1]$. As $\varepsilon\to 0$, we have $x(t,a,\varepsilon)=y(\varepsilon t,a)+o(1)$ uniformly for $t\in[0,1/\varepsilon]$.
\end{theorem}

We return now to the proof of Theorem \ref{thm1}, item 1. 
Using Theorem \ref{thm2}, the growth rate of the equation \eqref{A(omega)} is given by \eqref{Lambda=Lambda[rho]1},
where $\theta^*(t,T)$ is the $T$-periodic solution of the equation \eqref{eqthetaT}. For clarity we rewrite here this equation:
$$
\frac{d\theta}{dt}=A(t/T)\theta- \langle A(t/T)\theta,
{\bf 1}\rangle \theta.$$
Since $T\to 0$ we can apply the averaging method to this equation. 
Let $\theta(t,T)$ be the solution of this equation, with initial condition 
$\theta(0,T)=\theta_0$.
From the Theorem \ref{avthm} we deduce that, as $T\to0$, $\theta(t,T)$  is approximated by $\overline{\theta}(Tt)$, where $\overline{\theta}(t)$ is the solution of the averaged equation
\color{black}
\begin{equation}
\label{eqrhothetaAveraged}
\begin{array}{lcl}
\frac{d\theta}{dt}&=&\overline{A}\theta- \langle \overline{A}\theta,
{\bf 1}\rangle \theta,
\end{array}
\end{equation}
with the same initial condition $\overline{\theta}(0)=\theta_0$.
The averaged equation \eqref{eqrhothetaAveraged} has a globally asymptotically stable equilibrium in $\Delta$. 
Indeed, let $w=(w_1,\cdots,w_n)^\top$ be the
Perron-Frobenius vector of $\overline{A}$, i.e., the unique positive eigenvector corresponding to the eigenvalue $\lambda_{max}\left(\overline{A}\right)$ of the matrix $\overline{A}$, such that $\langle w,
{\bf 1}\rangle=1$. We have $w\in\Delta$ and  $w$ is the unique positive equilibrium of \eqref{eqrhothetaAveraged}. Using Proposition \ref{periodicPsi1} in Appendix \ref{CLDE}, $w$ is 
GAS for \eqref{eqrhothetaAveraged}  in the simplex $\Delta$. 
Since the averaged equation has an attractive equilibrium $w$, as $T\to 0$, the $T$-periodic solution $\theta^*(t,T)$ of the second equation in  \eqref{eqrhoTtau} converges toward $w$.
Hence,  using \eqref{Lambda=Lambda[rho]1}, as $T\to 0$, we have
$$\Lambda(T)=\int_0^{1}
\langle A(\tau)\theta^*(T\tau,T),{\bf 1}\rangle 
d\tau
=
\int_0^{1}
\langle A(\tau)w,{\bf 1}\rangle 
d\tau+o(1)
$$
Using
$\overline{A}w=\lambda_{max}\left(\overline{A}\right)w$ and $\langle w,{\bf 1}\rangle=1$,  we have
$$\textstyle
\int_0^1\langle A(\tau)w,{\bf 1}\rangle d\tau
=\langle \left(\int_0^1A(\tau)d\tau\right) w,{\bf 1}\rangle=
\langle \overline{A}w,{\bf 1}\rangle
=
\lambda_{max}\left(\overline{A}\right)\langle w,{\bf 1}\rangle 
=\lambda_{max}\left(\overline{A}\right).
$$  
Therefore, as $T\to 0$,  $\Lambda(T)=\lambda_{max}\left(\overline{A}\right)+o(1).$
This proves  \eqref{GR0} or, equivalently, \eqref{T=0}.

\subsection{Proof of Theorem \ref{thm1}, item 2}\label{LFL}
We give the proof for the system \eqref{A(omega)}, see Remark \ref{notmoregeneral}.
\color{black}
Our aim is to determine $\lim_{T\to \infty}\Lambda(T)$.
As shown by  \eqref{eqthetaT}, the equation on the simplex $\Delta$ is
\begin{equation}\label{zt}
\begin{array}{l}
\frac{d\theta}{dt}=A(\tau)\theta- \langle A(\tau )\theta,1\rangle\\[1mm]
\frac{d\tau}{dt}=\frac{1}{T}
\end{array}
\end{equation}
We use the change of variable $\tau=t/T$ and $\eta(\tau)=\theta(T\tau)$. The equation \eqref{zt} becomes
\begin{equation}\label{ztau}
\frac{1}{T}\frac{d\eta}{d\tau}=
A(\tau)\eta- \langle A(\tau )\eta,1\rangle\eta
\end{equation}
When $T\to\infty$  this is a singularly perturbed equation whose study is achieved using Tikhonov's theorem, see Appendix \ref{SecTikhonov}.
The systems \eqref{zt} and \eqref{ztau} are equivalent. The first one is written using the fast time $t$, while the second one is written using the slow time $\tau$. These systems have $n-1$ fast variables $\theta$ and one slow variable $\tau$. The fast dynamics, obtained from \eqref{zt} by letting $\frac{1}{T}=0$ is 
\begin{equation}
\label{zfast}
\begin{array}{l}
\frac{d\theta}{dt}
=A(\tau)\eta- \langle A(\tau )\eta,1\rangle\eta
\end{array}
\end{equation}
where $\tau$ is considered as a parameter. Let us prove that the hypotheses of the Proposition \ref{PropTikhonov} in Appendix \ref{SecTikhonov} are satisfied. The conditions H0' and H1 hold. It remains to prove that the condition H2' also hold. This is true since the fast equation \eqref{zfast} admits the Perron-Frobenius vector $v(\tau)$ of $A(\tau)$ as an equilibrium which is globally asymptotically stable in the simplex $\Delta$, see Proposition \ref{periodicPsi1} in Appendix \ref{CLDE}. 
Therefore, according to Proposition \ref{PropTikhonov} (see Remark \ref{Noslowvariable} following this proposition), the solution $\eta(\tau,T)$ of \eqref{ztau} is approximated by the slow curve $v(\tau)$. More precisely, for any $\nu>0$, as small as we want, as $T\to \infty$, we have
\begin{equation*}
\eta(\tau,T)=v(\tau)+o(1)
\mbox{ uniformly on }[0,1]\setminus \bigcup_{k=0}^p[\tau_k,\tau_k+\nu],
\end{equation*}
where $\tau_0=0$ and $\tau_k$, $1\leq k\leq p$, are the discontinuity points of $A(\tau)=R(\tau)+mL(\tau)$.
Hence the unique $T$-periodic solution $\eta^*(t,T)$ of the second equation in \eqref{eqrhoTtau} (see the proof of Theorem \ref{thm2}),  satisfies  
$$\eta^*(\tau,T)=v(\tau)+o(1)
\mbox{ uniformly on }[0,1]\setminus \bigcup_{k=1}^p[\tau_k,\tau_k+\nu].$$
From this formula and $\theta^*(T\tau,T)=\eta^*(\tau,T)$ we deduce that
$$\theta^*(T\tau,T)=v(\tau)+o(1)
\mbox{ uniformly on }[0,1]\setminus \bigcup_{i=1}^p[\tau_k,\tau_k+\nu].$$ 
Since $\nu$ can be chosen as small as we want, as $T\to\infty$, using \eqref{Lambda=Lambda[rho]1}, we have
$$\Lambda(T)=\int_0^{1}
\langle A(\tau)\theta^*(T\tau,T),{\bf 1}\rangle 
d\tau
=
\int_0^{1}
\langle A(\tau)v(\tau),{\bf 1}\rangle 
d\tau+o(1)\\[1mm]
$$
Using  $A(\tau)v(\tau)=\lambda_{max}(A(\tau)) v(\tau)$ and $\langle v(\tau),{\bf 1}\rangle=1$,  we have
$$
\begin{array}{ll}
\int_0^{1}
\langle A(\tau)v(\tau),{\bf 1}\rangle 
d\tau
=
\int_0^{1}
\lambda_{max}(A(\tau))\langle v(\tau),{\bf 1}\rangle 
d\tau
=
\int_0^{1}
\lambda_{max}(A(\tau)) 
d\tau.
\end{array}
$$  
Therefore, as $T\to\infty$, $\Lambda(T)=\int_0^{1}
\lambda_{max}(A(\tau))d\tau+o(1)=\overline{\lambda_{max}(A)}+o(1).$
This proves  \eqref{GRinfini} or, equivalently, \eqref{T=infini}.

The behavior of $\theta^*(T\tau,T)$ as $T\to\infty$ is illustrated in the supplementary material Figure S6. The figure shows that the approximation of $\theta^*(T\tau,T)$ by the Perron-Frobenius vector $v(\tau)$ is uniform except on the small intervals $[\tau_k,\tau_{k}+\nu]$, where $\tau_k$ is a discontinuity of $v(\tau$. In these thin layers, the solution jumps quickly from the left limit of $v(\tau)$ at $\tau_k$ to its right limit.

\subsection{Proof of Theorem \ref{thm1}, item 3}\label{SM}
The proof of \eqref{m=0} follows from the continuous dependence of the solutions of \eqref{eq3} in the parameter $m$.  
Let $x(t,m)$ be the solution of \eqref{eq3}
with initial condition 
$x(0,m)=x^0\geq 0$. 
Recall that the matrix $A(\tau)=R(\tau)+mL(\tau))$ has a finite number of discontinuity points $\tau_k$, $1\leq k\leq p$ in the interval $[0,1]$.
Using the continuous dependence of the solutions on the parameter $m$, in each sub interval $[\tau_k,\tau_{k+1}]$ on which the matrix $A(\tau)$ is continuous, we deduce that, as $m\to 0$, we have $x(t,m)=\xi(t)+o(1)$, uniformly for $t\in[0,T]$, where $\xi(t)$ is the solution, with initial condition 
$\xi(0)=x^0$,  of the diagonal system 
\begin{equation}
\label{eqm=0}
\frac{d\xi}{dt}=R(t/T)\xi,
\end{equation} 
obtained from \eqref{eq3} by letting $m=0$.
The solution $\xi(t)$ is given by
$$\xi_i(t)=x_i^0e^{\int_{0}^t r_i(s/T)ds}=x_i^0e^{T\int_{0}^{t/T} r_i(\tau)d\tau},\quad 1\leq i\leq n.$$  
Therefore, if $\Phi(m,T)$ is the monodromy  matrix  of \eqref{eq3},
as $m\to 0$, we have  
$$\lim_{m\to 0}\Phi(m,T)=
{\rm diag}(e^{T\overline{r}_1},\ldots,e^{T\overline{r}_n}),$$
where  the diagonal matrix is the fundamental matrix of \eqref{eqm=0}.
The dominant eigenvalue of this diagonal matrix is $e^{T\max_{1\leq i\leq n}\overline{r}_i}$. 
Using the continuity of the  Perron root \cite{Meyer}, we have
$\lim_{m\to 0}\mu(m,T)=e^{T\max_{1\leq i\leq n}\overline{r}_i}.$ 
Using \eqref{Lambda}, we have 
$\lim_{m\to 0}\Lambda(m,T)=
\lim_{m\to 0}\frac{1}{T}\ln(\mu(m,T))
=\max_{1\leq i\leq n}\overline{r}_i.$ This proves \eqref{m=0}.

\subsection{Proof of Theorem \ref{thm1}, item 4}\label{FM}
\color{black}

In this section we consider the following more general system
\begin{equation}\label{RmL}
\frac{dx}{dt}=R(t/T)+mL(t/T)
\end{equation}
where  the matrix $R$ is not assumed to be diagonal as in \eqref{eq3}.
We have the following result.

\begin{propo}\label{Propm=infini}
Assume that for any $\tau\in[0,1]$ the matrix $L(\tau)$ is Metzler irreducible and its columns sum to 0. Let 
$p(\tau)$ be the Perron-Frobenius vector of $L(\tau)$. 
Let $\Lambda(m,T)$ be the growth rate of \eqref{RmL}.
We have 
\begin{equation}\label{FMFormula}
\lim_{m\to\infty}\Lambda(m,T)=
\overline{\langle Rp,{\bf 1}\rangle}.
\end{equation}
\end{propo}

 \begin{proof}
 The equation on the simplex $\Delta$ is
 \begin{equation}\label{eqt0}
\frac{d\theta}{dt}=R(t/T)\theta+mL(t/T)\theta-\langle
R(t/T)\theta,{\bf 1}\rangle\theta-
m\langle L(t/T)\theta,{\bf 1}\rangle\theta,
\end{equation}
Since the columns of $L(\tau)$ sum to 0 we have 
$\langle L(\tau)\theta,{\bf 1}\rangle=0$. Therefore, using the variables $\tau=t/T$ and $\eta(\tau)=\theta(T\tau)$, this equation is written
\begin{equation}\label{zt0}
\frac{d\eta}{d\tau}=TR(\tau)\eta+TmL(\tau)\eta-T\langle
R(\tau)\eta,{\bf 1}\rangle\eta,
\end{equation}
Dividing by $m$ we obtain
\begin{equation}\label{ztau0}
\frac{1}{m}\frac{d\eta}{d\tau}=TL(\tau)\eta
+\frac{1}{m}\left[TR(\tau)\eta-T\langle
R(\tau)\eta,{\bf 1}\rangle\eta\right].
\end{equation}
When $m\to\infty$  this is a singularly perturbed equation with $n-1$ fast variables $\theta$ and one slow variable $\tau$. 
Using the fast time $s=m\tau$, this equation is written
\begin{equation}\label{zs0}
\begin{array}{l}
\frac{d\eta}{ds}=TL(\tau)\eta
+\frac{1}{m}\left[TR(\tau)\eta-T\langle
R(\tau)\eta,{\bf 1}\rangle\eta\right],\\[1mm]
\frac{d\tau}{ds}=\frac{1}{m}
\end{array}
\end{equation}
Therefore, the fast dynamics, obtained by letting $\frac{1}{m}=0$ in 
\eqref{zs0} is
\begin{equation}
\label{zfast0}
\begin{array}{l}
\frac{d\eta}{ds}
=TL(\tau)\eta,
\end{array}
\end{equation}
where $\tau$ is considered as a parameter. 
 Let us prove that the hypotheses of the Proposition \ref{PropTikhonov} in Appendix \ref{SecTikhonov} are satisfied. The conditions H0' and H1 hold. It remains to prove that the condition H2' also hold. This is true since the Perron-Frobenius vector $p(\tau)$ of $L(\tau)$, is the unique positive equilibrium of \eqref{zfast0} and is GAS in the simplex $\Delta$, as shown in the Proposition \ref{periodicPsi1} in Appendix \ref{CLDE}. 
Therefore, according to Proposition \ref{PropTikhonov} (see Remark \ref{Noslowvariable} following this proposition), the solution $\theta(\tau,m)$ of \eqref{ztau0} is approximated by the slow curve $p(\tau)$. More precisely, for any $\nu>0$, as small as we want, as $T\to \infty$, we have
\begin{equation*}
\eta(\tau,T)=p(\tau)+o(1)
\mbox{ uniformly on }[0,1]\setminus \bigcup_{k=0}^p[\tau_k,\tau_k+\nu],
\end{equation*}
where $\tau_0=0$ and $\tau_k$, $1\leq k\leq p$, are the discontinuity points of $A(\tau)=R(\tau)+mL(\tau)$.
Hence the unique $T$-periodic solution $\theta^*(t,T)$ of \eqref{eqt0}  satisfies  
$$\theta^*(T\tau,m)=p(\tau)+o(1)
\mbox{ uniformly on }[0,1]\setminus \bigcup_{i=1}^p[\tau_k,\tau_k+\nu].$$ 
From \eqref{Lambda=Lambda[rho]1}, we have
$\Lambda(m,T)=\int_0^{1}
\langle R(\tau)\theta^*(T\tau,T),{\bf 1}\rangle 
d\tau$.
Since $\nu$ can be chosen as small as we want, as $m\to\infty$, we have
$$
\begin{array}{l}
\int_0^{1}\langle R(\tau)\theta^*(T\tau,T),{\bf 1}\rangle d\tau
=
\int_0^{1}\langle R(\tau)p(\tau),{\bf 1}\rangle d\tau+o(1)
=\overline{\langle Rp,{\bf 1}\rangle}+o(1).
\end{array}
$$
Therefore, as $m\to\infty$, $\Lambda(m,T)=
\overline{\langle Rp,{\bf 1}\rangle}+o(1)$. 
\end{proof}
In the particular case of \eqref{eq3}, where $R(\tau)={\rm diag}(r_1(\tau),\cdots,r_n(\tau))$ is diagonal, the formula \eqref{FMFormula} becomes 
$\lim_{m\to\infty}\Lambda(m,T)=\sum_{i=1}^n\overline{p_ir_i}.$ 
This proves \eqref{m=infini}.

\subsection{Proof of Proposition \ref{PropDL}}
\label{ProofPropDL}
We use the following result.

\begin{lem}
\label{lem:regularity-lambdamax}
{Let $B$ and $H$ two Metzler matrices.} For $\varepsilon\geq 0$, the function $\varepsilon\mapsto\lambda_{max}(\varepsilon):=\lambda_{max}(B+\varepsilon H)$ is continuous in a neighborhood of 0. 
If furthermore $B$ is irreducible, then 
$\varepsilon\mapsto\lambda_{max}(\varepsilon)$ is differentiable at $\varepsilon=0$, and
\begin{equation}\label{lambdamax'}
\lambda_{max}'(0)= y^\top Hx,
\end{equation}
where $x$ and $y$ are respectively the right- and left eigenvectors of $B$ associated to $\lambda_{max}(B)$
such that $1^\top x = 1$ and $y^\top x = 1$. 
\end{lem}{
\begin{proof}[Proof of Lemma \ref{lem:regularity-lambdamax}]

The continuity of the Perron root for matrices with non-negative entries is given in  \cite[Theorem 3.1]{Meyer}, in the irreducible case, and \cite[Theorem 3.2]{Meyer}, in the reducible case. The continuity of the spectral abscissa of $B+\varepsilon H$ follows from this result applied to the matrix $B+\varepsilon H+\eta Id$, where $\eta\geq 0$ is such that $B+\varepsilon H+\eta Id$ has non-negative entries for $\varepsilon\geq 0$ small enough.
A proof of the formula for the derivative can be found in \cite[Theorem 6.3.12]{HJ}.
\color{black}
\end{proof}}

Using $\Lambda(m,0)=\lambda_{max}(\overline{R}+m\overline{L})$, $B=\overline{R}$ and $H=\overline{L}$, we deduce that
$$
\lim_{m\to 0}\Lambda(m,0)=
\lambda_{max}(\overline{R})=\max_{1\leq i\leq n}\overline{r}_i.$$

Using $\Lambda(m,\infty)=\int_0^1\lambda_{max}(R(\tau)+mL(\tau))d\tau$, $B={R}(\tau)$ and $H=L(\tau)$, we deduce that
$$
\lim_{m\to 0}\Lambda(m,\infty)=
\int_0^1\lambda_{max}(R(\tau))d\tau=
\int_0^1{\max}_{1\leq i\leq n}(r_i(\tau))d\tau=
\chi.
$$
Note that we can exchange the limit and the integral by dominated convergence, since for all $\tau \in [0,1]$, 
\begin{equation}\label{min<L<max}
 \min_i r_i(\tau) \leq \lambda_{\max}( R(\tau)  + m L(\tau)) \leq \max_i r_i(\tau).   
\end{equation}
Now, using \eqref{lambdamax'}, 
$B=\overline{L}$  and $H=\overline{R}$, and the fact that $\lambda_{max}(\overline{L})=0$, $x=q$ and $y=1$, we have
\begin{equation*}
\begin{array}{lcl}
\lambda_{max}(\overline{R} + m\overline{L}) &=& 
m\lambda_{max}\left(\frac{1}{m}\overline{R}+ \overline{L}\right)\\
&=&
m\left(\lambda_{max}(\overline{L})+\frac{1}{m}1^\top \overline{R}q+o\left(\frac{1}{m}\right)\right)\\
&=&\sum_{i=1}^nq_i\overline{r}_i+o(1). 
\end{array}
\end{equation*}
Therefore,
$$\textstyle\lim_{m\to\infty}\Lambda(m,0)=\lim_{m\to\infty}\lambda_{max}(\overline{R} + m\overline{L})= 
\sum_{i=1}^nq_i\overline{r}_i.$$
Similarly, using \eqref{lambdamax'}, 
$B=L(\tau)$  and $H=R(\tau)$, and the fact that $\lambda_{max}(L(\tau))=0$, $x=p(\tau)$ and $y=1$, we have
\begin{equation*}
\begin{array}{lcl}
\lambda_{max}(L(\tau)+ mL(\tau)) &=& 
m\lambda_{max}\left(\frac{1}{m}R(\tau)+L(\tau)\right)
\\&=&m\left(\lambda_{max}(L(\tau))+\frac{1}{m}1^\top R(\tau)p(\tau)+o\left(\frac{1}{m}\right)\right)\\&=&
\sum_{i=1}^np_i(\tau)r_i(\tau)+o(1). 
\end{array}
\end{equation*}
Therefore, we obtain
$$
\begin{array}{lcl}
\lim_{m\to\infty}\Lambda(m,\infty)&=&\int_0^1\lim_{m\to\infty}\lambda_{max}({R}(\tau) + mL(\tau))d\tau\\
&=& 
\int_0^1\sum_{i=1}^np_i(\tau){r}_i(\tau)d\tau
= \sum_{i=1}^n\overline{p_ir_i}.
\end{array}
$$
Note that we can also exchange the limit and the integral by using \eqref{min<L<max} and dominated convergence.

By \cite[Theorem 1.1]{chen}, for any diagonal  matrix $S={\rm diag}(s_1,\ldots,s_n)$, and Metzler irreducible matrix $K$ whose columns sum to 0, we have
$$\frac{d}{dm}\lambda_{max}(S+mK)\leq \lambda_{max}(K)=0,\qquad \frac{d^2}{dm^2}\lambda_{max}(S+mK)\geq 0,$$
and the equality holds if and only if $s_1=\ldots=s_n$. 

Using $\Lambda(m,0)=\lambda_{max}(\overline{R}+m\overline{L})$, $S=\overline{R}$ and $K=\overline{L}$ we deduce that \eqref{Lambda(m,0)convex} is true and the equality holds if and only if $\overline{r}_1=\ldots=\overline{r}_n$. Similarly, using $\Lambda(m,\infty)=\int_0^1\lambda_{max}(R(\tau)+mL(\tau))d\tau$, $S={R}(\tau)$ and $K={L}(\tau)$ we deduce that \eqref{Lambda(m,infini)convex} is true and the equality holds if and only if ${r}_1(\tau)=\ldots={r}_n(\tau)$ for all $\tau$.

Finally, we prove \eqref{Lambda(m,0)<Lambda(m,T)}. For $T > 0$, and $t \geq 0$, let $C(t)=A(t/T) = R(t/T) + m L$ (recall that we have assumed here that the migration is constant). Then, $C$ is a $T$ - periodic function, with constant off-diagonal entries. By \cite[Theorem II.5.3]{Mierczynski}, 
\[
\Lambda(m,T) \geq \lambda_{\max}( \overline C),
\]
where
\[
\overline C = \frac{1}{T} \int_0^T C(t) dt = \overline R + m L.
\]
By \eqref{T=0}, $\lambda_{\max}( \overline C) = \Lambda(m,0)$, which concludes the proof.

\section{Discussion}\label{Discussion}

We have considered the non-autonomous linear differential system 
\begin{equation} \label{mod}
\frac{dx}{dt} = R(t/T)x + mL(t/T) x,  
\quad x = (x_1,\cdots,x_n)\in \mathbb{R}^n
\end{equation}
where
\begin{itemize}
\item $x_i(t)$ is the size of the population at time $t$ on the $i$th path,
\item $ s \mapsto R(s) $ is a  1-periodic diagonal matrix representing the growth rate on each patch,
\item $ s \mapsto L(s)$  is a 1-periodic  migration matrix  (i.e. the sum of each column is $0$).
\end{itemize}
We assumed that
\begin{itemize}
\item \textbf{$H_1$}  the functions  $ s \mapsto R(s) $ and $ s \mapsto L(s)$ are piecewise continuous,
\item \textbf{$H_2$} for every $s$, the matrix $L(s)$ is irreducible.
\end{itemize}

Our first point (see Proposition \ref{Prop3}) was to prove that, as soon as $m$ is positive 
$$
\lim_{t \to \infty}\frac{1}{t}\ln(x_i(t)) = \Lambda(m,T) :=  \frac{1}{T}\ln(\mu(m,T))
$$
where $\mu(m,T)$ is the Perron root of the monodromy matrix associated to \eqref{mod}.
Thus the asymptotic growth rate is the same on every patch, which is not surprising since we assumed that the migration matrix $L(s)$ is irreducible which means that the different patches are mixed altogether. Nevertheless this result needs a proof, either in the symmetric constant migration case considered by Katriel, (see  \cite[Eq. (20)]{Katriel}), either in our case 
which relies on properties of cooperative linear 1-periodic systems (see Appendix  \ref{CLDE}).

We have adopted Katriel's definition (see  \cite{Katriel}) of \\
\textbf{Dispersal Induced Growth}:  {\em One says that {\it{dispersal-induced growth}} (DIG) occurs 
if all patches are sinks (i.e. $\int_0^1{r}_i(s)ds<0,\;  i = 1,\cdots,n$), but  $\Lambda(m,T)>0$ for some values of $m$ and $T$}. 
We also adopted his index
 $$\chi =\int_0^1\max_i r_i(s)ds. $$

 This index  can be interpreted as the average growth rate in a kind of {\em idealized habitat} whose growth rate at any time is that of the habitat with {\em maximal growth} at this time. We have proved that DIG occurs if, and only if, $\chi > 0$, that is to say, {\em if the ideal habitat is a source}. This is our answer to the question addressed in the title.\\

To obtain this result we first 
characterized the set of parameters  $m$ and $T$ for which $\Lambda(m,T)>0$ (see Proposition \ref{PropDIG}). We have also proved the existence of the limits 
$$ \lim_{m \to 0} \;\Lambda(m,T) \quad  \lim_{m \to \infty } \;\Lambda(m,T)\quad  \lim_{T \to 0 } \;\Lambda(m,T)\quad  \lim_{T \to \infty } \;\Lambda(m,T)$$
and have analyzed their asymptotics for large/small $m$ and $T$ in Section \ref{asymptotics}.  We have shown  that the picture  depends on the sign of the quantity 
$$ \alpha = \sum_{i=1}^n \int_0^1p_i(s)r_i(s)ds$$
where $p_i(s)$ is the Perron vector of the matrix $L(s)$. 

\begin{itemize}
\item 
If $\alpha <0$,  then the equation $\Lambda(m,\infty)=0$ has a unique solution $m=m^*>0$, and  we have	
	\begin{itemize}
	\item If $m\in (0,m^*)$ then for any $T$ sufficiently large (depending on $m$), we have 
		$\Lambda(m,T)>0$ (growth) and for any $T$ sufficiently small (depending on $m$), we have 
		$\Lambda(m,T)<0$ (decay).
	\item If $m\geq m^*$ then for any $T$ sufficiently small or sufficiently large (depending on $m$), we have $\Lambda(m,T)<0$ (decay).  
	\end{itemize}	
\item If $\alpha \geq 0$,  then  $\Lambda(m,\infty)>0$ for all $m> 0$ and for any $T$ sufficiently large (depending on $m$), we have 
		$\Lambda(m,T)>0$ (growth) and for any $T$ sufficiently small (depending on $m$), we have 
		$\Lambda(m,T)<0$ (decay).
\end{itemize}

Our proofs rely mostly on a change of variable (appendix  \ref{CLDE}) which is familiar to people interested by growth in ``structured populations''. Here the ``structure'' is given by the patch where the population is located and instead of considering the size $x_i$ of the population on each patch we consider the total population $\rho  = \sum_{i=1}^n x_i$ and the proportion (``frequency'') $\theta_i = x_i/\rho$ on each patch. In these new variables $(\rho,\theta)$ the system has nice properties. It turns out that the system of  $\theta$ variables  is non linear, but {\em independent} of $\rho$.
We  prove then  that it has a globally asymptotically stable periodic solution $\theta^*$  from which we deduce a  general expression for the global growth rate as an integral along $\theta^*$ (see Section \ref{ReductionSimplex}). Moreover, on this $\theta$ system we can apply Tikhonov's theorem from which we deduce our large-$T$  asymptotics of $\Lambda(m,T)$.

The possibility that $R(s)$ and $L(s)$ have discontinuities is not a simple desire for mathematical generality. 
 {Indeed, it opens the way to thinking, for example in the piecewise constant case, about stochastic models (PDMP) as we sketched in the Section \ref{SE}.
Note that even if we assume in Section \ref{SE} that $R$ and $L$ are continuous, this does not contradict the fact that in our deterministic study we found it important to extend the results to the non-continuous case, since for a realization $t \mapsto \omega_t$ of many reasonable stochastic processes, $t \mapsto R(\omega_t) +mL(\omega_t)$ is discontinuous. 
  We refer to \cite{benaim} and \cite{BLSSarXiv} for further analysis of these issues.}

As one sees  we recover all results of Katriel  \cite{Katriel} except 
an important one which {plays} a significant role in the proofs of the main results of \cite{Katriel}. It says   that for all $m>0$, the function $T\mapsto\Lambda(m,T)$ is strictly increasing (except in the case where all the $r_i$'s are equals) (see  \cite[Lemma 2]{Katriel}), from which the existence of the curve $T_c(m)$ is deduced. This result follows from general results of Liu et al. \cite{liu} on the principal eigenvalue of a periodic linear system. Indeed, the growth rate $\Lambda(m,T)$ can be seen as a principal eigenvalue of a linear periodic problem, to which the results of \cite{liu} apply, see \cite[Section 3.1]{Katriel}.
But, as we have shown in Section \ref{NCM}   the monotonicity of $T\mapsto \Lambda(m,T)$ is no longer true if $L(s)$ is not constant. We conjecture that it is true in the non symmetric constant case but we are unable to prove it, even in the constant and symmetric case, with our methods.

 As indicated in the title, in the present article we concentrated on the case where all patches are sinks: the ``all-sink case''. We have not considered the ``source-sink'' case (where some patches are sources) as in \cite{Katriel}. Our methods clearly apply to this case and we do not expect big surprises in this direction but the analysis remains to be done. Since the persistence in the case of density dependent models depends on the linearized system at the origin our results are pertinent in these cases. 
As the concepts of ``source'' and ``sink'' are also relevant for structured populations, it would be relevant to study the existence of DIG in this more complex case.

But before looking  to these generalizations it seems to us that a question pointed by an anonymous referee deserves prior attention: the assumption that for all $t$'s the migration matrix is irreducible is certainly not realized in many real systems. For instance on a two-patches system with two seasons we can imagine that during
the first season there is migration from patch one to patch two and conversely from patch two to one during season two, like migrating birds do between places in north or south. In this case the migration matrix $L(s)$ is not irreducible but DIG is still observed as we have shown on an example (see Section \ref{UDMig}). But, for more than two patches, it might append that DIG is not present as shown by the example in Section \ref{ExLobryNoDIG}. What are conditions for DIG when $L(s)$ is not irreducible ? This is a question worth asking.

\color{black}

\appendix
\section{The Perron Frobenius theorem}\label{PFT}
The Perron theorem, implies the following result.

\begin{theorem}\label{Ptheorem}
Let $X$ be a matrix with positive entries. 
The matrix $X$ has a unique positive real eigenvalue, denoted $\mu$ and a unique corresponding eigenvector, denoted $\pi$, called respectively the \emph{Perron root} and the \emph{Perron vector}, such that
\begin{equation}\label{PerronVector}
\textstyle
X\pi=\mu\pi,\quad
\quad \pi_i > 0
\quad\mbox{and}\quad\|\pi\|_1=\sum_{i=1}^n\pi_i=1.
\end{equation}
Moreover
\begin{equation}\label{thelimit}
\lim_{k\to\infty}(X/\mu)^k = G,
\end{equation}
where $G$ is the projector onto the null space 
$N(M)$ along the range $R(M)$ of the matrix $M=X-\mu I$. 

As a consequence of this limit, for ll $x\neq 0$ in $\mathbb{R}_+^n$ we have
\begin{equation}\label{Conv}
\lim_{k\to\infty}\frac{X^kx}{\|X^kx\|_1} = \pi.
\end{equation}
\color{black}
\end{theorem}

The matrix $G$ is called the \emph{Perron projection}. We have the explicit formula $G=vw^\top/w^\top v$, where $v$ and $w^\top$ are left and right positive eigenvectors of $X$, with eigenvalue $\mu$, i.e. $Xv=\mu v$ and $w^\top X=\mu w^\top$. 

{However, this explicit formula will not be used in this paper}. 
For details and complements, see  \cite{MeyerBook}.

The Perron Frobenius theorem extends the Perron theorem to irreducible matrices with nonnegative entries, see  \cite{MeyerBook}. This theorem implies the following result.

 \begin{theorem}\label{PFtheorem}
Let $A$ be an irreducible Metzler matrix (i.e. the matrix $A$ has off diagonal non-negative entries).
The matrix $A$ has a unique real eigenvalue, denoted $\lambda_{max}(A)$ and a unique corresponding eigenvector, denoted $u$, called respectively the \emph{Perron-Frobenius  root} and the \emph{Perron-Frobenius vector}, such that
\begin{equation}\label{PFVector}
\textstyle
Ap=\lambda_{max}p,\quad
\quad p_i > 0
\quad\mbox{and}\quad\|p\|_1=\sum_{i=1}^np_i=1.
\end{equation}
Moreover any other eigenvalue $\lambda$ of $A$ satisfies $\Re(\lambda)<\lambda_{max}(A)$. 
\end{theorem}

This result is obtained by applying the Perron-Frobenius theorem to the matrix $X=A+rId$ where $r$ is chosen such that $X$ has non-negative entries.

\section{Cooperative linear 1-periodic systems}
\label{CLDE}

We consider the linear differential equation
\begin{equation}\label{eq0}
\frac{dx}{dt}=A(t)x.
\end{equation}
We assume that \begin{hyp}\label{AssumptionA}
The function $A:t\mapsto A(t)$ is a 
{
piecewise continuous 1-periodic function}, 
with a finite number of discontinuities on $[0,1)$ and has
left and right limits at the  discontinuity points.
Moreover, for each $t\geq 0$, $A(t)$ is an irreducible Metzler matrix.
\end{hyp}

Hypothesis \ref{AssumptionA} implies that the solutions of \eqref{eq0} are continuous and piecewise $\mathcal{C}^1$ functions satisfying  \eqref{eq0} excepted on the discontinuity points of $A(t)$. Moreover, the positive cone is positively invariant for  \eqref{eq0}. More precisely, since 
$A(t)$ is cooperative and irreducible then, as a consequence of \cite[Theorem 1.1] {hirsch} or \cite{Slom}, we have the following result.

\begin{lem}\label{lemme1}
Suppose that $x:[0+\infty)\to \mathbb{R}^n$ is a solution of \eqref{eq0}
such that $x(0)>0$. Then $x(t)\gg 0$ for all $t>0$. 
\end{lem}

Recall that the solution $x(t,x_0)$ to \eqref{eq0} such that $x(0,x_0)=x_0$ writes 
\begin{equation}\label{flot}
x(t,x_0)=\Phi(t)x_0
\end{equation}
where $\Phi(t)$, called the {\emph{fundamental matrix solution}}, is the solution to the matrix valued differential equation
\begin{equation}\label{fundamental0}
\frac{dX}{dt}=A(t)X,\qquad X(0)=Id.
\end{equation}
From Lemma \ref{lemme1} we deduce that for all $t>0$, $\Phi(t)$ has positive entries.

Let 
$\Delta:=\textstyle\left\{x\in\mathbb{R}_+^n:\sum_{i=1}^nx_i=1\right\}$ 
be the unit $n-1$ simplex of $\mathbb{R}^n_+$. Every $x\neq 0$ in $\mathbb{R}_+^n$ can be written as
\begin{equation}\label{rhotheta1}
x=\rho\theta,\textstyle
\quad\mbox{with}\quad
\rho=\sum_{i=1}^nx_i
\quad\mbox{and}\quad
\theta=\frac{x}{\rho}\in\Delta.
\end{equation}
The flow \eqref{flot} of \eqref{eq0} induces a flow on $\Delta$, given by
\begin{equation}\label{flotPsi}
\Psi(t,\theta)=\frac{\Phi(t)\theta}{\langle \Phi(t)\theta,{\bf 1}\rangle}.
\end{equation}
We have the following result.

\begin{propo}\label{periodicPsi}
Let $\pi\in\Delta$ {be the Perron vector} of $\Phi(1)$.  
Then  $t\mapsto\Psi(t,\pi)$ is a periodic orbit in $\Delta$. It is globally asymptotically stable, i.e. for any $\theta\in\Delta$
$$\lim_{t\to\infty}\|\Psi(t,\theta)-\Psi(t,\pi)\|=0.$$
\end{propo}
\begin{proof}
The {Perron vector} $\pi$ of $\Phi(1)$ is a fixed point for the induced flow $\Psi(1,\theta)$ on $\Delta$. Indeed, using \eqref{PerronVector} and\eqref{flotPsi}, we have
$$\Psi(1,\pi)=\frac{\Phi(1)\pi}{\langle \Phi(1)\pi,{\bf 1}\rangle}=\frac{\mu\pi}{\mu\langle \pi,{\bf 1}\rangle}=\pi. $$ 
 Therefore  $\Psi(t,\pi)$ is a periodic orbit in $\Delta$. 
 
Recall that the periodic orbit of a continuous dynamical system is asymptotically stable if and only if the corresponding fixed point of the Poincar\'e map is asymptotically stable.
 Using \eqref{flotPsi} and \eqref{Conv} we have
 \begin{align*}
 \lim_{k\to\infty}\|\Psi(k,\theta)-\Psi(k,\pi)\|
 =\lim_{k\to\infty}\left\|\frac{\Phi(1)^k\theta}{{\|\Phi(1)^k\theta\|_1}}-\pi\right\|=0.
 \end{align*}
 Therefore the fixed point $\pi$ of the Poincar\'e map is GAS.  
 This proves the global asymptotic stability of the periodic solution $\Psi(t,\pi)$.
 \color{black}
\end{proof}
%%%%%%%%%%%%%%%%%%%%%%

Let us write the differential equation on $\Delta$ corresponding to the flow \eqref{flotPsi}. Using the decomposition \eqref{rhotheta1}, the differential equation \eqref{eq0}, with initial condition $x(0)>0$, rewrites:
\begin{align}
\label{eqrho}
\frac{d\rho}{dt}&=\langle A(t)\theta,
{\bf 1}\rangle \rho
\\
\label{eqtheta}
\frac{d\theta}{dt}&=A(t)\theta- \langle A(t)\theta,
{\bf 1}\rangle \theta
\end{align}
with initial conditions $\rho(0)=\langle x(0),
{\bf 1}\rangle$ and $\theta(0)=\frac{x(0)}{\langle x(0),
{\bf 1}\rangle}$. For any $\theta_0\in\Delta$, the solution $\theta(t)$ of of \eqref{eqtheta} with initial condition $\theta(0)=\theta_0$ is given by 
$\theta(t,\theta_0)=\Psi(t,\theta_0)$, where $\Psi$ is given by \eqref{flotPsi}. We can also express the solution $x(t,x_0)$ by using the solution $\theta(t,\theta_0)$, as shown in the following remark.

\begin{rem}\label{x=rhotheta}
Let $\theta(t,\theta_0)$ be the solution of \eqref{eqtheta} with initial condition $\theta_0$. The solution $x(t,x_0)$  of \eqref{eq0} with initial condition $x_0$ is given by
\begin{equation}\label{x(t,x0)}
x(t,x_0)=\theta\left(t,{x_0}/{\rho_0}\right)\rho_0e^{\int_0^t\langle A(s)\theta\left(s,{x_0}/{\rho_0}\right),{\bf 1}\rangle ds},
\end{equation}
where $\rho_0=\langle x_0,{\bf 1}\rangle$.
\end{rem}

We have the following result which is the particular case of 
Proposition \ref{periodicPsi} when the matrix $A$ is constant. We state it here because it is used several times in the proofs of our results, see Sections \ref{HFL}, \ref{LFL}, and \ref{FM}.

\begin{propo}\label{periodicPsi1}
Let $A$ be an irreducible Metzler matrix. Let $\lambda_{max}(A)$ be its spectral abscissa. Let $p\in\Delta$ be the Perron Frobenius vector associated to $\lambda_{max}(A)$. Then $p$ is an equilibrium point of the differential equation 
\begin{equation}\label{eqsimplex}
\frac{d\theta}{dt}=A\theta- \langle A\theta,{\bf 1}\rangle\theta
\end{equation}
on the simplex $\Delta$ associated to the autonomous linear equation $\frac{dx}{dt}=Ax$. It is GAS in the simplex $\Delta$.
\end{propo}

We have the following result, which asserts the existence of the growth rate of \eqref{eq0} and which gives us two formulas to calculate it. One uses the periodic solution whose existence was given in Proposition \ref{periodicPsi}. The other uses the Perron root of the monodromy matrix of \eqref{eq0}.

\begin{theorem}\label{Lambda(x)}
Let $\Lambda=\ln(\mu)$, where $\mu$ is the Perron root of the monodromy matrix $\Phi(1)$ of \eqref{eq0}. Let $\pi$ its Perron vector. Let $\theta^*(t):=\Psi(t,\pi)$ is the periodic solution of \eqref{eqtheta} whose existence and global asymptotic stability are proved in Proposition \ref{periodicPsi}.
For any solution $x(t)$ of \eqref{eq0}, such that $x(0)>0$, we have
\begin{equation*}
\lim_{t\to\infty}\frac{1}{t}\ln(x_i(t))=\int_0^1\langle A(t)\theta^*(t),{\bf 1}\rangle dt=\Lambda,
\end{equation*}
\end{theorem}
\begin{proof}
Using \eqref{x(t,x0)},  the Lyapunov exponent of the components of any solution $x(t,x_0)$ of \eqref{eq0} can be computed as follows
\begin{align*}
\lim_{t\to\infty}\frac{1}{t}\ln(x_i(t,x_0))&=\lim_{k\to\infty}\frac{1}{k}\ln(x_i(k,x_0))\\
&=
\lim_{k\to\infty}\frac{1}{k}\left[
\ln(\theta_i(k,x_0/\rho_0)\rho_0)+
\int_0^{k}U(s) ds\right],
\end{align*}
where $U(s)=\langle A(s)\theta\left(s,{x_0}/{\rho_0}\right),{\bf 1}\rangle$.
Since $\|\theta\left(t,{x_0}/{\rho_0}\right)-\theta^*(t)\|$ tends to 0, as $t$ tends to $\infty$, the first term in the right hand side goes to 0. Therefore, for all $k_1\geq 0$,
$$
\lim_{t\to\infty}\frac{1}{t}\ln(x_i(t,x_0))=
\lim_{k\to\infty}\frac{1}{k}\left[
\int_0^{k_1}U(s) ds+\int_{k_1}^{k}U(s) ds\right]
$$
Using Proposition \ref{periodicPsi}, for $k_1$ large enough, we can replace in the second integral $\theta\left(t,{x_0}/{\rho_0}\right)$ by $\theta^*(t)$, and then, using the fact that the first term tends to 0 as $k\to\infty$, we have 
\begin{align*}
\lim_{t\to\infty}\frac{1}{t}\ln(x_i(t,x_0))&=
\lim_{k\to\infty}\frac{1}{k}\int_{k_1}^{k}
\langle A(s)\theta^*(s),{\bf 1}\rangle ds\\
&=\lim_{k\to\infty}\frac{k-k_1}{k}\int_{0}^{1}
\langle A(s)\theta^*(s),{\bf 1}\rangle ds
=\int_0^1\langle A(t)\theta^*(s),{\bf 1}\rangle ds.
\end{align*}
This proves the first equality. The 
second equality is proved as follows.
 Let $x(t)=\Phi(t)\pi$ be the solution of  \eqref{eq0}, with initial condition $x(0)=\pi$, where $\pi\in\Delta$ is the Perron vector of  $\Phi(1)$. Since  $x(1)=\Phi(1)\pi=\mu\pi$, we have $x(k)=\Phi(k)\pi=\Phi(1)^k\pi=\mu^k\pi$. Hence
 $$\lim_{k\to\infty}\frac{1}{k}\ln(x_i(k))=\ln(\mu)=\Lambda,$$
 which proves the formula, since all components of all solutions $x(t)$ of \eqref{eq0}, such that $x(0)>0$, have the same Lyapunov exponents. 
 \end{proof}

\begin{rem}
In the Floquet theory of linear periodic systems (see \cite[Section 4.6]{Hartman}), $\ln(\mu)$ is known as the \emph{largest Floquet exponent}, or the \emph{principal Lyapunov exponent}, that is, the 
{characteristic} 
multiplier corresponding to
the dominant eigenvalue   
$\mu$ of $\Phi(1)$. For further details, we refer the reader to \cite{Carmona} and \cite[Section II.2]{Mierczynski}. 
 \end{rem}

\section{Tikhonov's theorem}\label{SecTikhonov}
Tikhonov's theorem \cite{tikhonov} provides a mathematically rigorous basis for the quasi-steady-state approximation commonly used in the study of systems at several time scales \cite{Noethen,Schneider}. 
We consider the singularly perturbed initial value problem
\begin{equation}\label{LR}
\left\lbrace 
\begin{array}{rclcl}
\varepsilon\frac{dx}{d\tau}&=&f(\tau,x,y,\varepsilon),&&x(\tau_0)=x^0(\varepsilon),\\[1mm]
\frac{dy}{d\tau}&=&
g(\tau,x,y,\varepsilon)&&y(\tau_0)=y^0(\varepsilon),
\end{array}\right. 
\end{equation}
for an $m$ vector $x$ and an $n$ vector $y$ on some bounded interval, say $\tau_0\leq \tau\leq \tau_1$, where $\varepsilon$ is a small positive parameter, $0<\varepsilon\ll 1$. We assume that 

\begin{description}
\item[H0] The functions $f$ and $g$ are continuous in $\tau\in[\tau_0,\tau_1]$.
\item[H1] The functions $f$, $g$, $x^0$ and $y^0$ are continuous in $\varepsilon$. The functions $f$ and $g$ are differentiable in their $x$ and $y$ arguments. 
\end{description} 

If assumptions H0 and H1 are satisfied, then the initial value problem \eqref{LR} has a unique solution, denoted $x(\tau,\varepsilon)$, $y(\tau,\varepsilon)$.
When $\varepsilon\to 0$, \eqref{LR} is a \textit{slow-fast} system, with
$m$ \textit{fast variables} $x$,
and $n+1$ \textit{slow variables}, $y$ and $\tau$.
According to Tikhonov's theory, the so called \emph{fast equation} is
 \begin{equation}\label{FEq}
 \frac{dx}{dt}=f(\tau,x,y,0)
 \end{equation}
where $\tau$ and $y$ are considered as parameters. Note that this equation is obtained by replacing $\varepsilon$  by 0 in the right hand side of the system
\begin{equation}\label{LR1}
 \left\lbrace 
\begin{array}{rcl}
\frac{dx}{dt}&=&f(\tau,x,y,\varepsilon),\\[1mm]
\frac{dy}{dt}&=&
\varepsilon g(\tau,x,y,\varepsilon)\\[1mm]
\frac{d\tau}{dt}&=&\varepsilon.
\end{array}\right. 
\end{equation}
which is equivalent to the slow-fast system in \eqref{LR}, written 
with the time $t=\tau/\varepsilon$.
We refer to $\tau$ as the \emph{slow time} and to $t$ as the \emph{fast time}. 
We assume that 
\begin{description}
\item[H2] 
For any $y$ in a compact set $K$ and $\tau\in[\tau_0,\tau_1]$, the fast equation \eqref{FEq} has an equilibrium $x= \xi(\tau,y)$, which is asymptotically stable with a basin of attraction that is uniform in the parameters $(\tau,y)\in[0,1]\times K$. The function $\xi$ is continuous in $\tau$ and differentiable in $y$.
\end{description} 

The \emph{critical manifold}, also called \emph{slow manifold}, is the set of equilibrium points $x= \xi(\tau,y)$ of the fast equation \eqref{FEq}. 
 The \emph{reduced equation}, defined for $(\tau,y)\in[\tau_0,\tau_1]\times K$,
\begin{equation}
\label{REq}
\frac{dy}{d\tau}=
g(\tau,\xi(\tau,y),y,0),\qquad y(t_0)=y^0(0)
\end{equation}  
is obtained by replacing $\varepsilon$ by 0 and $x$ by $\xi(\tau,y)$ in the second equation of \eqref{LR}. Since the function $\xi$ is continuous in $\tau$ and differentiable in $y$, the equation \eqref{REq} is well defined.
Tikhonov's theorem states that the solution of \eqref{LR} jumps quickly near the critical manifold and is then approximated by the solution of \eqref{REq}. More precisely:

\begin{theorem}[Tikhonov's theorem] Assume that H0, H1 and H2 are satisfied. Assume that $y^0(0)\in K$ and  $x^0(0)$ belongs to the basin of attraction of $\xi(\tau_0,y^0(0))$.
Let 
$\overline{y}(\tau)$ be the solution of \eqref{REq}, which is assumed to exist on the interval $[\tau_0,\tau_1]$. 
Let $\nu>0$. For $\varepsilon$ small enough, $x(\tau,\varepsilon)$ and $y(\tau,\varepsilon)$ are defined on $[\tau_0,\tau_1]$ and, as  $\varepsilon\to 0$
$$
\begin{array}{ll}
y(\tau,\varepsilon)=\overline{y}(\tau)+o(1)
&\mbox{ uniformly on }\quad [\tau_0,\tau_1],\\ 
x(\tau,\varepsilon)=\xi(\tau,\overline{y}(\tau))+o(1)
&\mbox{ uniformly on }\quad  [\tau_0+\nu,\tau_1].
\end{array} 
$$
\end{theorem}

\begin{rem}
The approximation given by Tikhonov's theorem holds for all $\tau\in[\tau_0,\tau_1]$ for the slow variable $y(\tau,\varepsilon)$ and for all $\tau\in[\tau_0+\nu,\tau_1]$ for the fast variables $x(\tau,\varepsilon)$, where $\nu$ is as small as we want. Indeed we have a \emph{boundary layer} at $\tau=\tau_0$ since the fast variables jumps quickly from their initial conditions $x^0(\varepsilon)$ near the point 
$\xi(\tau_0,y_0(0))$ of the slow manifold.
\end{rem}

{This theorem was first stated by Tikhonov \cite{tikhonov}  and can be found in various forms in the classical literature, see the book by O'Malley \cite[Section 2.D]{O'Malley} and the book by Wasow \cite[Section X.39]{wasow}.
For a statement of Tikhonov's theorem that is very close to the one given here, the reader can refer to \cite{LST} or to Khalil's book \cite[Theorem 11.1]{Khalil}. The book by Banaisak and Lachowicz \cite[Chapter 3]{Banasiak} is also a highly recommended reference for the reader interested by applications in mathematical biology.
This result remains true under less restrictive conditions on the regularity of $f$ and $g$, see \cite{LST,wasow}.
This result was extended by Fenichel \cite{Fenichel} in the context of \emph{Geometric Singular Perturbation Theory}. See also \cite[Chapter 3]{Kuehn}.} 

We want to use Tikhonov's theorem when the functions $f$ and $g$ have discontinuities in the variable $\tau$. Assume that 

\begin{description}
\item[H0'] There exist a finite set 
$D=\{\tau_k, 1\leq k\leq p:  0< \tau_1<\cdots<\tau_p<1\}$ such that $f$ and $g$ are continuous on $[0,1]\setminus D$ and have right and left limits at the discontinuity points $\tau_k\in(0,1)$, $k=1,\ldots,p$.   
\end{description}
We extend $f$ by its right limit at each discontinuity point, so that the fast equation \eqref{FEq} is defined for all $\tau\in[0,1]$. We assume that
\begin{description}
\item[H2'] 
For all $y$ in some compact set $K$ and $\tau\in[0,1]$, the fast equation \eqref{FEq} 
has an asymptotically stable equilibrium 
$x= \xi(\tau,y)$ 
with a basin of attraction that
is uniform in the parameters 
$(\tau, y) \in[0,1]\times K$. The function $\xi$ is continuous for $\tau\in [0,1]\setminus D$ and differentiable in $y\in K$ and has left and right limits at the discontinuity points $\tau_k$ denoted by 
\begin{equation}\label{LRlimits}
\xi(\tau_k-0,y)=\lim_{\tau\to\tau_k,\tau<\tau_k}\xi(\tau,y),\quad
\xi(\tau_k+0,y)=\lim_{\tau\to\tau_k,\tau>\tau_k}\xi(\tau,y).
\end{equation}
Moreover, we assume that 
for all $y\in K$, $\xi(\tau_k+0,y)$ is an asymptotically stable equilibrium for the fast equation \eqref{FEq}, where $\tau=\tau_k$, and that $\xi(\tau_{k+1}-0,y)$ belongs to its basin of attraction.
\end{description} 

Since the function $\xi$ is continuous for $\tau\in [0,1]\setminus D$ and differentiable in $y\in K$, the equation \eqref{REq} is well defined on $[0,1]\times K$. 

As a consequence of Tikhonov's theorem, we have the following result, which is used in the proofs of items 2 and 4 of Theorem \ref{thm1}, see Sections  \ref{LFL} and \ref{FM}.

\begin{propo} 
\label{PropTikhonov}
Assume that H0', H1 and H2' are satisfied. Assume that $y^0(0)\in K$ and  $x^0(0)$ belongs to the basin of attraction of $\xi(\tau_0,y^0(0))$.
Let 
$\overline{y}(\tau)$ be the solution of \eqref{REq}, which is assumed to exist on the interval $[0,1]$. 
Let $\nu>0$. For $\varepsilon$ small enough, $x(\tau,\varepsilon)$ and $y(\tau,\varepsilon)$ are defined on $[0,1]$ and, as  $\varepsilon\to 0$
$$
\begin{array}{ll}
y(\tau,\varepsilon)=\overline{y}(\tau)+o(1)
&\mbox{ uniformly on }\quad [0,1],\\ 
x(\tau,\varepsilon)=\xi(\tau,\overline{y}(\tau))+o(1)
&\mbox{ uniformly on }\quad  [0,1]\setminus \bigcup_{k=0}^p[\tau_k,\tau_k+\nu].
\end{array} 
$$
where $\tau_0=0$ and $\tau_k$, $1\leq k\leq p$, are the discontinuity points of $f$ and $g$.
\end{propo}
\begin{proof}
For each $0\leq k\leq p$, we consider the system \eqref{LR} on the interval $[\tau_k,\tau_{k+1}]$, where $\tau_{p+1}=1$. We extend the functions $f$ and $g$ by continuity to $\tau_k$ and $\tau_{k+1}$ (such extensions exist because, according to H0', the functions have left and right limits at the discontinuity points $\tau_k$). This system satisfies assumptions H0, H1 and H2. 
For the first interval ($k=0$), we use the initial condition $x(0)=x^0(\varepsilon)$ and $y(0)=y^0(\varepsilon)$. We obtain an approximation on the interval $[\tau_0,\tau_1]$. For the second interval ($k=1$), we use as initial conditions  $x(\tau_1,\varepsilon)$ and $y(\tau_1,\varepsilon)$, which, according to the approximation obtained in the first interval, are close to $\xi(\tau_1-0,\overline{y}(\tau_1))$ and $\overline{y}(\tau_1)$, respectively. 
Since $\xi(\tau_1+0,\overline{y}(\tau_1))$ is an asymptotically stable equilibrium for the fast equation \eqref{FEq}, and $\xi(\tau_1-0,\overline{y}(\tau_1))$ belongs to its basin of attraction, Tikhonov's theorem can be used on the interval $[\tau_1,\tau_2]$. Similarly for the following intervals. Therefore, on each interval, for any $\nu>0$, as $\varepsilon\to 0$,
$$
\begin{array}{ll}
y(\tau,\varepsilon)=\overline{y}(\tau)+o(1)
&\mbox{ uniformly on }\quad [\tau_k,\tau_{k+1}],\\ 
x(\tau,\varepsilon)=\xi(\tau,\overline{y}(\tau))+o(1)
&\mbox{ uniformly on }\quad  [\tau_k+\nu,\tau_{k+1}].
\end{array} 
$$
This ends the proof of the proposition.
\end{proof}
\begin{rem}\label{Rem12}
In addition to the boundary layer  at $\tau=0$, we have now an \emph{inner layer} at each discontinuity point $\tau_k$, because the fast variables must jump quickly from a point close to
$\xi(\tau_k-0,y(\tau_k))$ to a point close to $\xi(\tau_k+0,y(\tau_k))$, where the left and right limits are defined by \eqref{LRlimits}. These behaviors are illustrated in the
supplementary material Figure S6 for the limit $T\to\infty$.
\end{rem}

\begin{rem}\label{Noslowvariable}
Consider the special case where there is no slow variable $y$ in \eqref{LR}, that is to say we have a singularly perturbed system of the form
$$\varepsilon\frac{dx}{d\tau}=f(\tau,x,\varepsilon),$$
where $\tau$ is the only slow variable.
The critical manifold is a curve (also called the slow curve) $x=\xi(\tau)$, where $\xi(\tau)$ is the equilibrium of the fast dynamics  $\frac{dx}{dt}=f(\tau,x,0)$. In this case, the fast variable $x(\tau,\varepsilon)$ is approximated by the slow curve, i.e.  the result of Proposition \ref{PropTikhonov} becomes
$x(\tau,\varepsilon)=\xi(\tau)+o(1)$
uniformly on each interval $[\tau_k+\nu,\tau_{k+1}]$.

\end{rem}

\section*{Acknowledgments} We warmly thank  the two anonymous reviewers
whose constructive remarks have produced a significantly
improved version of our article.

{}

\pagebreak
\begin{center}
{\Large\bf Supplementary material}
\end{center}

\renewcommand{\thesection}{S\arabic{section}}  
\renewcommand{\thetable}{S\arabic{table}}  
\renewcommand{\thefigure}{S\arabic{figure}} 
\renewcommand{\figurename}{Supplemental Material, Figure} 
\renewcommand{\tablename}{Supplemental Material, Table}
 \setcounter{section}{0}
 \setcounter{figure}{0}
 \setcounter{table}{0}

\section{Time-independent asymmetric migration}\label{Example3sites}

Our objective in this section is to show numerical simulations with three patches that illustrate all our findings and also corroborate  Conjecture 1. We consider the three patch case with time independent migration given by $\ell_{12}=\ell_{23}=\ell_{31}=0$ and $\ell_{21}=\ell_{32}=\ell_{13}=1$, which corresponds to a circular migration $1\to 2\to 3\to 1$. We consider the piecewise constant growth rates
$$
\begin{array}{c}
r_1(\tau)=\left\{\begin{array}{rcl}
0.15&\mbox{if}&0\leq \tau< 1/2\\
-0.45&\mbox{if}&1/2\leq \tau< 1
\end{array}
\right.,\quad
r_2(\tau)=\left\{\begin{array}{rcl}
-0.45&\mbox{if}&0\leq \tau< 1/2\\
0.15&\mbox{if}&1/2\leq \tau< 1
\end{array}
\right..\\[5mm]
r_3(\tau)=-0.2\mbox{ for }\tau\in[0,1].
\end{array}
$$
We have
$\chi=0.15$, $\overline{r}_1=\overline{r}_2=-0.15$ and $\overline{r}_3=-0.2$. Therefore, the patches are sinks and DIG occurs. 
In this case the the monodromy matrix is $\Phi(m,T)=e^{\frac{T}{2}B}e^{\frac{T}{2}A}$ where the matrices 
$A$ and $B$ 
are given by:
$$
A=\left[
\begin{array}{ccc}
0.15-m&0&m\\
m&-0.45-m&0\\
0&m&-0.2-m
\end{array}
\right],
$$
$$
B=\left[
\begin{array}{ccc}
-0.45-m&0&m\\
m&0.15-m&0\\
0&m&-0.2-m
\end{array}
\right].
$$
We show in Figure \ref{fig6}(a) the plot of $\Lambda(m,T)=\frac{1}{T}\ln(\mu(m,T))$ where $\mu(m,T)$ is the Perron root of the monodromy matrix $\Phi(m,T)$. The plot of the graphs of the functions $(m, T)\mapsto \Lambda(m, T)$, $T\mapsto \Lambda(m,T)$, for $m$ fixed,  $m\mapsto \Lambda(m,T)$, for $T$ fixed, and also the critical set where $\Lambda(m,T)=0$ are shown Figure \ref{fig6}(c,d,b).

 \begin{figure}[ht]
\begin{center}
\includegraphics[width=10cm,
viewport=150 350 430 690]{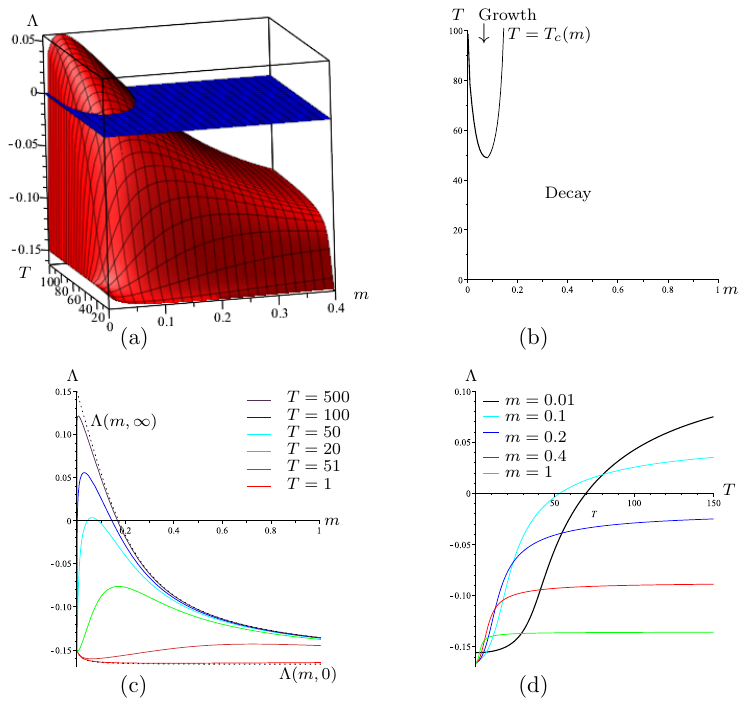}
\caption{(a) The graph of $(m,T)\mapsto \Lambda(m,T)$.  (b) The set $\Lambda(m,T)=0$. (c) Graphs of $m\mapsto \Lambda(m,T)$ with the indicated values of $T$. (d) Graphs of $T\mapsto \Lambda(m,T)$ with the indicated values of $m$. We have $\Lambda(m,\infty)=0$ for $m=m^*=0.172$.}	\label{fig6}	
\end{center}
\end{figure}

Using Remark 3, we 
have $p_1=p_2=p_3=1/3$. 
Using the theoretical formulas in Theorem 4, we have 
$$\begin{array}{l}
\Lambda(0,0)=\Lambda(0,T)=-0.15,\quad
\Lambda(0,\infty)=0.15,\\ 
\Lambda(\infty,0)=\Lambda(\infty,T)=
\Lambda(\infty,\infty)=-1/6.
\end{array}
$$
We can make the following comments on Figure \ref{fig6}: The strictly decreasing functions $m\mapsto \Lambda(m,0)$ and 
$m\mapsto \Lambda(m,\infty)$, are depicted in dotted line on panel (c) of the figure. 
Panel (d) of the figure shows that 
for all $m>0$, the functions $T\mapsto\Lambda(m,T)$ are strictly increasing, supporting Conjecture 1. Hence, there exists a critical curve $T=T_c(m)$ defined for $0<m<m^*$ such that $T_c(0)=T_c(m^*)=\infty$ and DIG occurs if and only if $T>T_c(m)$, as depicted in panel (b) of the figure. Panel (c) of the Figure shows the graphs of functions $m\mapsto\Lambda(m,T)$ and illustrates their convergence 
toward $\Lambda(m,0)$ and $\Lambda(m,\infty)$ as $T$ tends to 0 and $\infty$, respectively.
Notice that for $0<T<\infty$, the functions $m\mapsto\Lambda(m,T)$ are not monotonic.

\begin{figure}[ht]
\begin{center}
\includegraphics[width=10cm,
viewport=160 350 440 680]{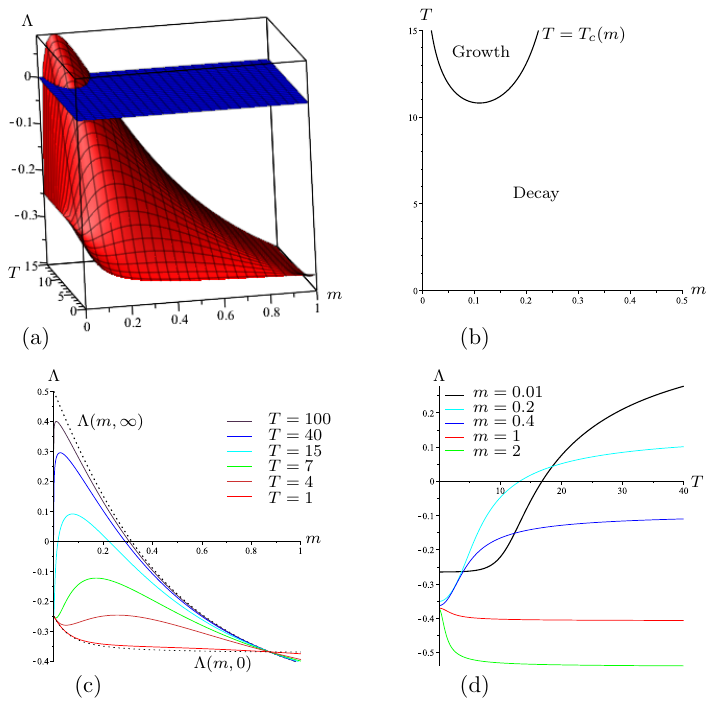}
\caption{(a) The graph of $(m,T)\mapsto \Lambda(m,T)$.  (b) The set $\Lambda(m,T)=0$. (c) Graphs of $m\mapsto \Lambda(m,T)$ with the indicated values of $T$. (d) Graphs of $T\mapsto \Lambda(m,T)$ with the indicated values of $m$. Here we used the two patch model corresponding to the matrices \eqref{AB2}.}	\label{fig3}	
\end{center}
\end{figure}

\section{Time-dependent migration} 
\subsection{The function $T\mapsto \Lambda(m,T)$ is not always increasing}

We show in Figure \ref{fig3}(a) the plot of $\Lambda(m,T)=\frac{1}{T}\ln(\mu(m,T))$ where $\mu(m,T)$ is the Perron root of the monodromy matrix $\Phi(m,T)=e^{\frac{T}{2}B}e^{\frac{T}{2}A}$,
in the case where the matrices 
$A$ and $B$ 
are given by
\begin{equation}\label{AB2}
A=\left[
\begin{array}{cc}
1/2-2m&m\\
2m&-3/2-m
\end{array}
\right],
\quad
B=\left[
\begin{array}{cc}
-1-m&2m\\
m&1/2-2m
\end{array}
\right].
\end{equation}
The migration is time dependent.
We have
$\chi=1/2$, $\overline{r}_1=-1/4$ and $\overline{r}_2=-1/2$.
Therefore, the patches are sinks and DIG can occur.

 Using Remark 2, we 
have
$q_1=q_2=1/2$ and 
 $$
 p_{1}(\tau)=\left\{\begin{array}{l}
1/3\mbox{ if }0\leq \tau< 1/2\\
2/3\mbox{ if }1/2\leq \tau< 1\\
\end{array}
\right.,
\quad
p_2(\tau)=\left\{\begin{array}{l}
2/3\mbox{ if }0\leq \tau<1/2\\
1/3\mbox{ if }1/2\leq \tau< 1\\
\end{array}
\right..
$$
Using the theoretical formulas in Theorem 4 and Proposition 12, we obtain the expressions shown in Table \ref{Ex2}.
%%%%%%%%%%%%%%%%%%
\begin{table}
\caption{Limits of $\Lambda(m,T)$ for the parameters values used in Figure \ref{fig3}}\label{Ex2}
\begin{center}
\begin{tabular}{l}
\hline
$\Lambda(0,\infty)=1/2$,\quad
$\Lambda(0,0)=\Lambda(0,T)=-1/4$\\ 
$\Lambda(\infty,0)=-3/8$,\quad 
$\Lambda(\infty,T)=
\Lambda(\infty,\infty)=-2/3$
\\
\hline
$\Lambda(m,0)=-\frac{3}{8}-\frac{3}{2}m+\frac{1}{8}\sqrt{1+144m^2}$\\[1mm]
$\Lambda(m,\infty)=-\frac{3}{8}-\frac{3}{2}m+\frac{1}{4}\sqrt{4-4m+9m^2}+\frac{1}{8}\sqrt{9-12m+36m^2}$\\[1mm]
$\Lambda(m,\infty)=0$ for $m=m^*\approx 0.315$.\\
\hline
\end{tabular}
\end{center}
\end{table}

Since $\Lambda(\infty,T)<\Lambda(\infty,0)$, 
for $m$ large enough, the condition 
$\Lambda(m,T)>\Lambda(m,0)$, established in Proposition 5 for time-independent migration is not satisfied. Therefore,  the function $T\mapsto \Lambda(m,T)$ is not increasing for $m$ large enough and Conjecture 1 is not true for time dependent migration. 

We can make the following comments on Figure \ref{fig3}: 
The strictly decreasing functions $m\mapsto \Lambda(m,0)$ and 
$m\mapsto \Lambda(m,\infty)$, are depicted in dotted line on panel (c) of the figure. 
Panel (d) of the figure shows that  there exists a threshold value $m_0<m^*$ such that the function  $T\mapsto \Lambda(m,T)$ is increasing for $0<m<m_0$, constant for $m=m_0$ and decreasing for $m>m_0$. Hence, there exists a critical curve $T=T_c(m)$ defined for $0<m<m^*$ such that $T_c(0)=T_c(m^*)=\infty$ and DIG occurs if and only if $T>T_c(m)$, as depicted in panel (b) of the figure. Panel (c) of the Figure shows the graphs of functions $m\mapsto\Lambda(m,T)$ and illustrates their convergence 
toward $\Lambda(m,0)$ and $\Lambda(m,\infty)$ as $T$ tends to 0 and $\infty$, respectively.
Notice that for $0<T<\infty$, the functions $m\mapsto\Lambda(m,T)$ are not monotonic.

\begin{figure}[ht]
\begin{center}
\includegraphics[width=10cm,
viewport=160 350 440 680]{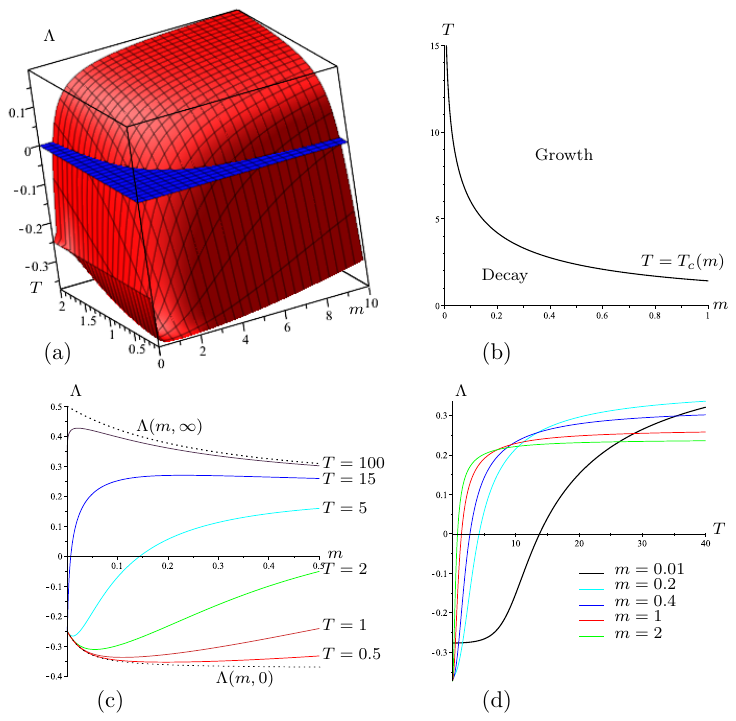}
\caption{(a) The graph of $(m,T)\mapsto \Lambda(m,T)$.  (b) The set $\Lambda(m,T)=0$. (c) Graphs of $m\mapsto \Lambda(m,T)$ with the indicated values of $T$. (d) Graphs of $T\mapsto \Lambda(m,T)$ with the indicated values of $m$. Here we used the two patch model corresponding to the matrices \eqref{AB2}.}	\label{fig4}	
\end{center}
\end{figure}

\subsection{DIG occurs for all $m>0$}

We show in Figure \ref{fig4}(a) the plot of $\Lambda(m,T)=\frac{1}{T}\ln(\mu(m,T))$ where $\mu(m,T)$ is the Perron root of the monodromy matrix $\Phi(m,T)=e^{\frac{T}{2}B}e^{\frac{T}{2}A}$,
in the case where the matrices 
$A$ and $B$ 
are given  by
$$
A=\left[
\begin{array}{cc}
1/2-m&5m\\
m&-3/2-5m
\end{array}
\right],
\quad
B=\left[
\begin{array}{cc}
-1-5m&m\\
5m&1/2-m
\end{array}
\right].
$$
The migration is time dependent.
We have
$\chi=1/2$, $\overline{r}_1=-1/4$ and $\overline{r}_2=-1/2$.
Therefore, the patches are sinks and DIG can occur.

 Using Remark 2, we 
have
$q_1=q_2=1/2$ and 
 $$
 p_{1}(\tau)=\left\{\begin{array}{l}
5/6\mbox{ if }0\leq \tau< 1/2\\
1/6\mbox{ if }1/2\leq \tau< 1\\
\end{array}
\right.,
\quad
p_2(\tau)=\left\{\begin{array}{l}
1/6\mbox{ if }0\leq \tau<1/2\\
5/6\mbox{ if }1/2\leq \tau< 1\\
\end{array}
\right..
$$
Using the theoretical formulas in Theorem 4 and Proposition 12, we obtain the expressions shown in Table \ref{Ex2m_star_infini}. 
%%%%%%%%%%%%%%%%%%
\begin{table}
\caption{Limits of $\Lambda(m,T)$ for the parameters values used in Figure \ref{fig4}}\label{Ex2m_star_infini}
\begin{center}
\begin{tabular}{l}
\hline
$\Lambda(0,\infty)=1/2$,\quad
$\Lambda(0,0)=\Lambda(0,T)=-1/4$\\ 
$\Lambda(\infty,0)=-3/8$,\quad 
$\Lambda(\infty,T)=
\Lambda(\infty,\infty)=5/24$
\\
\hline
$\Lambda(m,0)=-\frac{3}{8}-{3}m+\frac{1}{8}\sqrt{1+576m^2}$\\[1mm]
$\Lambda(m,\infty)=-\frac{3}{8}-{3}m+\frac{1}{2}\sqrt{1+4m+9m^2}+\frac{1}{8}\sqrt{9+48m+144m^2}$\\[1mm]
$\Lambda(m,\infty)>0$ for all $m\geq 0$.\\
\hline
\end{tabular}
\end{center}
\end{table}
Since for any $T>0$, $\Lambda(\infty,T)>0$,  
for fixed $T$ and $m$ large enough, the condition 
$\Lambda(m,T)>0$ is satisfied, so that DIG occurs for all $m>0$.
 
We can make the following comments on Figure \ref{fig4}: The strictly decreasing functions $m\mapsto \Lambda(m,0)$ and 
$m\mapsto \Lambda(m,\infty)$, are depicted in dotted line on panel (c) of the figure. 
Panel (d) of the figure shows that 
for all $m>0$, the functions $T\mapsto\Lambda(m,T)$ are strictly increasing. Hence, there exists a critical curve $T=T_c(m)$ defined for $m>0$ such that $T_c(0)=\infty$ and DIG occurs if and only if $T>T_c(m)$, as depicted in panel (b) of the figure. Panel (c) of the Figure shows the graphs of functions $m\mapsto\Lambda(m,T)$ and illustrates their convergence 
toward $\Lambda(m,0)$ and $\Lambda(m,\infty)$ as $T$ tends to 0 and $\infty$, respectively.
Notice that for $0<T<\infty$, the functions $m\mapsto\Lambda(m,T)$ are not monotonic. Therefore, in contrast with Figures \ref{fig6} and \ref{fig3}, the critical curve $T=T_c(m)$ is defined for all $m>0$ and DIG occurs if and only if $T>T_c(m)$.

\begin{figure}[ht]
\begin{center}
\includegraphics[width=10cm,
viewport=160 360 440 680]{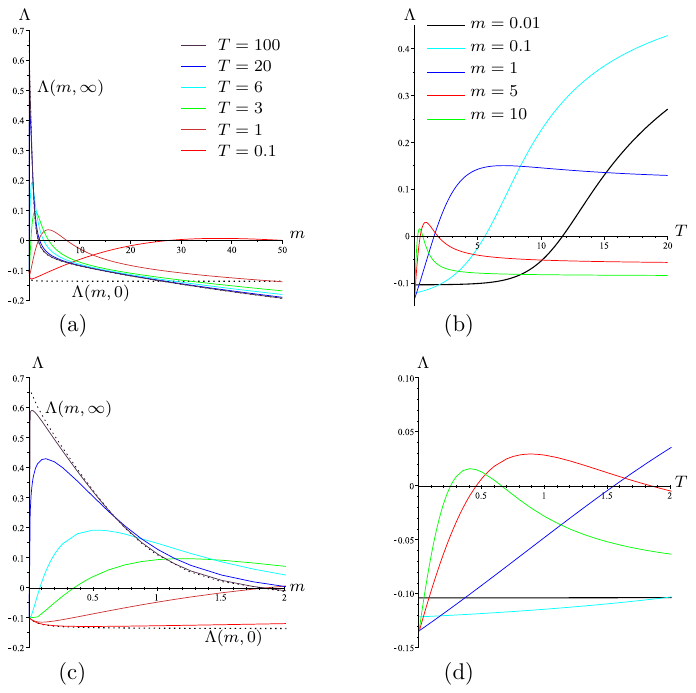}
\caption{(a) Graphs of $m\mapsto \Lambda(m,T)$, with the indicated values of $T$. (b) Graphs of $T\mapsto \Lambda(m,T)$ with the indicated values of $m$. (c) and (d) are zooms of (a) and (b) respectively. 
The parameter values correspond to \eqref{ABCmatrix}.}	\label{fig8}	
\end{center}
\end{figure}

\subsection{DIG can also occur for $m>m^*$}
\label{DIGm>m*}
\subsubsection{The two patch case}

We give more information on  $\Lambda(m,T)=\frac{1}{T}\ln(\mu(m,T))$ where $\mu(m,T)$ is the Perron root of the monodromy matrix 
$
\Phi(T)=e^{\frac{T}{3}C}e^{\frac{T}{3}B}e^{\frac{T}{3}A},
$
where the matrices $A$, $B$ and $C$ are defined by
\begin{equation}
\label{ABCmatrixS}
\begin{array}{c}
A=\left[
\begin{array}{rr}
-m&\frac{m}{10}\\
m&-\frac{1}{10}-\frac{m}{10}
\end{array}
\right],
\quad
B=\left[
\begin{array}{rr}
-\frac{4}{5}-\frac{m}{5}&{2m}\\[1mm]
\frac{m}{5}&\frac{3}{2}-2m
\end{array}
\right],\\[4mm]
C=\left[
\begin{array}{rr}
\frac{1}{2}-\frac{m}{100}&\frac{m}{100}\\
\frac{m}{100}&-2-\frac{m}{100}
\end{array}
\right].
\end{array}
\end{equation}
Using Remark 2, we 
have
 $q_1=\frac{211}{332}$, $q_2=\frac{121}{332}$ and 
 $$
 p_{1}(\tau)=\left\{\begin{array}{r}
\frac{1}{11}\mbox{ if }0\leq \tau< \frac{1}{3}\\[1mm]
\frac{10}{11}\mbox{ if }\frac{1}{3}\leq \tau< \frac{2}{3}\\[1mm]
\frac{1}{2}\mbox{ if }\frac{2}{3}\leq \tau< 1
\end{array}
\right.,
\quad
p_2(\tau)=\left\{\begin{array}{r}
\frac{10}{11}\mbox{ if }0\leq \tau< \frac{1}{3}\\[1mm]
\frac{1}{11}\mbox{ if }\frac{1}{3}\leq \tau< \frac{2}{3}\\[1mm]
\frac{1}{2}\mbox{ if }\frac{2}{3}\leq \tau< 1
\end{array}
\right..
$$
Using the theoretical formulas in Theorem 4 and Proposition 12, we obtain the expressions shown in Table \ref{Ex3}.
%%%%%%%%%%%%%%%%%%
\begin{table}
\caption{Limits of $\Lambda(m,T)$ for the parameters values in \eqref{ABCmatrixS}}\label{Ex3}
\begin{center}
\begin{tabular}{l}
\hline
$\Lambda(0,\infty)=2/3$,\quad
$\Lambda(0,0)=\Lambda(0,T)=-1/4$\\ 
$\Lambda(\infty,0)=-453/3320$,\quad 
$\Lambda(\infty,T)=
\Lambda(\infty,\infty)=-21/44$
\\
\hline
$\Lambda(m,0)=-\frac{3}{20}-\frac{83}{150}m+\frac{1}{300}\sqrt{225+1350m+27556m^2}$\\[1mm]
$\Lambda(m,\infty)=-\frac{3}{20}-\frac{83}{150}m+\frac{1}{60}\sqrt{1-18m+121m^2}$\\[1mm]
\qquad\qquad\qquad
$+\frac{1}{60}\sqrt{529-828m+484m^2}+\frac{1}{300}\sqrt{15625+m^2}$\\[1mm]
$\Lambda(m,\infty)=0$ for $m=m^*\approx 1.764$.\\\hline
\end{tabular}
\end{center}
\end{table}

Since 
$\Lambda(\infty,T)<\Lambda(\infty,0)$, 
for $m$ large enough, the condition 
$\Lambda(m,T)>\Lambda(m,0)$ cannot be satisfied. 
The graph of the function  
$\Lambda(m,T)=\frac{1}{T}\ln(\mu(m,T))$, is shown in Figure 4. The graphs of the  functions $T\mapsto \Lambda(m,T)$, for $m$ fixed and  also the graphs of the functions $m\mapsto \Lambda(m,T)$, for $T$ fixed are depicted in Figure \ref{fig8}.
We can make the following comments on this figure : The strictly decreasing functions $m\mapsto \Lambda(m,0)$ and 
$m\mapsto \Lambda(m,\infty)$, are depicted in dotted line on panel (a,c) of the figure. 
Panels (b,d) of the figure show that 
for $m>0$, the functions $T\mapsto\Lambda(m,T)$ can be increasing and then decreasing, so that Conjecture 1 is not true in the time(dependent migration case.  Panels (a,c) of the Figure show the graphs of functions $m\mapsto\Lambda(m,T)$ and illustrates their convergence 
toward $\Lambda(m,0)$ and $\Lambda(m,\infty)$ as $T$ tends to 0 and $\infty$, respectively.
Notice that for $0<T<\infty$, the functions $m\mapsto\Lambda(m,T)$ are not monotonic.
Also note that the functions $m\mapsto \Lambda(m,T)$ can take values greater than $\Lambda(m,\infty)$, see Figure \ref{fig8}(a,c).

\subsubsection{The three patch case}

We show in Figure \ref{fig10} the plot of $\Lambda(m,T)=\frac{1}{T}\ln(\mu(m,T))$ where $\mu(m,T)$ is the Perron root of the monodromy matrix $\Phi(m,T)=e^{\frac{T}{2}B}e^{\frac{T}{2}A}$,
in the case where the matrices 
$A$ and $B$ 
are given  by
\begin{equation}\label{ABESL}
\begin{array}{c}
A=\left[
\begin{array}{ccc}
9-10 m&0&0.1 m\\
10 m&-1-0.1 m& 0\\
0&0.1 m&-10-0.1 m
\end{array}
\right],
\\[5mm]
B=\left[
\begin{array}{ccc}
-10-0.1 m&0&10 m\\
0.1 m&-10 m&0\\
0&10 m&9-10 m
\end{array}
\right],
\end{array}
\end{equation}

\begin{figure}[ht]
\begin{center}
\includegraphics[width=10cm,
viewport=160 335 440 670]{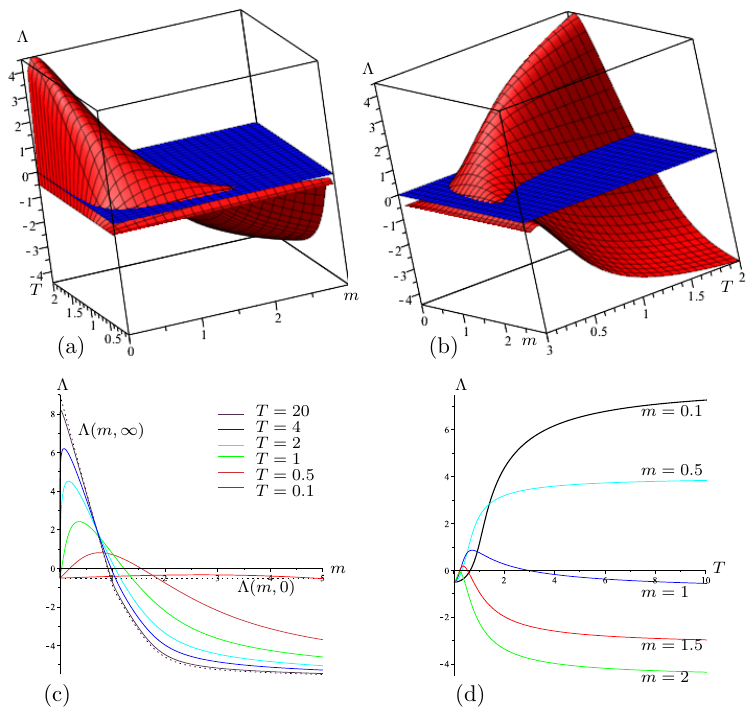}
\caption{(a,b) The graph of $(m,T)\mapsto \Lambda(m,T)$ corresponding to the matrices \eqref{ABESL}, seen from left and right.  
(c) Graphs of $m\mapsto \Lambda(m,T)$ with the indicated values of $T$. (d) Graphs of $T\mapsto \Lambda(m,T)$ with the indicated values of $m$. We have $\Lambda(m,\infty)=0$ for $m=m^*=0.172$.}	\label{fig10}	
\end{center}
\end{figure}

Using Remark 2, we 
have 
$q_1=q_2=q_3=1/3$ and
 $$
 p_{1}(\tau)=\left\{\begin{array}{r}
\frac{1}{201}\mbox{ if }0\leq \tau< \frac{1}{2}\\[1mm]
\frac{50}{51}\mbox{ if }\frac{1}{2}\leq \tau< 1
\end{array}
\right.,
\quad
p_2(\tau)=p_3(\tau)=\left\{\begin{array}{l}
\frac{100}{201}\mbox{ if }0\leq \tau<\frac{1}{2}\\
[1mm]
\frac{1}{102}\mbox{ if }\frac{1}{2}\leq \tau< 1
\end{array}
\right..
$$
Using the theoretical formulas in Theorem 4, we have
$$\begin{array}{l}
\Lambda(0,\infty)=9,\quad
\Lambda(0,0)=\Lambda(0,T)=
\Lambda(\infty,0)=\Lambda(m,0)=-1/2,\\ 
\Lambda(\infty,T)=
\Lambda(\infty,\infty)=-34497/4556\approx -7.572.
\end{array}
$$

We have $\Lambda(\infty,T)<\Lambda(\infty,0)$ and hence, for $m$ large enough we should have $\Lambda(m,T)<\Lambda(m,0)$, so that  $\Lambda(m,T)$ is not increasing with respect to $T$. 

\section{Slow regime}

\begin{figure}[ht]
\begin{center}
\includegraphics[width=10cm,
viewport=160 550 440 680]{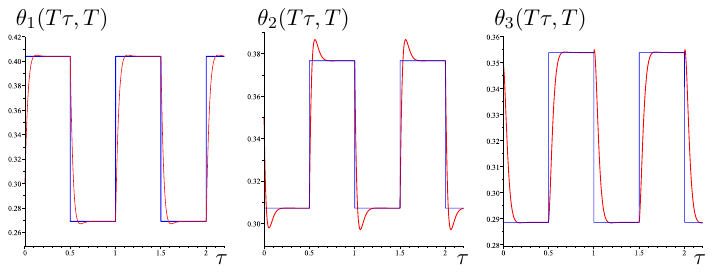}
\caption{The figure corresponds to the example discussed in Section \ref{Example3sites}, with $m=1$ and $T=20$. The solution $\theta(T\tau,T)$ with initial condition $\theta_1(0)=0.3$, $\theta_2(0)=0.35$, $\theta_3(0)=0.35$ is colored in red. The Perron-Frobenius vector $v(\tau)$ of the matrix $A(\tau)$ is colored in blue.}		\label{fig12}	
\end{center}
\end{figure}

The behavior of $\theta^*(T\tau,T)$ as $T\to\infty$ is illustrated in Figure \ref{fig12},  showing the approximation of $\theta^*(T\tau,T)$ by the Perron-Frobenius vector $v(\tau)$ when $T$ is large enough. Note that the approximation is uniform except on the small intervals $[\tau_k,\tau_{k}+\nu]$, where $\tau_k$ is a discontinuity of $v(\tau$. In these thin layers, the solution jumps quickly from the left limit of $v(\tau)$ at $\tau_k$ to its right limit. 

\begin{thebibliography}{1}


\bibitem{arino}
J. Arino and S. Portet. 
Epidemiological implications of mobility between a large urban centre and smaller satellite cities. 
\emph{J. Math. Biol.} 71, 1243--1265 (2015). 
\url{https://doi.org/10.1007/s00285-014-0854-z}

\bibitem{Baguette}
M. Baguette, T.G. Benton and J.M. Bullock.  Dispersal ecology and evolution. Oxford University Press, Oxford,  2012.

\bibitem{Banasiak}
J. Banasiak and M. Lachowicz.  \emph{Methods of Small Parameter in Mathematical Biology}. Modeling and Simulation in Science, Engineering and Technology. Birkh\"auser, Cham,  2014. \url{https://doi.org/10.1007/978-3-319-05140-6}
 
\bibitem{benaim}
	M. Bena\"im, C. Lobry, T. Sari and E. Strickler. Untangling the role of temporal and spatial variations in persistence of populations.
	{\it Theoretical Population Biology} {\bf 154} (2023) 1-26 \url{https://doi.org/10.1016/j.tpb.2023.07.003}.

	
\bibitem{BLSSarXiv}
	M. Bena\"im, C. Lobry, T. Sari and E. Strickler. A note on the top Lyapunov exponent of linear cooperative systems (2023). 
	\href{https://arxiv.org/abs/2302.05874}{
	arXiv:2302.05874}.
 
 \bibitem{bogolioubov}
    N.N. Bogoliubov and Yu.A. Mitropolskii.
    \emph{Asymptotic methods in the theory of nonlinear oscillations}, Gordon and Breach, New York, 1961.
 
 \bibitem{Brillinger}
 D.R. Brillinger.
 The Analyticity of the Roots of a Polynomial as Functions of the Coefficients,
{\it Mathematics Magazine} {\bf 39}, 3 (1966),  145--147.
\url{https://doi.org/10.1080/0025570X.1966.11975702}

\bibitem{Carmona}
P. Carmona. Asymptotic of the largest Floquet multiplier for cooperative matrices.
\emph{Annales de la Facult\'e 
des Sciences de Toulouse: Math\'ematiques}, Ser. 6,
31 (2022) 1213--1221. 
\url{https://doi.org/10.5802/afst.1716/}
 
 \bibitem{chen}
 S. Chen, J. Shi, Z. Shuai and Y. Wu. 
 Two novel proofs of spectral monotonicity of perturbed essentially nonnegative matrices with applications in population dynamics. 
 \emph{SIAM J Appl Math}, 82 (2022) 654--676.
\url{https://doi.org/10.1137/20M1345220}
	    
\bibitem{cosner}  C. Cosner,  J.C. Beier,  R.S. Cantrell,   D. Impoinvil,  L. Kapitanski,  M.D. Potts,  A. Troyo and S. Ruan. The effects of human movement on the persistence of vector-borne diseases, {\textit{Theor. Biol.}}, {{258}} (2009) 550-560.
\url{https://doi.org/10.1016/j.jtbi.2009.02.016}

\bibitem{elbetch2021}  B. Elbetch, T. Benzekri, D. Massart and T. Sari. The multi-patch logistic equation, {\textit{Discrete and Continuous Dynamical Systems - B}}, 26 (2021) 6405-6424. \url{https://dx.doi.org/10.3934/dcdsb.2021025}

\bibitem{elbetch2022}  B. Elbetch, T. Benzekri, D. Massart and T. Sari. The multi-patch logistic equation with asymmetric migration. \emph{Revista Integraci\'on, temas de matem\'aticas}, 40 (2022) 25-57. 
\url{https://doi.org/10.18273/revint.v40n1-2022002}

\bibitem{EVA13}
S.N. Evans, P.L. Ralph, S.J. Schreiber and A. Sen,
\newblock{Stochastic population growth in spatially heterogeneous environments.}
\newblock\emph{Journal of Mathematical Biology}, 66 (2013) 423--476.
\url{ttps://doi.org/10.1007/s00285-012-0514-0}

\bibitem{Fainshil}
L. Fainshil, M. Margaliot and P. Chigansky, On the Stability of Positive Linear Switched Systems Under Arbitrary Switching Laws, in 
\emph{IEEE Transactions on Automatic Control}, 54 (2009) 897-899. 
\url{https://doi.org/10.1109/TAC.2008.2010974}

\bibitem{Fenichel}
N. Fenichel. Geometric singular perturbation theory for ordinary differential equations. 
\emph{J. Differential Equat.}, 31 (1979) 53--98.
\url{https://doi.org/10.1016/0022-0396(79)90152-9}

\bibitem{freidlin}
M.I. Freidlin and A.D. Wentzell, \emph{Random Perturbations of
Dynamical Systems}. Springer, New York, 1998.

\bibitem{gao}
\newblock  D. Gao and  C.P. Dong. 
\newblock Fast diffusion inhibits disease outbreaks, 
\newblock \textit{Proc.  Am. Math. Soc.},
{148} (2020) 1709-1722. 
\url{https://doi.org/10.1090/proc/14868}

\bibitem{HOLT02}
A. Gonzalez and  R.D. Holt,
\newblock {The inflationary effects of environmental fluctuations in source-sink systems. }
\newblock \emph{Proc. Natl. Acad. Sci.} 99 (2002) 14872--14877.
\url{https://doi.org/10.1073/pnas.232589299} 

 \bibitem{guo}  H. Guo,  M.Y. Li and   Z. Shuai. Global stability of the endemic equilibrium of multigroup SIR epidemic models, {\textit{Can. Appl. Math. Q.}}, {{14}} (2006) 259-284.

\bibitem{haag}
J. Haag, Sur certains syst\`emes diff\'erentiels d\'efinis par des fonctions p\'eriodiques et discontinues. 
\emph{Bulletin des Sciences Math\'ematiques},  70 (1946) page 305.

\bibitem{Hanski}
I. Hanski, Metapopulation ecology. Oxford University Press. (1999)

\bibitem{Hartman}
\newblock{P. Hartman},
\newblock\emph{Ordinary Differential Equations},
\newblock{Society for Industrial and Applied Mathematics},
2002.
\url{https://epubs.siam.org/doi/abs/10.1137/1.9780898719222}

\bibitem{hirsch}
\newblock M.W. Hirsch
Systems of Differential Equations that are Competitive or Cooperative II: Convergence Almost Everywhere
\newblock\emph{
SIAM Journal on Mathematical Analysis} 16 (1985) 423-439.
\url{https://doi.org/10.1137/0516030}

\bibitem{Holt}
R.D. Holt. On the evolutionary stability of sink populations. \emph{Evolutionary Ecology}
11 (1997) 723--731.
\url{https://doi.org/10.1023/A:1018438403047}

\bibitem{Holt-et-al}
R.D. Holt, M. Barfield and A. Gonzalez. Impacts of environmental variability in open populations and communities: ``inflation'' in sink environments. \emph{Theoretical population biology} 64 (2003) 315--330.
\url{https://doi.org/10.1016/S0040-5809(03)00087-X}

\bibitem{HJ}
R.A. Horn and C.R. Johnson, \emph{Matrix analysis}. Cambridge university press, 2012.

\bibitem{Huston}
V. Hutson, W. Shen, G.T. Vickers, Estimates for the principal spectrum point for certain time-dependent parabolic operators, 
\emph{Proc. Amer. Math. Soc.} 129 (2000) 1669--1679.
\url{https://doi.org/10.1090/S0002-9939-00-05808-1}


\bibitem{Jansen}
V.A. Jansen and J. Yoshimura. Populations can persist in an environment consisting of
sink habitats only. 
\emph{Proc. Natl. Acad. Sci}, 95 (1998) 3696--3698.
\url{https://doi.org/10.1073/pnas.95.7.3696}

\bibitem{Katriel}
\newblock G. Katriel. Dispersal-induced growth in a time-periodic environment. 
\newblock\emph{J. Math. Biol.} 85 (2022) 24. 
\url{https://doi.org/10.1007/s00285-022-01791-7}	

\bibitem{Khalil}
H.K. Khalil, \emph{Nonlinear systems},
Prentice Hall, 2002.

\bibitem{KLA08}
C.A. Klausmeier,  
\newblock{Floquet theory: a useful tool for understanding nonequilibrium dynamics}.
\newblock \emph{Theoretical Ecology}, 1 (2008) 153--161.
\url{https://doi.org/10.1007/s12080-008-0016-2}
 
\bibitem{HOLTPNAS20}
N. Kortessis, M.W. Simon, M. Barfield, G. Glass, B.H. Singer, R.D. Holt.
\newblock{The interplay of movement and spatiotemporal variation in transmission degrades pandemic control.}
\newblock {\it Proc. Natl. Acad. Sci.} 117 (2020) 30104--30106.
\url{https://doi.org/10.1073/pnas.2018286117}



\bibitem{Kuehn}
C. Kuehn, \emph{Multiple Time Scale Dynamics},
Applided Mathematical Sciences 191, Springer, 2015.
    
\bibitem{liu}
	S. Liu, Y. Lou, P. Song. A new monotonicity for principal eigenvalues with applications to time-periodic patch models. \emph{SIAM J. Appl. Math} 82 (2022) 576--601. 
\url{https://doi.org/10.1137/20M1320973}


\bibitem{LST}C. Lobry,  T. Sari and  S. Touhami. On Tykhonov's theorem for convergence of
solutions of slow and fast systems, { \textit{Electron. J. Differ. Equ.}}, {{19}} (1998) 1-22. 

\bibitem{MeyerBook}
 C.D. Meyer. \emph{Matrix analysis and applied linear algebra},  Siam, 2000.
 

\bibitem{Meyer}
 C.D. Meyer. Continuity of the Perron root, \textit{Linear and Multilinear Algebra}, 63  (2015) 1332--1336. 
 \url{https://doi.org/10.1080/03081087.2014.934233} 
 
\bibitem{Mierczynski}
J. Mierczy\'{n}ski. Estimates for principal Lyapunov exponents: A
survey. \textit{Nonautonomous Dynamical Systems}, 1(2014) 137-162.
\url{http://eudml.org/doc/269442}


\bibitem{mitropolsky}
Y.A. Mitropolsky and N.V. Dao, \emph{Applied Asymptotic Methods in Nonlinear Oscillations}.
Springer, 1997. 

\bibitem{Noethen}
L. Noethen and S. Walcher. Tikhonov's theorem and quasi-steady state. 
\emph{Discr. Cont. Dyn. Syst. B},
16 (2011) 945--961.
\url{https://doi.org/10.3934/dcdsb.2011.16.945}

\bibitem{O'Malley}
R.E. O'Malley. 
\emph{Singular Perturbation Methods for Ordinary Differential Equations}. Springer, 1991.

\bibitem{Pulliam}
H.R. Pulliam. Sources, sinks, and population regulation. 
\emph{Am Nat} 132 (1988) 652--661. 
\url{https://doi.org/10.1086/284880}

\bibitem{RoseauVib}
M. Roseau. \emph{Vibrations non lin\'eaires et th\'eorie de la stabilit\'e}. Springer-Verlag, 1966.

\bibitem{RoseauED}
M. Roseau. \emph{Equations diff\'erentielles}. Masson, 1976.

\bibitem{ROY05}
M. Roy, R.D. Holt and M. Barfield.
\newblock{Temporal autocorrelation can enhance the persistence and abundance of metapopulations comprised of coupled sinks.}
\newblock{\it The American Naturalist} 166 (2005) 246--261.
\url{https://doi.org/10.1086/431286}

\bibitem{SVM}
J.A. Sanders, F. Verhulst, and J. Murdock, 
\emph{Averaging Methods in Nonlinear Dynamical Systems}, Springer, New York, 2007.


\bibitem{Schneider}
K.R. Schneider and T. Wilhelm.
Model reduction by extended quasi-steady-state
approximation, \emph{J. Math. Biol.} 40 (2000) 443--450.
\url{https://doi.org/10.1007/s002850000026}


\bibitem{SCH10}
S.J. Schreiber, 
\newblock{Interactive effects of temporal correlations, spatial heterogeneity and dispersal on population persistence.}
\newblock{it Proceedings of the Royal Society B: Biological Sciences}, 277  (2010) 1907-1914.
\url{https://doi.org/10.1098/rspb.2009.2006}

\bibitem{Slom}
\newblock
W. Slomczynski. 
\newblock
Irreducible cooperative systems are strongly monotone
\newblock \emph{Universitatis Iagellonicae Acta Mathematica}
30 (1993) 159-163.




\bibitem{tikhonov}  A.N. Tikhonov. Systems of differential equations containing small parameters in the derivatives, {\textit{Mat. Sb. (N.S.)}} {31} (1952) 575--586. 

\bibitem{wasow}  W.R. Wasow. \textit{Asymptotic Expansions for Ordinary Differential Equations}, {Robert E. Krieger Publishing Company, Huntington, NY}, 1976.

\end{thebibliography}
\end{document}